\documentclass[11pt]{article}
\usepackage{amsmath, latexsym, amsfonts, amssymb, amsthm, amscd}
\usepackage{mathtools} % Mathematical tools to use with amsmath
\usepackage[left=2cm, right=2cm, top=2.5cm, bottom=2.5cm]{geometry}
\usepackage{upquote} % Show "realistic" quotes in verbatim
\usepackage{color}
\usepackage{setspace} % space between lines
\usepackage[labelfont=bf]{caption} % Customising captions in floating environments
\usepackage{stmaryrd} % St Mary Road symbols for theoretical computer science
\usepackage{multicol} % multi-columns
\usepackage[dvipsnames]{xcolor}
\usepackage{esvect} % For special small arrow above a letter
\usepackage[title]{appendix}

\usepackage[utf8]{inputenc}
\usepackage[T1]{fontenc}

\usepackage{dsfont} % A type of font for indicators

\usepackage[english]{babel} % Language

\usepackage{bm} % Access bold symbols in maths mode

\usepackage{enumitem} % To refer to enumerate conditions
\newcounter{Cond1}
\newcounter{Cond2}
\newcounter{Cond3}

\newcounter{Cond5}

\usepackage{xcolor}
\definecolor{Maroon}{HTML}{ad2231}
\definecolor{webgreen}{HTML}{008000}
\usepackage{dsfont}
\usepackage{hyperref}
\hypersetup{colorlinks, breaklinks, urlcolor=Maroon, linkcolor=Maroon, citecolor=webgreen} % Set link colors
\allowdisplaybreaks



\newtheorem{theorem}{Theorem}[section]
\newtheorem{corollary}[theorem]{Corollary}
\newtheorem{proposition}[theorem]{Proposition}
\newtheorem{lemma}[theorem]{Lemma}
\newtheorem{remark}[theorem]{Remark}
\newtheorem{definition}[theorem]{Definition}

\theoremstyle{definition}
\newtheorem{example}[theorem]{Example}
\AtBeginEnvironment{example}{
\pushQED{\qed}%
}
\AtEndEnvironment{example}{\popQED}

\usepackage{mdframed}
\usepackage{lipsum}
\newmdtheoremenv[linewidth=1pt, innertopmargin=1ex]{Condition}{Assumption}[section]
\usepackage{todonotes}
\usepackage{chngcntr}
\usepackage{apptools}
\AtAppendix{\counterwithin{lemma}{section}}
\newcommand{\indep}{\perp\nolinebreak\!\!\!\perp}

\numberwithin{equation}{section}

\begin{document}
\title{Central limit theorem for supercritical Crump-Mode-Jagers processes counted with non-individual random characteristics}
\author{Gabriel Berzunza Ojeda\footnote{E-mail: \href{mailto:Gabriel.Berzunza-Ojeda@liverpool.ac.uk}{Gabriel.Berzunza-Ojeda@liverpool.ac.uk}} \,\, and \,\, Harlan Connor\footnote{E-mail: \href{h.c.g.connor@liverpool.ac.uk}{h.c.g.connor@liverpool.ac.uk}}  \\ \vspace*{10mm}
{\small Department of Mathematical Sciences, University of Liverpool, United Kingdom} } 
\maketitle

\vspace{0.1in}

\begin{abstract} 
Consider a supercritical Crump-Mode-Jagers process $(\mathcal{Z}_{t}^{\varphi})_{t \geq 0}$ counted with a random characteristic $\varphi$ that depends on an individual's life and their descendant process up to a fixed generation. Under second moment assumptions, we establish a central limit theorem for $\mathcal{Z}_{t}^{\varphi}$ as $t \rightarrow \infty$. Our result extends the recent work of Iksanov, Kolesko, and Meiners \cite{Iksanov2021} by relaxing their assumption of independent characteristics across individuals. We further demonstrate the applicability of our results to the study of fringe trees in several important random tree families, thereby providing insights into questions raised by Holmgren and Janson \cite{Holmgren2017}.
\end{abstract}

\noindent {\sc Key words and phrases}: Asymptotic fluctuations, central limit theorem, fringe trees, general Crump-Mode-Jagers branching process, Nerman's martingale.

\noindent {\sc MSC 2020 Subject Classifications}: 60J80; 60F05; 60J85.

%\tableofcontents

%%%%%%%%%%%%%%%%%%%%%%%%%%%%%%%%%%%%%%%%%%%%%%%%%%%%%%%%%%%%%%%%%%%%%%%%%%%
\section{Introduction}
%%%%%%%%%%%%%%%%%%%%%%%%%%%%%%%%%%%%%%%%%%%%%%%%%%%%%%%%%%%%%%%%%%%%%%%%%%%

The general Crump-Mode-Jagers branching process (CMJ-branching process) begins at time $0$ with a single ancestor. This ancestor lives for a random lifespan and produces offspring at times determined by a reproduction point process on $[0,\infty)$. Each subsequent individual evolves independently according to the same dynamics. This framework encompasses classical models such as
Bienaym\'e–Galton–Watson processes, Yule processes, continuous-time Markov processes, Sevastyanov processes, and Bellman–Harris processes (see \cite{Jagers1975} for historical details).

Our focus is on {\em random characteristics} assigned to individuals---random
elements of a suitable function space. Furthermore, by assigning a random characteristic to each individual, we study the process that sums the contributions of all individuals at a given time, where each individual's contribution is weighted by their random characteristic.
(This process is known in the literature as ``the general branching process counted with a random characteristic''.). This process captures features such as age, lifespan, or other individual-level metrics. A formal description is provided in Section \ref{CMJclassic}. Our study establishes asymptotic normality of this process in greater generality than that found in Iksanov et al.\ \cite{Iksanov2021}, facilitating the analysis of quantities such as the total number of births, the number of living individuals, the number of individuals with living descendants, the number of living individuals below a certain age, or the number of ``families'' consisting of a fixed number of individuals.

Laws of large numbers for supercritical CMJ-branching processes were established by Nerman \cite{Nerman1981} for the single-type, non-lattice case, with the lattice case treated by Gatzouras \cite{Dimitris2000}. Central limit theorems had previously been obtained
for related models: Athreya \cite{Athreya1969, Athreya1969II} and Janson \cite{Janson2004I} for specific multitype Markov branching processes with a finite type space, while Asmussen and Hering \cite[Section VIII.3]{Asmussen1983} for more general multitype Markov branching processes.  However, these results either do not directly apply to the
general single-type case or their assumptions hold only in special situations. Additionally, Janson and Neininger \cite{Svante2008} established a
central limit theorem for Kolmogorov's conservative fragmentation model, which
can be interpreted as a specific CMJ-branching process (see Example \ref{Example2}).

Janson \cite{Svante2018} studied asymptotic fluctuations for single-type
supercritical CMJ-processes in the lattice case. More recently, Iksanov et al.\ \cite{Iksanov2021} (see also \cite{Iksanov2021II})  proved a general central limit theorem for
single-type supercritical CMJ-processes, under the key assumption that characteristics are independent across individuals. Limit theorems for general branching processes counted with {\em individual characteristics} were also studied by Jagers and Nerman \cite{Jagers1984II}, though their conditions are often more difficult to verify. 

In this paper, we consider {\em general characteristics}, which may depend not
only on an individual's own life but also on their descendant processes up to a
fixed generation. This extends the case of {\em individual characteristics} and
allows us to remove the independence assumption of \cite{Iksanov2021II}.
Our main result, Theorem \ref{Theo2}, applies to single-type CMJ-processes
counted with characteristics depending on both an individual's life and their descendancy.

One motivation comes from biology, where CMJ-processes with
characteristics are widely employed for population modelling. Metrics that depend on an
individual's descendancy can yield new insights (see \cite[Section 7]{Nerman1981}).
Another motivation arises from random fractals: Charmoy, Croydon, and Hambly
\cite{Charmoy2017} studied CMJ-processes to analyse fluctuations in the
eigenvalue counting function, though under more restrictive assumptions than we
employ here.

A primary application motivating this work is the study of {\em fringe trees}
in random trees. Holmgren and Janson \cite{Holmgren2017} observed that many
families of random trees can be constructed from the genealogy of CMJ-branching
processes, and that their fringe trees can be analysed using
non-individual characteristics. They established laws of large numbers for the
number of fringe trees of a given type but left asymptotic normality as an open
problem. In Section \ref{Applicattions}, we address some of the questions
raised in \cite[Section 14]{Holmgren2017}. \\

The paper is organized as follows. In Section \ref{CMJclassic}, we
formally introduce the CMJ-branching process, the notion of a characteristic,
and the general branching process counted with a random characteristic.
Section \ref{mainresults} presents the main results (\S \ref{MainSec}),
which rely on a number of natural assumptions on the distribution of the
reproduction point process (\S \ref{AssumpCMJ}) and on the characteristics
themselves (\S \ref{AssumpCMJCharac}).
The main result is stated as Theorem \ref{Theo2}, followed by
Corollary \ref{corollary5}, which provides a formulation useful for applications.
Before proceeding to the proofs, we illustrate in Section \ref{Applicattions}
the applications to fringe trees--the primary motivation for this work.
Section \ref{Preliminaries} collects auxiliary and preliminary results,
together with relevant facts about the process.
The proofs are divided into two sections:
Section \ref{SecCLTcentred} establishes Theorem \ref{Theo1} for centred processes,
laying the groundwork for the general case,
which is fully developed in Section \ref{ProofMainTheandCoro},
where the detailed proofs of Theorem \ref{Theo2}
and Corollary \ref{corollary5} are presented.

\paragraph{Notation.} We use the following notation throughout the paper. Let $\mathbb{Z}$ denote the set of all integers,  $\mathbb{N} \coloneqq \{1, 2, \dots, \}$, $\mathbb{N}_{0} \coloneqq \{0, 1, \dots, \}$, and $\mathbb{R}_{+} = [0, \infty)$  the set of all nonnegative real numbers. The symbol $\xrightarrow[]{\rm d}$ indicates convergence in distribution, and $\xrightarrow[]{\rm p}$ indicates convergence in probability. For random variables $\mathcal{X}$ and $\mathcal{Y}$, we write $\mathcal{X} \stackrel{a.s.}{=} \mathcal{Y}$ if $\mathcal{X} = \mathcal{Y}$ almost surely.

%%%%%%%%%%%%%%%%%%%%%%%%%%%%%%%%%%%%%%%%%%%%%%%%%%%%%%%%%%%%%%%%%%%%%%%%%%%
\section{The CMJ-branching process and characteristics}
\label{CMJclassic}
%%%%%%%%%%%%%%%%%%%%%%%%%%%%%%%%%%%%%%%%%%%%%%%%%%%%%%%%%%%%%%%%%%%%%%%%%%%

In this section, we introduce the general Crump-Mode-Jagers branching process (CMJ-branching process) following \cite{Jagers1975, JagersN1984, Jagers1989}. The process starts with a single individual, the ancestor, born at time $0$. The ancestor produces offspring born at the points of a birth point process $\xi$ on the Borel sets of $[0, \infty)$ and is almost surely finite on compact sets (thus, $\xi$ has almost surely countable support with no accumulation points); see for instance \cite[Chapter 12]{Kallenberg2002} for background. We can thus arbitrarily enumerate the elements of its support by a sequence of real numbers, which is finite if and only if $\xi([0, \infty))< \infty$ almost surely, and, therefore, write $\xi=\sum_{i=1}^{N} \delta_{X_{i}}$ where $N \coloneqq \xi([0, \infty))$. If $N=\infty$ then the sum does not include the index $N$--as per the usual
convention. It is convenient to choose the sequence in a way that
$i < j \Rightarrow X_i \le X_j$. A concrete sequence is given by the formula $X_i = \inf \{t \ge 0:\, \xi([0, t]) \ge i\}$, where $i$ ranges over $\mathbb{N}$. Here and throughout the paper, we set $\inf \emptyset \coloneqq \infty$. %If $1 \le N < \infty$ then $i$ ranges over $1,\dots,N$. If $N=\infty$ then $i$ ranges over $\mathbb N$. 
If $N=0$ then $\xi=0$ and the sequence is not defined at all. Note that $\xi(\{t\}) < \infty$ almost surely, for all $t \ge 0$, and that $\xi(\{0\})$ may be positive almost surely. %We refer to the $X_i$ as the points of $\xi$. 
%If all points are distinct then $\xi$ is called simple.
The interpretation of $\xi$ is that its points represent the births of all children of the ancestor: the $i$-th child is born at time $X_i$. (Usually, one also assumes that the ancestor has a random lifetime,  $\zeta$, which is a random variable with values in $[0, \infty]$ and may be dependent on $\xi$.)

As each child may reproduce, it is necessary to define the set of names for all potential individuals in the evolving population. Individuals will be indexed by $\mathcal{I}=\bigcup_{n \in \mathbb{N}_{0}} \mathbb{N}^{n}$ according to their genealogy. Here, $\mathbb{N}^{0}\coloneqq \{\varnothing\}$ is the singleton set containing only the empty tuple $\varnothing$. It is the name of the ancestor. We use the usual Ulam-Harris notation: we abbreviate a tuple $u=\left(u_{1}, \ldots, u_{n}\right) \in \mathbb{N}^{n}$ by $u_{1} \cdots u_{n}$ and refer to $n$ as the length or generation or height of $u$; we write $|u|=n$. In this context, any $u= u_{1} \cdots u_{n} \in \mathcal{I}$ is called (potential) individual. If $v=v_{1} \cdots v_{m} \in \mathcal{I}$, then $u v$ is short for $u_{1} \cdots u_{n} v_{1} \cdots v_{m}$. For $u \in \mathcal{I}$ and $i \in \mathbb{N}$, the individuals $u i$ will be called children of $u$. Conversely, $u$ will be called parent of $u i$. 
More generally, $w$ will be called descendant of $u$ (short: $u \preceq w$ ) iff $u v=w$ for some $v \in \mathcal{I}$. Conversely, $u$ will be called an ancestor/progenitor of $w$. We write $u \prec w$ if $u \preceq w$ and $u \neq w$. Often, we shall refer to $\mathbb{N}^{n}$ as the (potential) $n^{\text{th}}$ generation $\left(n \in \mathbb{N}_{0}\right)$. With these notations, we have $|u|=n$ iff $u \in \mathbb{N}^{n}$ iff $u$ is an $n^{\text {th }}$ generation (potential) individual. For $u \in \mathcal{I}$ and $A \subseteq \mathcal{I}$, define $u A \coloneqq \{u v: v \in A\}$. In particular, $u \mathcal{I}$ denotes the subtree of $\mathcal{I}$ rooted at $u$. \\

For each $u \in \mathcal{I}$, one assigns an independent copy $(\xi_{u}, \zeta_{u})$ of the pair $(\xi, \zeta)$ that determines the birth times of $u$'s offspring relative to its time of birth, and the duration of its life. By the {\L}omnicki-Ulam theorem (see e.g., \cite[Corollary 6.18]{Kallenberg2002}) such random variables exists. Quantities derived from $\left(\xi_{u}, \zeta_{u}\right)$ are indexed by $u$. For instance, $\xi_{u} = \sum_{i=1}^{N_{u}} \delta_{X_{u,i}}$, where $N_{u} = \xi_{u}([0,\infty))$ is the number of offspring of $u$. 
%and $X_{u, k}$ is the difference between the birth-time of the $k^{\text {th}}$ child of $u$ and $u$ itself, etc. 
The birth-times $S(u)$ for $u \in \mathcal{I}$ are defined recursively. We set $S(\varnothing)\coloneqq 0$ and, for $n \in \mathbb{N}_{0}$ and $u \in \mathbb{N}^{n}$ such that $S(u)<\infty$, $S(u i) \coloneqq S(u)+X_{u, i}$ for  $i \in \mathbb{N}$. Otherwise, if $S(u)=\infty$, the individual $u \in \mathbb{N}^{n}$ is never born and we let $S(u i) \coloneqq \infty$ for $i \in \mathbb{N}$. %The time of death of individual $u$ is $S(u)+\zeta_{u}$. An individual $u$ is alive at time $t \geq 0$ if is born, but not yet dead at time $t$, i.e., if $S(u) \leq t < S(u) + \zeta_{u}$. 
We will omit the individual lifetime, $\zeta_{u}$, as it is generally not relevant to our analysis. However, there may also be other random variables associated with the individuals.

We now construct the canonical space for the CMJ-branching process. For $u \in \mathcal{I}$, let $(\Omega_{u}, \mathcal{F}_{u}, P_{u})$ be a copy of a given probability space $(\Omega_{\varnothing}, \mathcal{F}_{\varnothing}, P_{\varnothing})$, the {\sl life space} of the ancestor. An element $\omega \in \Omega_{u}$ is a possible life career for individual $u$ and any property of interest of $u$. In particular, the birth point process $\xi$ and the life span $\zeta$ are measurable functions defined on $\left(\Omega_{\varnothing}, \mathcal{F}_{\varnothing}\right)$. From the life space, we construct the {\sl population space}:
\begin{align} \label{eq100}
(\Omega, \mathcal{F}, \mathbb{P}) \coloneqq \left(\bigtimes_{u \in \mathcal{I}} \Omega_{u}, \bigotimes_{u \in \mathcal{I}} \mathcal{F}_{u}, \bigotimes_{u \in \mathcal{I}} P_{u}\right).
\end{align}
\noindent For $u \in \mathcal{I}$, we write $\pi_{u}$ for the projection $\pi_{u}: \times_{v \in \mathcal{I}} \Omega_{v} \rightarrow \Omega_{u}$ and $\theta_{u}$ for the shift operator $\theta_{u}\left((\omega_{v})_{v \in \mathcal{I}}\right)=(\omega_{u v})_{v \in \mathcal{I}}$, for $(\omega_{v})_{v \in \mathcal{I}} \in \Omega$. To formally lift a function $\eta$ defined on the life space $\Omega_{u}$ to the population space, we define $\eta_{u} \coloneqq  \eta \circ \pi_{u}$. In particular, $\xi_{u}=\xi \circ \pi_{u}$ and $\zeta_{u}=\zeta \circ \pi_{u}$. In slight abuse of notation, if $\eta$ is defined on the life space of the ancestor, when working on the population space, we write $\eta$ instead of $\eta_{\varnothing} = \eta \circ \pi_{\varnothing}$. For instance, we sometimes write $\mathbb{P}(\eta \leq t)$ for $P_{u}(\eta \leq t)=\mathbb{P}\left(\eta_{\varnothing} \leq t\right)$. Occasionally, random variables independent of $\mathcal{F}$ appear. In such instances, we extend the probability space $(\Omega, \mathcal{F}, \mathbb{P})$ as necessary. \\

A {\sl random characteristic} $\phi$, see e.g.\ \cite{Jagers1975, JagersN1984, Jagers1989}, is a random process on $(\Omega_{\varnothing}, \mathcal{F}_{\varnothing}, P_{\varnothing})$ that, in this work, is assumed to take values in the Skorokhod space of c\`adl\`ag  functions (see e.g., \cite[Chapter VI, Section 1]{Jacod2003}). The characteristic $\phi$ may also be viewed as a stochastic process $\phi: \Omega_{\varnothing} \times \mathbb{R} \rightarrow \mathbb{R},(\omega, t) \mapsto \varphi(\omega, t)$ with c\`adl\`ag  paths. Define $\phi_{u}=\phi \circ \pi_{u}$. By product measurability, $\phi_{u}(t-S(u))$ is a random variable. Note that, for a given $u \in \mathcal{I}, \phi_{u}$ is independent of $S(u)$. However, $\phi_{u}$ and $S(v)$ can be dependent, when $u$ is an ancestor of $v$. The general branching process $\mathcal{Z}^{\phi}=(\mathcal{Z}_{t}^{\phi})_{t \in \mathbb{R}}$ counted with characteristic $\phi$ is defined by
\begin{align} \label{eq5}
\mathcal{Z}_{t}^{\phi} \coloneqq \sum_{u \in \mathcal{I}} \phi_{u}(t-S(u)), \quad \text{for} \quad t \in \mathbb{R}.
\end{align}
\noindent Here, we use the convention $\phi(-\infty) \coloneqq 0$ and so the above sum involves only terms associated with individuals that are eventually born. In the special case $\phi(t) \coloneqq \mathbf{1}_{[0, \infty)}(t)$, for $t \in \mathbb{R}$, 
\begin{align} \label{eq15}
\mathcal{Z}_{t} \coloneqq \mathcal{Z}_{t}^{\mathbf{1}_{[0, \infty)}} = \sum_{u \in \mathcal{I}} \mathbf{1}_{[0, \infty)}(t-S(u)), \quad \text{for} \quad t \in \mathbb{R},
\end{align}
\noindent is the total number of births up to and including time $t$. Another example is given in \eqref{eq181}-\eqref{eq16}.

We have introduced the classical setting, that is, the general branching processes counted with random characteristics, where the characteristic associated with each individual is independent of the life histories of all other individuals. In particular, this corresponds to the setting considered in \cite{Iksanov2021}. In this work, we are interested in general branching processes counted with random characteristics that depend not only on an individual's own life but also on their whole descendant processes. 

To be precise, we consider a random characteristic $\varphi$ on $(\Omega, \mathcal{F}, \mathbb{P})$ taking values in the Skorokhod space of c\`adl\`ag  functions. Again, $\varphi$ may also be viewed as a stochastic process $\varphi: \Omega \times \mathbb{R} \rightarrow \mathbb{R}, (\omega, t) \mapsto \varphi(\omega, t)$ with c\`adl\`ag  paths. Define $\varphi_{u}=\varphi \circ \theta_{u}$, where $\theta_{u}$ is the shift operator. Note that $\varphi_{u}(t-S(u))$ is a random variable and that, for given $u \in \mathcal{I}, \varphi_{u}$ is independent of $S(u)$. However, $\varphi_{u}$ and $S(v)$ can be dependent, when $u$ is an ancestor of $v$.  Indeed, if $u, v \in \mathcal{I}$ are such that $u \not \in v\mathcal{I}$ and $v \not \in u\mathcal{I}$, then $\varphi_{u}$ and $\varphi_{v}$ are independent, otherwise can be dependent. The general branching process $\mathcal{Z}^{\varphi}=(\mathcal{Z}_{t}^{\varphi})_{t \in \mathbb{R}}$ counted with characteristic $\varphi$ is defined as in \eqref{eq5} with $\phi$ replaced by $\varphi$.  Again, we use the convention $\varphi(-\infty) \coloneqq 0$. 

We sometimes make the distinction between characteristics in the classical setting (i.e., characteristics independent of the life histories of different individuals) and characteristics that may depend not only on an individual's own life but also on its whole descendant processes. Referring to the former as {\sl individual characteristics} and the second as {\sl descendant characteristics}. The set-up of descendent characteristics have a wide range of possible applications; see e.g.\ \cite[Section 7]{Nerman1981}, \cite[Section 2]{JagersN1984}, \cite{Holmgren2017} or Section \ref{Applicattions}. 

Note that we consider individual or descendant characteristics that can be real-valued and need not vanish on the negative half-line. Thus, if $\varphi$ is a characteristic (individual or descendant) the series $\mathcal{Z}^{\varphi}$ is not necessarily well-defined. Conditions, ensuring the finiteness of $\mathcal{Z}^{\varphi}$ are provided in, for instance, \cite[Theorem 6.1]{JagersN1984} (and \cite[Proposition 2.2]{Iksanov2021} for individual characteristics only). In fact, the existence of a Malthusian parameter (assumptions \ref{A1}, \ref{A3} and \ref{A4} below and assumed throughout this paper) ensures that $\mathcal{Z}_{t}$ is finite almost surely; see Remark \ref{Remark3}. In particular, if the characteristic $\varphi$ vanishes on the negative half-line (i.e., $\varphi(t)= 0$, for $t<0$), then $\mathcal{Z}_{t}^{\varphi}$ is almost surely finite, since it has only finitely many non-zero summands almost surely. In this work, we will provide conditions on the characteristic $\varphi$ under which the general branching process $\mathcal{Z}^{\varphi}$ is well-defined; see Proposition \ref{Proposition3}.

While the individual characteristic, say $\phi$, is formally defined on the ancestor's life space $(\Omega_{\varnothing}, \mathcal{F}_{\varnothing}, P_{\varnothing})$, descendant characteristics are defined on the population space $(\Omega, \mathcal{F}, \mathbb{P})$. The mapping $\phi \circ \pi_{\varnothing}$ yields a descendant characteristic corresponding to $\phi$. For simplicity, we will use the same letter to represent both the individual characteristic and its corresponding descendant characteristic, writing for instance $\phi$ instead of $\phi \circ \pi_{\varnothing}$. This simplification is justified by the fact that $\phi \circ \pi_{\varnothing} \circ  \theta_{u} = \phi \circ \pi_{u}$, for $u \in \mathcal{I}$. Therefore, when results or remarks are applicable to both individual and descendant characteristics, we will simply use the term ``characteristic''. Moreover, we only consider c\`adl\`ag  characteristics in this work, as mentioned before.

%%%%%%%%%%%%%%%%%%%%%%%%%%%%%%%%%%%%%%%%%%%%%%%%%%%%%%%%%%%%%%%%%%%%%%%%%%%
\section{The main results}
\label{mainresults}
%%%%%%%%%%%%%%%%%%%%%%%%%%%%%%%%%%%%%%%%%%%%%%%%%%%%%%%%%%%%%%%%%%%%%%%%%%%

%%%%%%%%%%%%%%%%%%%%%%%%%%%%%%%%%%%%%%%%%%%%%%%%%%%%%%%%%%%%%%%%%%%%%%%%%%%
\subsection{Assumptions on the birth point process} \label{AssumpCMJ}
%%%%%%%%%%%%%%%%%%%%%%%%%%%%%%%%%%%%%%%%%%%%%%%%%%%%%%%%%%%%%%%%%%%%%%%%%%%

We now introduce the key assumptions concerning the birth point process that drives the CMJ-branching process, and which will be used throughout this work. Let $\mu(\cdot) \coloneqq \mathbb{E}[\xi(\cdot)] = \mathbb{E}[\xi_{\varnothing}(\cdot)]$ be the intensity measure of the point process $\xi$ and denote by $\hat{\mu}$ its Laplace transform, i.e.,
\begin{align} \label{eq74} 
\hat{\mu}(z) \coloneqq \int_{[0, \infty)} e^{-zs} \mu({\rm d} s) = \mathbb{E}\left[ \sum_{i=1}^{N} e^{-z X_{i}} \right], \quad \text{for all} \quad  z \in \mathbb{C},
\end{align}
\noindent for which the above integral converges absolutely. 

\begin{Condition}
In this paper, we differentiate between the lattice and the non-lattice case. We say that $\xi$ is {\sl lattice} if $\mu([0, \infty) \setminus m \mathbb{N}_{0})=0$ for some $m >0$, and we say that $\xi$ is {\sl non-lattice}, otherwise. Moreover, without loss of generality, we assume that in the lattice case the span is $1$, that is, $\mu([0, \infty) \setminus \mathbb{N}_{0}) > 0$ and $\mu([0, \infty) \setminus m \mathbb{N}_{0})=0$ for all $m > 1$. We set $\mathbb{G} \coloneqq \mathbb{Z}$ in the lattice case and $\mathbb{G} \coloneqq \mathbb{R}$ in the non-lattice case. Let $\ell({\rm d} x)$ represent the counting measure on $\mathbb{Z}$, in the lattice case, and the Lebesgue measure on $\mathbb{R}$, in the non-lattice case.
\end{Condition}

We make the following assumptions:
\begin{enumerate}[label=(\textbf{A.\arabic*})]
\item $\mu(\{0\}) =  \mathbb{E}[\xi(\{0\})] < 1$. (This rules out a rather trivial case with explosions already at the start.)\label{A1}
\item $\mathbb{E}[N] = \mathbb{E}[\mu([0, \infty))] > 1$. (This is known as the {\sl supercritical regime}. In this case, the underlying branching process survives with positive probability.) \label{A3}
\item There exists a real number $\alpha$ (the {\sl Malthusian parameter}) such that $\hat{\mu}(\alpha) =1$. (Indeed, by \ref{A3}, $\alpha >0$.) \label{A4}
\item There exists $\vartheta \in (0, \alpha/2)$ such that $\mathbb{E} \left[  \left(\sum_{i=1}^{N} e^{-\vartheta X_{i}} \right)^{2} \right]  < \infty$.  \label{A7}
\setcounter{Cond1}{\value{enumi}}
\end{enumerate}

Observe that, for $t, \theta \in \mathbb{R}_{+}$,
\begin{align} \label{eq45}
\xi([0,t]) = \int_{[0,t]} \xi({\rm d} x) \leq  \int_{[0,t]} e^{\theta x} e^{-\theta x} \xi({\rm d} x) \leq e^{\theta t} \int_{[0,\infty)} e^{-\theta x} \xi({\rm d} x).
\end{align}

\begin{remark} \label{Remark3}
If \ref{A3} and \ref{A4} hold, then, by \eqref{eq45} (with $\theta = \alpha$), $\mu([0,t]) \leq e^{\alpha t} < \infty$, for all $t\in \mathbb{R}_{+}$. Moreover, if additionally \ref{A1} holds, then $\mathbb{E}[\mathcal{Z}_{t}] < \infty$, for all $t\in \mathbb{R}$; see Section \ref{ExpSection} below or \cite[Theorem 6.2.2]{Jagers1975}. In particular, since $\mathcal{Z}_{t} \geq 0$, we can conclude that $\mathcal{Z}_{t} < \infty$ almost surely.
\end{remark}

For $\theta \in \mathbb{R}_{+}$, we set
\begin{align} \label{eq125}
\hat{\xi}(\theta)  \coloneqq \int_{0}^{\infty} e^{-\theta x} \xi({\rm d} x) =  \sum_{i=1}^{N} e^{-\theta X_{i}}.
\end{align}
\noindent In particular, by \ref{A7}, $\mathbb{E}[(\hat{\xi}(\vartheta))^{2}] < \infty$. Observe also that \ref{A7} and Cauchy–Schwarz's inequality imply that 
\begin{align} \label{eq31}
\text{there exists} \quad \vartheta \in (0, \alpha/2) \quad \text{such that} \quad \hat{\mu}(\vartheta) = \mathbb{E} \left[\sum_{i=1}^{N} e^{-\vartheta X_{i}} \right]  < \infty.
\end{align}
\noindent In particular, \ref{A1} and \ref{A7} imply that 
\begin{align}  \label{eq3}
\beta \coloneqq \int_{[0,\infty)} x e^{-\alpha x} \mu({\rm d}x) \in (0,\infty).
\end{align} 

\noindent It is important to point-out that assumptions \ref{A1}-\ref{A7} imply \cite[assumptions (A1)-(A3)]{Iksanov2021}. Specifically, assumption \ref{A7}, which imposes a finite second moment condition on the point process $\xi$, is stronger than \cite[assumptions (A3)]{Iksanov2021}. We adopt this stronger condition for technical reasons. Furthermore, as highlighted in \cite[Remark 2.1]{Iksanov2021}, assumption \ref{A7} is often easier to verify when applicable and corresponds to \cite[assumption (A6)]{Svante2018}. Janson \cite{Svante2018} investigated the fluctuations of supercritical general branching processes counted with individual characteristics within the lattice setting.

Suppose that \ref{A1}-\ref{A4} hold. The existence of the Malthusian parameter $\alpha >0$ allows us to define a non-negative martingale, called {\sl Nerman's martingale} \cite{Nerman1981}. Consider the individual characteristic,
\begin{align} \label{eq181}
\phi(t) = \mathbf{1}_{[0, \infty)}(t) e^{\alpha t} \int_{(t, \infty)} e^{-\alpha x} \xi({\rm d} x) , \quad \text{for} \quad t \in \mathbb{R}.
\end{align}
\noindent Then
\begin{align} \label{eq16}
W_{t} \coloneqq e^{-\alpha t} \mathcal{Z}_{t}^{\phi} = e^{- \alpha t} \sum_{u \in \mathcal{I}} \phi_{u}(t-S(u)) = \sum_{u \in \mathcal{C}_{t}} e^{-\alpha S(u)}, \quad \text{for} \quad t \in \mathbb{R},
\end{align}
\noindent where 
\begin{align} 
\mathcal{C}_{t} \coloneqq \{v\in \mathcal{I} \setminus \{ \varnothing \} : S(pr(v)) < \infty  \, \, \text{and} \, \, S(pr(v)) \leq t < S(v) \}
\end{align}
\noindent is the coming generation at time $t \in \mathbb{R}$ and $pr(v)$ denotes the parent of $v \in \mathcal{I} \setminus \{ \varnothing \}$. Here, $\mathcal{C}_{t} = \emptyset$, for $t \in (-\infty, 0)$. The process $(W_{t})_{t \geq 0}$ is a martingale. In particular, there exists a random variable $W < \infty$ almost surely such that $W_{t}  \rightarrow  W$, almost surely, as $t \rightarrow \infty$. Moreover, if $\mathbb{E}[\hat{\xi}(\alpha) \log_{+} \hat{\xi}(\alpha)] < \infty$, then \cite[Corollary 3.3]{Nerman1981} implies that $W_{t}  \rightarrow  W$, in $L^{1}(\Omega, \mathcal{F}, \mathbb{P})$, as $t \rightarrow \infty$. Thus, $\mathbb{E}[W]=1$ and $W > 0$ almost surely (see e.g., \cite[Proposition 1.1]{Nerman1981}). Note that \ref{A7} implies $\mathbb{E}[\hat{\xi}(\alpha) \log_{+} \hat{\xi}(\alpha)] < \infty$. Consequently, in some of our results below, $\mathbb{E}[\hat{\xi}(\alpha) \log_{+} \hat{\xi}(\alpha)] < \infty$ will be implicitly satisfied when \ref{A7} holds.

Nerman's martingale is essential for understanding the law of large numbers of CMJ-branching processes; see e.g., \cite{Nerman1981, JagersN1984}. Iksanov et al.\ \cite{Iksanov2021} recently introduced a family of martingales, namely Nerman's martingales with complex parameters. These martingales hold significant implications for the central limit theorem related to CMJ-branching processes counted with individual characteristics, and this will also be the case when considering descendant characteristics, as will be shown later.

Denote by $\Lambda$ the set of roots to the equation $\hat{\mu}(\lambda)-1=0$ such that ${\rm Re}(\lambda)>\frac{\alpha}{2}$ and by $\partial \Lambda$ we denote the set of roots on the critical line ${\rm Re}(\lambda)=\frac{\alpha}{2}$. In the lattice case, $\hat{\mu}(\cdot)$ is $2 \pi \mathrm{i}$-periodic, and we denote by $\Lambda$ the set of roots to the equation $\hat{\mu}(\lambda)-1=0$ such that ${\rm Re}(\lambda)>\frac{\alpha}{2}$ and $\mathrm{Im}(\lambda) \in (-\pi, \pi]$. Similarly, in the lattice case, $\partial \Lambda$ denotes the set of roots on the critical line ${\rm Re}(\lambda)=\frac{\alpha}{2}$ such that $\mathrm{Im}(\lambda) \in (-\pi, \pi]$. In both cases, we set $\Lambda_{\geq} \coloneqq \Lambda \cup \partial \Lambda$. Note that $\alpha \in \Lambda$ (recall \ref{A4}) and that every other element $\lambda  \in \Lambda_{\geq}$, $\lambda \neq \alpha$ satisfies ${\rm Re}(\lambda) \in\left[\frac{\alpha}{2}, \alpha\right)$ and ${\rm Im}(\lambda) \neq 0$. Further, $\lambda = \theta + \mathrm{i} \eta \in \Lambda_{\geq}$ implies that the complex conjugate $\bar{\lambda} = \theta - \mathrm{i} \eta \in \Lambda_{\geq}$ except if $\eta = \pi$ in the lattice case. As in \cite{Iksanov2021}, 
we also consider the following assumption:
\begin{enumerate}[label=(\textbf{A.\arabic*})]
\setcounter{enumi}{\value{Cond1}}
\item The set of roots $\Lambda_{\geq}$ is finite. \label{A8}
\end{enumerate}

For each root $\lambda \in \Lambda_{\geq}$, we denote its multiplicity by $k(\lambda) \in \mathbb{N}$. Then, for any $i=0, \dots, k(\lambda)-1$, we can define
\begin{align}
W_{t}^{(i)}(\lambda) \coloneqq (-1)^{i} \sum_{u \in \mathcal{C}_{t}} S(u)^{i} e^{-\lambda S(u)}, \quad \text{for} \quad t \in \mathbb{R}.
\end{align}
\noindent The Malthusian parameter $\alpha>0$ in \ref{A4} is a root of multiplicity $1$ and gives rise to Nerman's martingale $(W_{t})_{t \geq 0}=(W_{t}^{(0)}(\alpha))_{t \geq 0}$. Moreover, if \ref{A1}-\ref{A7} hold, then \cite[Theorem 2.7]{Iksanov2021} (and \cite[Lemma 5.4]{Iksanov2021}) implies that, for any $\lambda \in \Lambda$ and $i= 0, \dots, k(\lambda)-1$, the process $(W_{t}^{(i)}(\lambda))_{t \geq 0}$ is a martingale and there is a random variable $W^{(i)}(\lambda) \in L^{2}(\Omega, \mathcal{F}, \mathbb{P})$ such that
\begin{align}
W_{t}^{(i)}(\lambda) \rightarrow W^{(i)}(\lambda) \quad \text { almost surely and in } L^{2}(\Omega, \mathcal{F}, \mathbb{P}), \text { as } t \rightarrow \infty.
\end{align}
\noindent Moreover, $\mathbb{E}[W^{(i)}(\lambda)] = \delta_{0i}$, for $i= 0, \dots, k(\lambda)-1$, where $\delta_{0i}$ denotes the Kronecker delta function.

%%%%%%%%%%%%%%%%%%%%%%%%%%%%%%%%%%%%%%%%%%%%%%%%%%%%%%%%%%%%%%%%%%%%%%%%%%%
\subsection{Assumptions on the random characteristics} \label{AssumpCMJCharac}
%%%%%%%%%%%%%%%%%%%%%%%%%%%%%%%%%%%%%%%%%%%%%%%%%%%%%%%%%%%%%%%%%%%%%%%%%%%

We continue with assumptions regarding the random characteristics. 

If $\varphi$ is a non-negative or integrable characteristic (the later meaning $\varphi(t) \in L^{1}(\Omega, \mathcal{F}, \mathbb{P})$ for every $t \in \mathbb{R}$), then we write $\mathbb{E}[\varphi]$ for the function that maps $t \in \mathbb{R} \mapsto \mathbb{E}[\varphi](t) \coloneqq \mathbb{E}[\varphi(t)]$. If $\varphi$ is an integrable characteristic, then we write ${\rm Var}[\varphi]$ for the function that maps $t \in \mathbb{R} \mapsto {\rm Var}[\varphi](t) \coloneqq {\rm Var}[\varphi(t)] = \mathbb{E}[(\varphi(t)-\mathbb{E}[\varphi(t)])^{2}]$. 

\begin{enumerate}[label=(\textbf{C.\arabic*})]
\item $\varphi(t) \in L^{1}(\Omega, \mathcal{F}, \mathbb{P})$ for every $t \in \mathbb{R}$ and $t \in \mathbb{R} \mapsto \mathbb{E}[\varphi](t) e^{-\alpha t}$ is directly Riemann integrable. \label{C1}

\item $\varphi(t) \in L^{2}(\Omega, \mathcal{F}, \mathbb{P})$ for every $t \in \mathbb{R}$ and $t \in \mathbb{R} \mapsto {\rm Var}[\varphi](t) e^{-\alpha t}$ is directly Riemann integrable. \label{C2}

\item For any $t \in \mathbb{R}$ there exists $\varepsilon >0$ such that
$(|\varphi(s)|^{2})_{|s-t| \leq \varepsilon}$ is uniformly integrable.  \label{C3}
\setcounter{Cond2}{\value{enumi}}
\end{enumerate}

Assumptions \ref{C1}-\ref{C3} corresponds to \cite[assumptions (A4)-(A6)]{Iksanov2021}. Note that \ref{C2} implies that $\varphi(t) \in L^{1}(\Omega, \mathcal{F}, \mathbb{P})$ for every $t \in \mathbb{R}$. If $\varphi$ is deterministic and real-valued, then it satisfies \ref{C2}. Moreover, \ref{C3} also holds because $\varphi$ is c\`adl\`ag and thus it is locally bounded. On the other hand, if a random characteristic $\varphi$ satisfies \ref{C3}, then, by \cite[Proposition 4.12]{Kallenberg2002}, both the expectation function $\mathbb{E}[\varphi]$ and the variance function ${\rm Var}[\varphi]$ are c\`adl\`ag. Thus, the functions $\mathbb{E}[\varphi]$ and ${\rm Var}[\varphi]$ are deterministic (individual) characteristics. 

In the lattice case, if $t \in \mathbb{Z}$, then $t-S(u) \in \mathbb{Z}$, for all individuals $u \in \mathcal{I}$ with $S(u) < \infty$. Hence, $\mathcal{Z}^{\varphi}_{t}$ depends only on $\varphi_{u}(s)$ for $s \in \mathbb{Z}$, making $\varphi$ irrelevant on $\mathbb{R} \setminus \mathbb{Z}$. Therefore, as in \cite{Iksanov2021}, we always assume the following in the lattice case.
\begin{Condition}
In the lattice case, we assume that $\varphi$ has paths that are constant on intervals of the form $[m, m+1)$, for $m \in \mathbb{Z}$. 
\end{Condition}

Under the preceding assumption, \ref{C3} becomes redundant in the lattice setting. It simplifies to $\varphi(s) \in L^{2}(\Omega, \mathcal{F}, \mathbb{P})$, for $s \in \mathbb{Z}$, which is already implied by condition \ref{C2}. This is because uniform integrability is automatically satisfied in this setting. \\

We continue by stating a proposition that gives sufficient conditions for the general branching process $\mathcal{Z}^{\varphi}$ to be well-defined. We first introduce the notion of $h$-dependent descendant characteristic.  Subsequently, and following \cite{Iksanov2021}, we also introduce the concept of an admissible ordering of $\mathcal{I}$. 

\begin{definition}
Define the sub-$\sigma$-algebras $(\mathcal{A}^{(n)})_{n \geq 0}$ of $\mathcal{F}$ by letting, for each $n \in \mathbb{N}_{0}$,
\begin{align} \label{eq83} 
\mathcal{A}^{(n)} \coloneqq \sigma( \pi_{v}: v \in \mathcal{I} \, \, \text{such that} \, \, |v| \leq  n).
\end{align}
\noindent We say that a characteristic $\varphi: \Omega \times \mathbb{R} \rightarrow \mathbb{R}$ is an {\sl $h$-dependent descendant characteristic}, for $h \in \mathbb{N}_{0}$, if the map $(\omega, t) \mapsto \varphi(\omega, t)$ is measurable with respect the product $\sigma$-algebra $\mathcal{B}(\mathbb{R})\times \mathcal{A}^{(h)}$, where $\mathcal{B}(\mathbb{R})$ is the Borel $\sigma$-algebra.
\end{definition}

An $h$-dependent descendant characteristic depends solely on an individual's life and their descendants up to generation $h$. An individual characteristic may be considered a $0$-dependent descendant characteristic. Let $\mathcal{A}^{(-1)}$ be the trivial $\sigma$-algebra. Note that a deterministic characteristic is always measurable with respect to the product $\sigma$-algebra $\mathcal{B}(\mathbb{R})\times \mathcal{A}^{(-1)}$. In particular, it is measurable with respect to $\mathcal{B}(\mathbb{R})\times \mathcal{A}^{(0)}$. Hereafter, deterministic characteristics will be considered $0$-dependent descendant characteristics, which means we treat them as individual characteristics.

Given a characteristic $\varphi$ with $\varphi(t) \in L^{1}(\Omega, \mathcal{F}, \mathbb{P})$ for all $t \in \mathbb{R}$, define for $k \in \mathbb{N}_{0}$,
\begin{align} \label{eq27}
\varphi^{(k)}(t) \coloneqq \mathbb{E}[\varphi(t) \mid \mathcal{A}^{(k-1)}], \quad \text{for} \quad t \in \mathbb{R}.  
\end{align}
\noindent Clearly, the function $\varphi^{(k)}$ is well-defined. Furthermore, if \ref{C3} holds, then $\varphi^{(k)}$ has almost surely c\`adl\`ag  paths. (This follows from the definition of conditional expectation and \cite[Remark 1, Section 6.6, p.\ 190]{Resnick2014}.) In particular, for $t \in \mathbb{R}$, $\varphi^{(0)}(t) \stackrel{a.s.}{=} \mathbb{E}[\varphi(t)] = \mathbb{E}[\varphi](t)$, that is, $\varphi^{(0)}$ is a (deterministic) $0$-dependent descendant characteristic. Moreover, if $\varphi$ is an $h$-dependent descendant characteristic, then, for $t \in \mathbb{R}$, $\varphi^{(k)}(t) \stackrel{a.s.}{=} \varphi(t)$,   for $k \geq h+1$. Indeed, for $k \in \mathbb{N}$, $\varphi^{(k)}$ is a $(k-1)$-dependent descendant characteristic. 

For fixed $h \in \mathbb{N}_{0}$, define the (random) function,
\begin{align} \label{eq53}
\chi^{(\varphi, h)}(t) \coloneqq  \sum_{0 \leq |u| \leq h} (\varphi_{u}^{(h+1-|u|)}(t-S(u)) - \varphi_{u}^{(h-|u|)}(t-S(u))), \quad \text{for} \quad  t \in \mathbb{R},
% \chi^{(\varphi, h)}(t) \coloneqq  \sum_{i =0}^{h}\sum_{|u| = i} (\varphi_{u}^{(h+1-i)}(t-S(u)) - \varphi_{u}^{(h-i)}(t-S(u))), \quad \text{for} \quad  t \in \mathbb{R},
\end{align}
\noindent where the sum $\sum_{0 \leq |u| \leq h}$ is over all vertices $u \in \mathcal{I}$ with height $|u| \leq h$. If $\varphi$ is an $h$-dependent descendant characteristic, Lemma \ref{lemma4} implies that, under our assumptions \ref{A1}-\ref{A7} and \ref{C1}-\ref{C2}, the series \eqref{eq53} converges absolutely almost surely (see also Remark \ref{Remark5}). In particular, Lemma \ref{lemma6} \ref{lemma6Pro1} below shows that $|\chi^{(\varphi, h)}(t)|< \infty$, for every $t \in \mathbb{R}$, almost surely. Moreover, if \ref{C4} below also hold, then $\chi^{(\varphi, h)}$ almost surely has c\`adl\`ag  paths (see Lemma \ref{lemma6} \ref{lemma6Pro1b}). The characteristic $\chi^{(\varphi, h)}$ is of significant relevance to our analysis and is central to our main result, Theorem \ref{Theo2}. \\ 

We call a sequence $u_{1}, u_{2}, \ldots \in \mathcal{I}$ an {\sl admissible ordering} of $\mathcal{I}$ if $\mathcal{I}_{n} \coloneqq \left\{u_{1}, \ldots, u_{n}\right\}$ is a subtree of the Ulam-Harris tree $\mathcal{I}$ of cardinality $n \in \mathbb{N}$ such that $u_{1} = \varnothing$, and $\mathcal{I}=\bigcup_{n \in \mathbb{N}} \mathcal{I}_{n}$. Admissible orderings exist (an example can found in \cite{Iksanov2021}). Recall that a series $\sum_{n \in \mathbb{N}} x_{n}$ in a Banach space $(X,\|\cdot\|)$ is said to {\sl converge unconditionally} if, for any $\varepsilon>0$, there is a finite $I \subseteq \mathbb{N}$ such that $\left\|\sum_{n \in J} x_{n}\right\|<\varepsilon$ for any $J \subseteq \mathbb{N} \backslash I$ (an equivalent definition is that the series converges for any rearrangement); see e.g., \cite{Hildebrandt1940}. 

\begin{proposition} \label{Proposition3}
Suppose that \ref{A1}-\ref{A7} hold and that $\varphi$ is an $h$-dependent descendant characteristic, for some $h \in \mathbb{N}_{0}$,  satisfying \ref{C1}-\ref{C3}. Then, for every $t \in \mathbb{R}$, 
\begin{align}
\mathcal{Z}^{\varphi}_{t} \coloneqq \sum_{u \in \mathcal{I}} \varphi_{u}(t- S(u))
\end{align}
\noindent converges unconditionally in $L^{1}(\Omega, \mathcal{F}, \mathbb{P})$ and almost surely over every admissible ordering of $\mathcal{I}$. 
\end{proposition}

Proposition \ref{Proposition3} extends \cite[Proposition 2.2]{Iksanov2021}, which applies only to individual characteristics. When $|\varphi|$ satisfies \ref{C1}, \cite[condition (6.1)]{JagersN1984} holds, and in particular, by \cite[Theorem 6.1]{JagersN1984}, $\mathbb{E}[|\mathcal{Z}_{t}^{\varphi}|] < \infty$, for all $t \in \mathbb{R}$; see also Section \ref{ExpSection} below. Consequently, $\mathcal{Z}_{t}^{\varphi}$ converges unconditionally in $L^{1}(\Omega, \mathcal{F}, \mathbb{P})$ and almost surely over every admissible ordering of $\mathcal{I}$. However, this work considers cases where $|\varphi|$ does not satisfy \ref{C1}, necessitating Proposition \ref{Proposition3}. The proof is given in Section \ref{ProofProposition}. 

\begin{remark} \label{Remark4}
Suppose that \ref{A1}-\ref{A7} hold and that $\varphi$ is a characteristic satisfying \ref{C1}-\ref{C3}. Clearly, the expectation function $\mathbb{E}[\varphi]$ (a $0$-dependent descendant characteristic) satisfies \ref{C1}-\ref{C3}. Proposition \ref{Proposition3} then implies that for every $t \in \mathbb{R}$, the series $\mathcal{Z}^{\mathbb{E}[\varphi]}_{t}$ converges unconditionally in $L^{1}(\Omega, \mathcal{F}, \mathbb{P})$ and almost surely over every admissible ordering of $\mathcal{I}$. 
\end{remark}

\begin{remark}{(\cite[Remark 2.5]{Iksanov2021})} \label{Remark1}
Let $\varphi$ and $\phi$ be random characteristics satisfying \ref{C1}. Then, any linear combination of $\varphi$ and $\phi$ also satisfies \ref{C1}. If they additionally satisfy \ref{C3}, then so does any linear combination. Furthermore, by \ref{C3} and \cite[Proposition 4.12]{Kallenberg2002}, the expectation and variance functions of any such linear combination are c\`adl\`ag, implying local boundedness and almost everywhere continuity. Thus, if $\varphi$ and $\phi$ satisfy both \ref{C2} and \ref{C3}, then any linear combination of $\varphi$ and $\phi$ also satisfies \ref{C2}. This follows from the inequality,
\begin{align} \label{eq108}
\mathrm{Var}(a\varphi(t) + b\phi(t)) \leq 2 a^{2} \mathrm{Var}(\varphi(t))  + 2b^{2} \mathrm{Var}(\phi(t)), \quad \text{for} \quad  a,b,t \in \mathbb{R},
\end{align}
\noindent and for example, (a slightly extended version of) \cite[Remark 3.10.5, p.\ 237]{Sidney1992} (or \cite[Proposition 4.1, Chapter V]{Soren2003}), which states that a locally Riemann integrable function dominated by a directly Riemann integrable function is itself directly Riemann integrable.
\end{remark}

Finally, we require a couple of technical assumptions.
\begin{enumerate}[label=(\textbf{C.\arabic*})]
\setcounter{enumi}{\value{Cond2}}
\item Let $\varphi$ be an $h$-dependent descendant characteristic, for some $h \in \mathbb{N}_{0}$. For any $t \in \mathbb{R}$ there exists $\varepsilon >0$ such that  \label{C4}
\begin{align} \label{eq107}
\mathbb{E} \left[ \sum_{i=0}^{h} \sum_{|u|=i} \sup_{|s-t|\leq \varepsilon}|\varphi_{u}(s-S(u))\mathbf{1}_{(-\infty,0)}(s-S(u))| \right] < \infty. 
\end{align}

\item Let $\varphi$ be an $h$-dependent descendant characteristic, for some $h \in \mathbb{N}_{0}$. For any $t \in \mathbb{R}$ there exists $\varepsilon >0$ such that \label{C5}
\begin{align}
\left( \sum_{i=0}^{h} \left(\sum_{|u|=i} |\varphi_{u}(s-S(u))| \right)^{2} \right)_{|s-t| \leq \varepsilon} \quad \text{is uniformly integrable}.
\end{align}
\end{enumerate}

These last assumptions \ref{C4}-\ref{C5} are essentially only needed to guarantee that $t \in \mathbb{R} \mapsto \chi^{(\varphi, h)}(t)$ and $t \in \mathbb{R} \mapsto \mathbb{E}[(\chi^{(\varphi, h)}(t))^{2}]$ are c\`adl\`ag; see Lemma \ref{lemma6} and its proof. 

Observe that if $\varphi$ is a $0$-dependent descendant characteristic (i.e., an individual characteristic), then clearly assumption \ref{C5} reduces to assumption \ref{C3}. For individual characteristics, assumption \ref{C4} is, in fact, not needed. This can be verified by a close inspection of our arguments, particularly in the proof of Lemma \ref{lemma6} \ref{lemma6Pro1b}-\ref{lemma6Pro4}, where it is used; see Remark \ref{Remark8}. Unfortunately, we were unable to prove that perhaps a combination of our previous assumptions implies \ref{C4} or \ref{C5}. Thus, we leave this as an open problem.

On the other hand, assumption \ref{C5} is clearly satisfied for the class of characteristics that vanish on the negative half-line (i.e., $\varphi(t)= 0$, for $t<0$). This class of characteristics has been of particular interest in previous works on CMJ-branching processes (see e.g.\ \cite{Nerman1981} or \cite{JagersN1984}) and is perhaps the most relevant for applications (see also \cite{Jagers1975}, \cite{Jagers1989}, \cite{Dimitris2000}, \cite{Holmgren2017} and \cite{Svante2018}). 

\begin{remark} \label{Remark6} 
Let $\varphi$ and $\phi$ be two $h$-dependent descendant characteristics, for some $h \in \mathbb{N}_{0}$, satisfying \ref{C4}. Then, any linear combination of $\varphi$ and $\phi$ also satisfies \ref{C4}. Similarly, if $\varphi$ and $\phi$ satisfy \ref{C5}, then any linear combination of them also satisfies \ref{C5}. 
\end{remark}

The following proposition provides sufficient conditions under which assumptions \ref{C1}-\ref{C5} hold; see Appendix \ref{Append1} for the proof. Let $\vartheta \in (0, \alpha/2)$ be as defined in \ref{A7}. 

\begin{lemma} \label{lemma12}
Let $\varphi$ be a random characteristic. 
\begin{enumerate}[label=(\textbf{\roman*})]
\item If there exist $\delta_{1} \in [0, \alpha)$ and $\delta_{2} \in  (\alpha, \infty)$ such that $\sup_{t\in \mathbb{R}} e^{-(\delta_{1} t \wedge \delta_{2} t)} |\varphi(t)| \in L^{1}(\Omega, \mathcal{F}, \mathbb{P})$, then $\varphi$ satisfies \ref{C1}. \label{lemma12Pro3}

\item If there exist $\delta_{1} \in [0, \alpha/2)$ and $\delta_{2} \in  (\alpha/2, \infty)$ such that $\sup_{t\in \mathbb{R}} e^{-(\delta_{1} t \wedge \delta_{2} t)} |\varphi(t)| \in L^{2}(\Omega, \mathcal{F}, \mathbb{P})$, then $\varphi$ satisfies \ref{C2}-\ref{C3}. \label{lemma12Pro4}
\setcounter{Cond5}{\value{enumi}}
\end{enumerate}

\noindent Suppose further that that assumptions \ref{A1}-\ref{A7} holds. 
\begin{enumerate}[label=(\textbf{\roman*})]
\setcounter{enumi}{\value{Cond5}}
\item If there exists $\delta_{2} \in  [\vartheta, \infty)$ such that $\sup_{t\in (-\infty, 0)} e^{-\delta_{2} t} |\varphi(t)| \in L^{1}(\Omega, \mathcal{F}, \mathbb{P})$, then, for all $h \in \mathbb{N}_{0}$ and $\varepsilon >0$, \label{lemma12Pro2}
\begin{align}  \label{eq158}
\mathbb{E} \left[ \sum_{i=0}^{h} \sum_{|u|=i} \sup_{|s-t|\leq \varepsilon}|\varphi_{u}(s-S(u))\mathbf{1}_{(-\infty,0)}(s-S(u))| \right] < \infty. 
\end{align}

\item If there exist $\delta_{1} \in [\vartheta, \alpha]$ and $\delta_{2} \in  [\delta_{1}, \infty)$ such that $\sup_{t\in \mathbb{R}} e^{-(\delta_{1} t \wedge \delta_{2} t)} |\varphi(t)| \in L^{2}(\Omega, \mathcal{F}, \mathbb{P})$, then, for all $h \in \mathbb{N}_{0}$ and $\varepsilon >0$,  \label{lemma12Pro1}
\begin{align} \label{eq155} 
\left(\sum_{i=0}^{h} \left( \sum_{|u|=i} |\varphi_{u}(t-S(u))| \right)^{2}\right)_{|s-t|\leq \varepsilon} \quad \text{is uniformly integrable}.
\end{align}
\end{enumerate}
\end{lemma}

%%%%%%%%%%%%%%%%%%%%%%%%%%%%%%%%%%%%%%%%%%%%%%%%%%%%%%%%%%%%%%%%%%%%%%%%%%%
\subsection{Main results} \label{MainSec}
%%%%%%%%%%%%%%%%%%%%%%%%%%%%%%%%%%%%%%%%%%%%%%%%%%%%%%%%%%%%%%%%%%%%%%%%%%%

This section presents our main result. However, additional technical details and  one more assumption is required before stating our main theorem (Theorem \ref{Theo2} below) in its full generality. Following \cite{Iksanov2021}, we also impose the following technical condition.

\begin{enumerate}[label=(\textbf{E.\arabic*})]
\item \label{E1} Let $\varphi$ be a characteristic that is either non-negative or satisfies $\varphi(t) \in L^{1}(\Omega, \mathcal{F}, \mathbb{P})$ for every $t \in \mathbb{R}$. The mean $m_{t}^{\mathbb{E}[\varphi]} \coloneqq \mathbb{E}[\mathcal{Z}_{t}^{\mathbb{E}[\varphi]}]$ admits the expansion, 
\begin{align}
m_{t}^{\mathbb{E}[\varphi]}=\mathbf{1}_{[0, \infty)}(t) \sum_{\lambda \in \Lambda_{\geq}} e^{\lambda t} \sum_{i=0}^{k(\lambda)-1} a_{\lambda, i} t^{i}+r(t), \quad \text{for} \quad t \in \mathbb{G},
\end{align}
\noindent for some constant $a_{\lambda, i} \in \mathbb{C}$ and a function $r: \mathbb{G} \rightarrow \mathbb{C}$ satisfying $|r(t)| \leq C e^{\alpha t / 2} /(1+t^{2})$ for all $t \in \mathbb{G}$ and some finite constant $C \geq 0$. 
\end{enumerate}

Assumption \ref{E1} is analogous to \cite[assumption (2.20)]{Iksanov2021}, but we require the expansion for $m_{t}^{\mathbb{E}[\varphi]}$ instead of $m_{t}^{\varphi}$. This modification is intentional and not a limitation. A careful examination of the proofs in \cite{Iksanov2021} reveals that the expansion of $m_{t}^{\mathbb{E}[\varphi]}$ is the essential condition. Moreover, if \ref{C1} is satisfied, then $m_{t}^{|\mathbb{E}[\varphi]|} <\infty$, which ensures $m_{t}^{\mathbb{E}[\varphi]} <\infty$ for all $t \in \mathbb{R}$ (see \eqref{eq109} in Section \ref{ExpSection}). However, neither \ref{C1} nor \ref{C2} guarantees the finiteness of $m_{t}^{\varphi}<\infty$, making the assumption about a potentially undefined quantity somewhat inconsistent. It is also worth noting from \cite[Remark 2.16]{Iksanov2021} that $r(t)$ is real for every $t \in \mathbb{G}$.

In \cite{Iksanov2021}, sufficient conditions for \ref{E1} are given. Specifically, \cite[Lemma 7.1]{Iksanov2021} covers the lattice case ($\mathbb{G} = \mathbb{Z}$) and \cite[Lemma 7.6]{Iksanov2021} the non-lattice case ($\mathbb{G} = \mathbb{R}$). Indeed, the proofs of these lemmas establish the expansion for $m_{t}^{\mathbb{E}[\varphi]}$ rather than $m_{t}^{\varphi}$. In Lemma \ref{lemma13}, we provide a different set of sufficient conditions for \ref{E1} when the characteristic function vanishes on the negative half-line. As highlighted in \cite{Iksanov2021}, there are examples where the previously mentioned results do not apply. In such cases, directly verifying \ref{E1} might be more straightforward. Therefore, following the approach in \cite{Iksanov2021}, we present our main result, Theorem \ref{Theo2}, for general branching processes counted with a characteristic $\varphi$ that satisfies \ref{E1}. \\

For $t \in \mathbb{R}$, define 
\begin{align} \label{eq84} 
H_{\Lambda}(t) \coloneqq \sum_{\lambda \in \Lambda} e^{\lambda t} \sum_{i=0}^{k(\lambda)-1} \sum_{j=0}^{i} a_{\lambda, i}\binom{i}{j} t^{j} W^{(i-j)}(\lambda)
\end{align}
\noindent and 
\begin{align} \label{eq86}
H_{\partial \Lambda}(t) \coloneqq \sum_{\lambda \in \partial \Lambda} e^{\lambda t} \sum_{i=0}^{k(\lambda)-1} a_{\lambda, i} t^{i},
\end{align}
\noindent with $H_{\partial \Lambda}(t) = 0$, if $\partial \Lambda = \emptyset$.  We set $H(t) \coloneqq H_{\Lambda}(t)+H_{\partial \Lambda}(t)$, for $t \in \mathbb{R}$. Since $\mathbb{E}[W^{(i)}(\lambda)] = \delta_{0i}$, for $\lambda \in \Lambda$ and $i= 0, \dots, k(\lambda)-1$, note that
\begin{align} \label{eq87} 
\mathbb{E}[H_{\Lambda}(t)] = \sum_{\lambda \in \Lambda} e^{\lambda t} \sum_{i=0}^{k(\lambda)-1}  a_{\lambda, i} t^{i}.
\end{align}
\noindent Further, for any $\lambda \in \partial \Lambda$ and $i=0, \ldots, k(\lambda)-1$, we define the random variable 
\begin{align} \label{eq119}
R_{\lambda, i} \coloneqq \sum_{j=i}^{k(\lambda)-1} a_{\lambda, j}\binom{j}{i} \sum_{k=1}^{N}\left(-X_{k}\right)^{j-i} e^{-\lambda X_{k}}.
\end{align}
\noindent Since \ref{A7} implies \cite[assumption (A3)]{Iksanov2021} (see \cite[Remark 2.1]{Iksanov2021}), then $R_{\lambda, i} \in L^{2}(\Omega, \mathcal{F}, \mathbb{P})$ for all $\lambda \in \partial \Lambda$ and $i=0, \ldots, k(\lambda)-1$. Thus, we may define
\begin{align} \label{eq120} 
\rho_{i}^{2}\coloneqq \sum_{\substack{\lambda \in \partial \Lambda: \\ k(\lambda)>i}} {\rm Var}(R_{\lambda, i}),
\end{align}
\noindent where ${\rm Var}[R_{\lambda, i}]=\mathbb{E}[|R_{\lambda, l}- \mathbb{E}[R_{\lambda, l}]|^{2}]$. %(If $\mathcal{Y}$ is a complex-valued random variable with finite mean, we set $\mathrm{Var}(\mathcal{Y}) \coloneqq \mathbb{E}[|\mathcal{Y} - \mathbb{E}[\mathcal{Y}]|^{2}]$.)

For a sequence of random variables $(\mathcal{X}_{n})_{n \in \mathbb{N}_{0}}$ defined on $(\Omega, \mathcal{F}, \mathbb{P})$ that converges in distribution to some random variable $\mathcal{X}$, we say that the convergence is stable if additionally for all continuity points $x$ of the distribution function of $\mathcal{X}$ and all $E\in \mathcal{F}$, the limit $\lim \mathbb{P}(\{ \mathcal{X} \leq x\} \cap E)$ exists (see e.g., \cite[p. 56]{Hall1980}). An alternative characterization of stable convergence can be found in \cite[Condition (B')]{Aldous1978}. We use $\xrightarrow[]{\rm st}$ for stable convergence.

For a measure $\nu$ (possibly random) on the Borel $\sigma$-algebra $\mathcal{B}(\mathbb{R})$ and a measurable function $f: \mathbb{R} \rightarrow \mathbb{R}$, the Lebesgue-Stieltjes convolution of $\nu$ and $f$ is denoted by $f \ast \nu$, i.e., for $t \in \mathbb{R}$, $f \ast \nu(t) = \int_{(-\infty,\infty)} f(t-x)  \nu( {\rm d} x)$ whenever the integral is well-defined.

\begin{theorem} \label{Theo2} 
Suppose that assumptions \ref{A1}-\ref{A8} hold and let 
$\varphi$ be an $h$-dependent descendant characteristic, for some $h \in \mathbb{N}_{0}$, satisfying \ref{C1}-\ref{C5}. Suppose also that $m_{t}^{\mathbb{E}[\varphi]}$ satisfies \ref{E1}. Let $n \coloneqq \max\{i \in \mathbb{N}_{0}: \rho_{i} >0\}$ with $n = -1$ if the set is empty; in that case we set $\rho_{-1} \coloneqq 0$. Then, there exists a finite constant $\sigma \geq 0$ such that, with
\begin{align} \label{eq164} 
a_{t}^{2} \coloneqq \sigma^{2} + \frac{\rho_{n}^{2}}{2n+1}t^{2n +1}, \quad t >0,
\end{align}
\noindent it holds:
\begin{enumerate}[label=(\textbf{\roman*})]
\item if $\sigma^{2} = \rho_{n}^{2} = 0$, then there exists an individual characteristic $\phi$ such that 
\begin{align} \label{eq68}
t \in \mathbb{G}\cap \mathbb{R}_{+} \mapsto H(t) +  \sum_{0\leq |u| \leq h-1} (\varphi_{u}^{(h-|u|)}(t-S(u)) - \phi_{u}(t-S(u)))  + \sum_{|u| = h} r(t - S(u))
\end{align}
\noindent is a version of $t \in \mathbb{G}\cap \mathbb{R}_{+} \mapsto \mathcal{Z}^{\varphi}_{t}$. (If $h=0$, the sum $\sum_{0\leq |u| \leq h-1}$ in \eqref{eq68} is defined to be $0$.) \label{Theo2Claim1}
\item if $\sigma^{2}>0$ or  $\rho_{n}^{2} > 0$, then \label{Theo2Claim2}
\begin{align} \label{eq165} 
a_{t}^{-1} e^{-\frac{\alpha}{2}t} \left(\mathcal{Z}_{t}^{\varphi} - H(t) \right) \xrightarrow[]{\rm st} \left( \frac{W}{\beta} \right)^{1/2} \mathcal{N},
\end{align}
\noindent as $t \rightarrow \infty$, $t \in \mathbb{G}$,  where $\beta \in (0, \infty)$ is defined in \eqref{eq3}, $W$ is the limit of Nerman's martingale and $\mathcal{N}$ is a standard normal random variable independent of $\mathcal{F}$ (and thus of $W$).
\end{enumerate}

If $n =-1$, the constant $\sigma$ can be explicitly computed, specifically:
\begin{align} \label{eq121}
\sigma^{2} \coloneqq  \int_{\mathbb{G}} \mathbb{E}[(\chi^{(\varphi + (m^{\mathbb{E}[\varphi]} - \mathbb{E}[H_{\Lambda}])\ast \xi,h)}(s))^{2}] e^{-\alpha s}  \ell({\rm d} s),
\end{align}
\noindent where $\mathbb{E}[H_{\Lambda}]$ denotes the function $t \in \mathbb{R} \mapsto \mathbb{E}[H_{\Lambda}(t)]$. 
\end{theorem}

Theorem \ref{Theo2} extends \cite[Theorem 2.8]{Iksanov2021} to descendant characteristics. Indeed, recall that the individual characteristics considered in \cite{Iksanov2021} are $0$-dependent descendant characteristics. Moreover, if $\varphi$ is a $0$-dependent descendant characteristic and $n =-1$ in Theorem \ref{Theo2}, then by \eqref{eq53} with $h=0$, we obtain
\begin{align}
\sigma^{2} = \int_{\mathbb{G}} \mathrm{Var}(\varphi(s) + (m^{\mathbb{E}[\varphi]} - \mathbb{E}[H_{\Lambda}])\ast \xi(s)) e^{-\alpha s}  \ell({\rm d} s),
\end{align}
\noindent which is the same constant as in \cite[Theorem 2.8]{Iksanov2021}. Moreover, if $\sigma^{2} = \rho_{n}^{2} = 0$, then from Theorem \ref{Theo2} \ref{Theo2Claim1} it follows that $t \in \mathbb{G}\cap \mathbb{R}_{+} \mapsto H(t) + r(t)$ is a version of $t \in \mathbb{G}\cap \mathbb{R}_{+} \mapsto \mathcal{Z}^{\varphi}_{t}$; see also  \cite[Theorem 2.15 (i)]{Iksanov2021}. We also point out that the individual characteristic $\phi$ in Theorem \ref{Theo2} \ref{Theo2Claim1} can be given explicitly; its expression can be found in the proof of Theorem \ref{Theo2} (see \eqref{eq131}), which we omitted here to avoid introducing further notation.

As a corollary to Theorem \ref{Theo2}, we present a simplified version that holds under some additional assumptions. To this end, we require the introduction of additional notation. For a measure $\eta$ on the Borel $\sigma$-algebra $\mathcal{B}(\mathbb{R})$ and $n \in \mathbb{N}$, $\eta^{\ast n}$ denotes the $n$-fold convolution of $\eta$, and $\eta^{\ast 0}$ the Dirac measure at $0$. Let $\eta$ be a probability measure on $\mathbb{R}$. We say that $\eta$ is {\sl spread-out} if for some $k \geq 1$, $\eta^{\ast k}$ has a non-zero absolutely continuous component (i.e.\ it has a density w.r.t.\ Lebesgue measure); see for e.g.\ \cite[Section 1, Chapter VII]{Soren2003}. 

Define the measure $\mu_{\alpha}$ on $\mathbb{R}$ by 
\begin{align} \label{eq153}
\mu_{\alpha}({\rm d}x) \coloneqq  e^{-\alpha x} \mu({\rm d}x),
\end{align}
\noindent where $\alpha >0$ is the Malthusian parameter defined in $\ref{A4}$. In particular,  $\mu_{\alpha}$ is a probability measure concentrated on $\mathbb{R}_{+}$, as $\mu_{\alpha}((-\infty,0)) =0$ and $\mu_{\alpha}([0,\infty)) =1$. For $k \in \mathbb{N}_{0} \cup \{\infty\}$, define the measure
\begin{align} \label{eq62}
\nu^{(k)}({\rm d} x) \coloneqq  \sum_{i=0}^{k}\mu^{\ast i}({\rm d} x).
\end{align}
\noindent Note that $\nu^{(\infty)}$ is a renewal measure that may be defective or excessive. Moreover, $\nu^{(\infty)}([0,t])<\infty$ (and thus, $\nu^{(k)}([0,t])<\infty$, for all $k \in \mathbb{N}_{0}$) for all $t \in \mathbb{R}_{+}$, if $\mu(\{ 0\}) < 1$ and $\mu([0,t]) < \infty$; see for e.g.\ \cite[Lemma 5.2.1, p.\ 158]{Jagers1975}. Our assumptions \ref{A1}-\ref{A4} imply the latter two conditions (see Remark \ref{Remark3}). Note also that 
\begin{align} \label{eq123}
e^{-\alpha t} \nu^{(\infty)}({\rm d}x) = \sum_{i=0}^{\infty} \mu_{\alpha}^{\ast i}({\rm d}x).
\end{align}
\noindent Recall from \eqref{eq3} that $\beta \coloneqq \int_{[0,\infty)} x e^{-\alpha x} \mu({\rm d}x) \in (0,\infty)$ (by \ref{A1}-\ref{A7}). 

If $\mu_{\alpha}$ is spread-out, then by \cite[Theorem 1.1, Chapter VII]{Soren2003}, there exists a decomposition of the renewal measure $e^{-\alpha t} \nu^{(\infty)}({\rm d}x) = \nu^{(\infty)}_{\alpha, 1}({\rm d}x) + \nu^{(\infty)}_{\alpha, 2}({\rm d}x)$, where $\nu^{(\infty)}_{\alpha, 1}$ and $\nu^{(\infty)}_{\alpha, 2}$ are non-negative measures on $[0, \infty)$, $\nu^{(\infty)}_{\alpha, 2}$ is a finite measure and $\nu^{(\infty)}_{\alpha, 1}$ has a bounded continuous density w.r.t.\ Lebesgue measure $f^{(\infty)}_{\alpha, 1}: [0, \infty) \rightarrow [0, \infty)$ (i.e., $\nu^{(\infty)}_{\alpha, 1}({\rm d}x) = f^{(\infty)}_{\alpha, 1}(x) {\rm d} x$) satisfying $f^{(\infty)}_{\alpha, 1}(x) \rightarrow \beta^{-1}$, as $x \rightarrow \infty$. This  decomposition is known as the {\sl Stone's decomposition} of the renewal measure $e^{-\alpha t} \nu^{(\infty)}({\rm d}x)$. We consider the following  condition.
\begin{enumerate}[label=(\textbf{S.\arabic*})]
\item \label{S1} 
If $\mu_{\alpha}$ is spread-out, there exists a Stone's decomposition of its associated renewal measure $e^{-\alpha x} \nu^{(\infty)}({\rm d}x)$ such that, for some $\varepsilon \in (\alpha/2, \alpha)$, 
\begin{align}
\int_{[0, \infty)} e^{\varepsilon x} \nu^{(\infty)}_{\alpha, 2}({\rm d}x) < \infty \quad \text{and} \quad f^{(\infty)}_{\alpha, 1}(x) = \beta^{-1} + g(x),
\end{align}
\noindent for $x \in [0,\infty)$, where $g: [0, \infty) \rightarrow \mathbb{R}$ is a function satisfying $|g(t)| \leq C e^{-\varepsilon t}$ for all $t \in [0, \infty)$ and some finite constant $C \geq 0$. 
\end{enumerate}

\begin{corollary} \label{corollary5} 
Suppose that assumptions \ref{A1}-\ref{A8} hold and that $\mu_{\alpha}$ is spread-out such that there exists a Stone's decomposition of its associated renewal measure $e^{-\alpha x} \nu^{(\infty)}({\rm d}x)$ that satisfies \ref{S1}. Let $\varphi$ be any non-negative $h$-dependent descendant characteristic that vanishes on the negative half-line (i.e., $\varphi(t)= 0$, for $t<0$) and satisfies \ref{C2}-\ref{C3} and \ref{C5}. Suppose further that 
\begin{align} \label{eq129}
\mathbb{E}[\varphi(t)] \leq C e^{(\alpha-\delta) t} \mathbf{1}_{[0,\infty)}(t), \quad \text{for every} \, \, t \in \mathbb{R},
\end{align}
\noindent for some $\delta \in (\varepsilon, \alpha)$ and finite constant $C > 0$, where $\varepsilon \in (\alpha/2, \alpha)$ is as in \ref{S1}. Then, there exists $\sigma \geq 0$ so that, for $a_{\alpha} \coloneqq \beta^{-1} \int_{0}^{\infty} \mathbb{E}[\varphi(x)] e^{- \alpha x} {\rm d} x$,
\begin{align}
e^{-\frac{\alpha}{2}t} \left(\mathcal{Z}_{t}^{\varphi} - a_{\alpha} e^{\alpha t} W \right) \xrightarrow[]{\rm st} \sigma \left( \frac{W}{\beta} \right)^{1/2} \mathcal{N}, \quad \text{as} \quad t \rightarrow \infty, \, t \in \mathbb{G},
\end{align}
\noindent where $\mathcal{N}$ is a standard normal random variable independent of $W$. The constant $\sigma$ can be explicitly computed; see \eqref{eq121} (and in particular, $\mathbb{E}[H_{\Lambda}(t)] = a_{\alpha} e^{\alpha t}$, for $t \in \mathbb{R}$).
\end{corollary}

In the setting of Theorem \ref{Theo2}, we also make the following remarks. 

\begin{remark}
Recall that $\mathcal{Z}_{t}$, defined in \eqref{eq15}, denotes the total number of births up to and including time $t \in \mathbb{R}$. Then, Nerman's law of large numbers (\cite[Theorem 6.1]{JagersN1984} or \cite[Theorem 3.2]{Dimitris2000}) implies that
\begin{align}
e^{-\alpha t} \mathcal{Z}_{t} \rightarrow \frac{W}{\beta} \int_{[0, \infty)} e^{- \alpha x} \ell({\rm d} x) = \frac{c_{\alpha}}{\beta} W, \quad \text{almots surely as} \quad t \rightarrow \infty, \, t \in \mathbb{G},
\end{align}
\noindent where $c_{\alpha} = (1-e^{-\alpha})^{-1}$ in the lattice case and $c_{\alpha} = \alpha^{-1}$ in the non-lattice case. Let $\varphi$ be an $h$-dependent descendant characteristic (for some $h \in \mathbb{N}_{0}$) satisfying the assumptions in Theorem \ref{Theo2}. Then, the stable convergence \eqref{eq165} in Theorem \ref{Theo2} yields
\begin{align} \label{eq166}
a^{-1}_{t} \left( \frac{c_{\alpha}}{\mathcal{Z}_{t}} \right)^{1/2} (\mathcal{Z}_{t}^{\varphi} -H(t))  \xrightarrow[]{\rm d} \mathcal{N}, \quad \text{as} \quad t \rightarrow \infty, \, t \in \mathbb{G},
\end{align}
\noindent conditionally on the event that the underlying CMJ-branching process survives, where $a_{t}$ is defined in \eqref{eq165}. We recall that \ref{A3} implies the CMJ-branching process survives with positive probability. The convergence in distribution \eqref{eq166} follows similarly to that of \cite[(2.26)]{Iksanov2021}. Thus, the details are left to the reader.
\end{remark}

\begin{remark}
Suppose that assumptions \ref{A1}-\ref{A8} hold. For $d \in \mathbb{N}$ and $i =1, \dots, d$,  let $\varphi_{i}$ be an $h_{i}$-dependent descendant characteristic for some $h_{i} \in \mathbb{N}_{0}$. Set $h = \max\{h_{1}, \dots, h_{d} \}$. Then, the characteristics $\varphi_{1}, \dots, \varphi_{d}$ are $h$-dependent descendant characteristics. Suppose that they satisfy \ref{C1}-\ref{C5}. Suppose also that each $m_{t}^{\mathbb{E}[\varphi_{i}]}$ satisfies \ref{E1} (with potentially different coefficients and remainder function). Then, applying the Cram\'er-Wold device and Theorem \ref{Theo2}, we obtain the joint convergence in distribution, up to renormalization, of the vector  $(\mathcal{Z}_{t}^{\varphi_{1}}, \dots, \mathcal{Z}_{t}^{\varphi_{d}})$, as $t \rightarrow \infty$, $t \in \mathbb{G}$. (It is straightforward to see that the characteristic function corresponding to any linear combination of $\varphi_{1}, \dots, \varphi_{d}$ satisfies the assumptions of Theorem \ref{Theo2}.) 

On the other hand, in the specific setting of Theorem \ref{Theo2} and up to renormalization, one can show that  $(\mathcal{Z}_{t-s}^{\varphi} -H(t-s))_{s \in \mathbb{R}}$ converges in the sense of finite-dimensional distributions to a centred Gaussian process. We leave the precise details to the interested reader (see also \cite[Corollary 2.20]{Iksanov2021} for the case of individual characteristics). 
\end{remark}

Let us outline the approach to prove Theorem \ref{Theo2}. Although the approach may appear similar to that used by Iksanov et al. \cite{Iksanov2021}, there are several key differences and technicalities that need to be addressed. A main ingredient in the proof of our central result is Theorem \ref{Theo1}, a central limit theorem for general branching processes counted with centred characteristics. This theorem is analogous to \cite[Theorem 6.3]{Iksanov2021}, which is valid only for individual characteristics. The proof of Theorem \ref{Theo1} (and similarly that of \cite[Theorem 6.3]{Iksanov2021}) relies on the application of the well-known central limit theorem for square-integrable martingales. However, our more general setting of $h$-dependent descendant characteristics introduces a more intricate dependency structure due to the lack of independence among characteristics, unlike the setting considered in \cite{Iksanov2021}.  Therefore, we have to construct a different martingale structure to handle these dependencies and must be more meticulous with the estimations needed to verify the conditions of the martingale central limit theorem. Once Theorem \ref{Theo1} is established, the subsequent steps of the proof are analogous to those in \cite[Theorem 2.15]{Iksanov2021}. However, a non-trivial aspect is ensuring that certain characteristics, also appearing in \cite{Iksanov2021}, satisfy our technical conditions \ref{C4} and \ref{C5}.

We should also point out that the general case of descendant characteristics, allowing dependencies on the entire descendant process, introduces considerable complexities. In this paper, we consider $h$-dependent descendant characteristics, a framework that allows for some dependency between individuals across different generations while crucially maintaining independence between individuals separated by more than $h$ generations. This specific independence property is essential for defining a suitable martingale, which is the key ingredient for proving Theorem \ref{Theo1}. If, however, characteristics were to depend on the entire descendant process, our martingale approach would no longer be viable. In a way, our dependency structure is reminiscent of the framework for a central limit theorem for $h$-dependent variables. Therefore, we leave the general case of descendant characteristics open for future research.

%%%%%%%%%%%%%%%%%%%%%%%%%%%%%%%%%%%%%%%%%%%%%%%%%%%%%%%%%%%%%%%%%%%%%%%
\section{Application to fringe trees}  \label{Applicattions}
%%%%%%%%%%%%%%%%%%%%%%%%%%%%%%%%%%%%%%%%%%%%%%%%%%%%%%%%%%%%%%%%%%%%%%%

In this section, we apply our results to study the asymptotic normality of the number of fringe trees of a specific type within family trees generated by Crump-Mode-Jagers branching processes. In particular, we address one of the open questions posed in \cite[Section 14]{Holmgren2017}.

Let $\mathbb{T}$ be the set of all finite plane rooted trees (also called
ordered rooted trees); see e.g., \cite{Drmota2009}. Denote the size, i.e.\ the number of vertices, of a tree $T \in \mathbb{T}$ by $|T|$. Recall that the depth $h(v)$ of a vertex $v \in T$ is its distance from the root of $T$. The {\sl height} $h_{T}$ of $T$ is defined as $h_{T} \coloneqq \max_{v \in T} h(v)$, the maximum depth of a vertex. For $T \in \mathbb{T}$ and a vertex $v \in T$, let $T_{v}$ be the subtree of $T$ rooted at $v$ consisting of $v$ and all its descendants. We call $T_{v}$ a {\sl fringe subtree} of $T$. We regard $T_v$ as an element of $\mathbb{T}$. We write $T \simeq T^{\prime}$ when the trees $T$ and $T^{\prime}$ in $\mathbb{T}$ are isomorphic rooted trees.  Here, we consider trees to be isomorphic if there exists an isomorphism that preserves the root and the order of children. Let, for $T, T^{\prime} \in \mathbb{T}$, 
\begin{align}
N_{T^{\prime}}(T) \coloneqq  |\{v \in T: T_{v} \simeq T^{\prime}\}| = \sum_{v \in T} \mathbf{1}_{\{T_{v} \simeq T^{\prime}\}},
\end{align}
\noindent i.e., the number of fringe subtrees of $T$ that are equal (i.e., isomorphic to) to $T^{\prime}$. 

The {\sl family tree} of all individuals ever born in the CMJ-branching process is denoted by 
\begin{align}
\mathcal{T}\coloneqq \{u \in \mathcal{I}: S(u)<\infty\}. 
\end{align} 
\noindent This is a (generally infinite) tree obtained from the branching process by ignoring the time structure; in other words, it has the individuals as vertices, with the initial individual as the root, and the children of a vertex in the tree are the same as the children in the branching process. For $u \in \mathcal{I}$ and $t \in \mathbb{R}_{+}$, define the subtree of $\mathcal{T}$,
\begin{align}
\mathcal{T}_{t}^{u} \coloneqq \{v \in u\mathcal{I}:  S(v) - S(u) \leq t\}.
\end{align}
In particular, $\mathcal{T}_{t} \coloneqq \mathcal{T}_{t}^{\varnothing}$ is the subtree consisting of all individuals born up to and including time $t$. Note that the number of vertices $|\mathcal{T}_{t}| = \mathcal{Z}_{t}$. Under assumptions \ref{A1}-\ref{A4}, Remark \ref{Remark3} implies that $\mathcal{Z}_{t} < \infty$ is almost surely finite for every $t \in \mathbb{R}$. The tree $\mathcal{T}_{t}$ is an unordered tree. We then add an ordering of the children of each individual by taking them in order of birth (which is the standard custom). 

The subtree $\mathcal{T}^{u}_{t-S(u)}$ is the fringe subtree associated to the individual $u$ born before time $t$ relative to the tree $\mathcal{T}_{t}$ (i.e., $\mathcal{T}^{u}_{t-S(u)}$ is the subtree of $\mathcal{T}_{t}$ rooted at $u$ consisting of $u$ and all its descendants born up to and including time $t$). Naturally, we also view the subtree $\mathcal{T}^{u}_{t-S(u)}$ of $\mathcal{T}_{t}$ as an ordered rooted tree. For a finite rooted tree $T \in \mathbb{T}$, define the random (descendent) characteristic, 
\begin{align} 
\varphi^{T}(t) \coloneqq \mathbf{1}_{\{\mathcal{T}_{t} \simeq T \}} \mathbf{1}_{[0, \infty)}(t), \quad \text{for} \quad t \in \mathbb{R}.
\end{align}
\noindent That is, $\varphi^{T}(t)$, for $t \in \mathbb{R}_{+}$, is the indicator of $\mathcal{T}_{t}$ being isomorphic to $T$. Note that, for $u \in \mathcal{I}$ and $t \in \mathbb{R}_{+}$, $\varphi^{T}_{u}(t) = \mathbf{1}_{\{\mathcal{T}^{u}_{t} \simeq T \}}$. Then, 
\begin{align} 
N_{T}(\mathcal{T}_{t}) = \mathcal{Z}_{t}^{\varphi^{T}} = \sum_{u \in \mathcal{I}} \varphi_{u}^{T}(t-S(u)), \quad \text{for} \quad t \in \mathbb{R},
\end{align}
\noindent where $\mathcal{Z}_{t}^{\varphi^{T}} = (\mathcal{Z}_{t}^{\varphi^{T}})_{t \in \mathbb{R}}$ is the general branching process counted with characteristic $\varphi^{T}$. Since the characteristic $\varphi^{T}$ is non-negative, the series above always makes sense, but it may be infinite. On the other hand, $\varphi^{T}(t) = 0$, for $t <0$ and $|\varphi^{T}(t)| \leq \mathbf{1}_{[0, \infty)}(t)$, for $t \in \mathbb{R}$. Therefore, if \ref{A1}-\ref{A4} hold, then $N_{T}(\mathcal{T}_{t}) \leq  \mathcal{Z}_{t} < \infty$ almost surely for every $t \in \mathbb{R}$. Furthermore, $\mathbb{E}[N_{T}(\mathcal{T}_{t})] \leq \mathbb{E}[\mathcal{Z}_{t}] < \infty$, for all $t \in \mathbb{R}$ as noted in Remark \ref{Remark3}.

Theorem \ref{Theo2} provides a solution to \cite[Problem 14.1]{Holmgren2017}. To be precise, observe that, for $T \in \mathbb{T}$, the characteristic $\varphi^{T}$ is an $h_{T}$-dependent descendant characteristic. Furthermore, under assumptions \ref{A1}-\ref{A7}, Lemma \ref{lemma12} implies that $\varphi^{T}$ satisfies \ref{C1}-\ref{C5}. Therefore, assuming \ref{A8} and \ref{E1},  Theorem \ref{Theo2} establishes that, up to a renormalization, $N_{T}(\mathcal{T}_{t})$, the number of fringe trees equal to $T$, is asymptotically normal as $t \rightarrow \infty$, $t \in \mathbb{G}$.

On the other hand, recall the probability measure $\mu_{\alpha}$ on $\mathbb{R}$ defined in \eqref{eq153}. If $\mu_{\alpha}$ is spread-out and the conditions of Corollary \ref{corollary5} are satisfied, then it follows that, up to a renormalization, $N_{T}(\mathcal{T}_{t})$ is asymptotically normal as $t \rightarrow \infty$, $t \in \mathbb{G}$. 

We now present specific examples of Crump-Mode-Jagers branching processes associated with several known families of random trees, for which our assumptions hold.

\begin{example}[supercritical Galton-Watson process]
Consider a random variable $N$ taking values in $\mathbb{N}_{0}$ with $m \coloneqq \mathbb{E}[N] \in (1, \infty)$ and $\mathbb{E}[N^{2}] < \infty$. Let $\xi(\cdot) = N \delta_{1}(\cdot)$. Therefore, the branching process is a supercritical Galton-Watson process, where  $\xi([0, \infty)) = N$ represents the number of offspring per individual. In this case, the intensity measure $\mu$ of $\xi$ is lattice with span $1$. In particular, for $t \in \mathbb{N}_{0}$, the family tree $\mathcal{T}_{t}$ of the general branching process corresponds to the family tree of a supercritical Galton-Watson process up to generation $t$. Clearly, $\xi$ satisfies \ref{A1}-\ref{A3}. Further,
\begin{align}
\hat{\mu}(\lambda)  = me^{-\lambda}, \quad  \text{for} \quad \lambda \in \mathbb{C},
\end{align}
\noindent and thus, \ref{A4} holds with $\alpha = \log m$. Furthermore, $\beta=1$. Note that
\begin{align} \label{eq162}
\mathbb{E} \left[  \left(\sum_{i=1}^{N} e^{-\vartheta} \right)^{2} \right] = e^{2 \vartheta} \mathbb{E}[N^{2}] < \infty, \quad  \text{for all} \quad \vartheta \in \mathbb{R},
\end{align}
\noindent that is, \ref{A7} holds. The equation $\hat{\mu}(\lambda) -1=0$ has only one root in the strip $\mathrm{Im}(\lambda) \in (-\pi, \pi]$ and thus \ref{A8} also holds with $\Lambda_{\geq} = \{ \alpha\}$. 

Set $G_{\mu}(z) \coloneqq \sum_{k =0}^{\infty} \mu(\{k\}) z^{k} = m z$, for $z \in \mathbb{C}$. Then, for any for $\vartheta \in (0,\alpha/2)$ such that \eqref{eq162} holds, and for any $\theta \in (\vartheta, \alpha/2)$, there are not roots to $G_{\mu}(z) - 1=0$ in $\{ z \in \mathbb{C}: e^{-\alpha/2} < |z| \leq e^{-\theta} \}$. Since the characteristic $\varphi^{T}$, for  $T \in \mathbb{T}$, satisfies \cite[condition (7.3)]{Iksanov2021},  \cite[Lemma 7.1]{Iksanov2021} implies that, for $t \in \mathbb{Z}$,
\begin{align}
m_{t}^{\mathbb{E}[\varphi^{T}]} = \mathbf{1}_{\mathbb{N}_{0}}(t) e^{\alpha t} b_{\alpha,0}\sum_{k=0}^{\infty} \mathbb{E}[\varphi^{T}(k)] e^{-\alpha k}  + O(e^{\theta t} \wedge e^{\gamma t}), \quad \text{as} \, \, t \rightarrow \pm \infty,
\end{align}
\noindent for some constants $b_{\alpha,0} \in \mathbb{R}$ and $\gamma \in (\alpha/2, \alpha)$. From \eqref{eq66} below (or \cite[Proposition 7.9]{Iksanov2021}), we know that $b_{\alpha,0} = \beta^{-1}=1$. Thus, for $t \in \mathbb{Z}$,
\begin{align}
m_{t}^{\mathbb{E}[\varphi^{T}]} = \mathbf{1}_{\mathbb{N}_{0}}(t) m^{ t}  \sum_{k=0}^{\infty} \mathbb{E}[\varphi^{T}(k)] m^{- k}  + O(e^{\theta t} \wedge e^{\gamma t}), \quad \text{as} \, \, t \rightarrow \pm \infty,
\end{align}
\noindent that is, $\varphi^{T}$ satisfies \ref{E1}. 
\end{example}

\begin{example}[conservative fragmentation process] \label{Example2}
We consider the fragmentation process introduced by Kolmogorov (see \cite{Kolmogoroff1941}, \cite[Chapter 1]{Bertoin2006}, and \cite[Section 1.6]{Bertoin2006} for further details and references). Let $b \geq 2$ an integer, and $(V_{1}, \dots, V_{b})$ be a random vector with $V_{i} \in [0,1]$, for $i = 1, \dots b$, such that $\sum_{i=1}^{b}V_{i} =1$. Assume $V_{i} < 1$ almost surely, allowing $V_{i} =0$ (but note that, a.s., $0<V_{i} < 1$ for at least one $i$). The distribution of the random vector $(V_{1}, \dots, V_{b})$ is commonly referred to as the {\sl dislocation law}. 

Starting with an object of mass $x \geq 1$ break it into $b$ pieces with masses $V_{1}x, \dots, V_{b}x$. This process is repeated recursively for each piece with mass $\geq 1$, using independent copies of the random vector $(V_{1}, \dots, V_{b})$ at each step. The process terminates almost surely after a finite number of steps, yielding a finite set of fragments with masses $<1$. The fragments with mass $\geq 1$ formed during this process are treated as the internal vertices of a random tree, the {\sl fragmentation tree} $\mathcal{T}_{x}^{\rm Frag}$, while the final fragments with mass $< 1$ are considered external vertices. Holmgren and Janson \cite[Section 9]{Holmgren2017} noted that the fragmentation tree $\mathcal{T}_{x}^{\rm Frag}$ coincides with the family tree $\mathcal{T}_{\log(x)}$ of the general branching process driven by the point process $\xi$. 

Let $\xi(\cdot) = \sum_{i=1}^{b} \mathbf{1}_{\{V_{i}>0 \}} \delta_{-\log V_{i}}(\cdot)$. Note that $\xi$ satisfies \ref{A1}-\ref{A3}. In particular, 
\begin{align}
\hat{\mu}(\lambda) =  \sum_{i=1}^{b} \mathbb{E}[V_{i}^{\lambda}], \quad  \text{for} \quad \mathrm{Re}(\lambda) > 0.
\end{align}
\noindent As pointed out in \cite[Section 1]{Svante2008}, $\hat{\mu}(\cdot)$ is bounded and analytic in the open right half-plane $\{ \lambda \in \mathbb{C}: \mathrm{Re}(\lambda) > 0\}$. Clearly, \ref{A4} holds with $\alpha = 1$ and \ref{A7} is satisfied. We also have that $\beta = \sum_{i=1}^{b} \mathbb{E}[V_{i}\log(1/V_{i})] \in (0, \infty)$. In the lattice case, \cite[Lemma 7.1]{Iksanov2021} provides sufficient conditions for $\varphi^{T}$ to satisfy condition \ref{E1}. In the non-lattice case, for instance, if $\mu_{\alpha}$ is spread-out and the conditions of Lemma \ref{lemma13} are satisfied, then $\varphi^{T}$ satisfies \ref{E1}. 
\end{example}

\begin{example} \label{Example1} %[linear preferential attachment] 
Suppose that $\xi$ is a Poisson point process on $[0, \infty)$ with intensity measure $\mu({\rm d} x)= a e^{bx} {\rm d} x$, for $a >0$ and $b \in \{-1, 0, 1\}$. If $b=-1$ or $b=1$, we additionally require $a >1$. Thus, $\xi$ satisfies \ref{A1}-\ref{A3}. Observe that
\begin{align}
\hat{\mu}(\lambda) = \int_{0}^{\infty} e^{-\lambda x} a e^{bx} {\rm d} x =\frac{a}{\lambda -b}, \quad  \text{for} \quad \mathrm{Re}(\lambda) > b.
\end{align}
\noindent Then, \ref{A4} holds with $\alpha = a+b$. In particular, we have that $\beta = a^{-1}$. Note that
\begin{align}
\int_{0}^{\infty} e^{-\vartheta x} a e^{bx} {\rm d} x < \infty \quad \text{and} \quad \int_{0}^{\infty} e^{-2\vartheta x} a e^{bx} {\rm d} x < \infty, \quad  \text{for all} \quad \vartheta \in \mathbb{R} \, \, \, \text{such that} \, \, \, \vartheta > b. 
\end{align}
\noindent Thus, by Campbell's theorem, \ref{A7} is satisfied. Moreover, it is not difficult to see that $\Lambda_{\geq} = \{ \alpha\}$, and thus \ref{A8} also holds. 

On the other hand, by \eqref{eq153}, $\mu_{\alpha}({\rm d} x) = a e^{-ax} {\rm d} x$, i.e., $\mu_{\alpha}$ is the distribution of an exponential random variable with parameter $a$. In particular, by \eqref{eq62} and \eqref{eq123}, $e^{-\alpha x}\nu^{(\infty)}({\rm d} x) = a{\rm d} x$; see \cite[Example 3.3.1]{Sidney1992}. Clearly, $t \in \mathbb{R} \mapsto \mathbb{E}[\varphi^{T}](t)e^{-\alpha t}$ is directly Riemann integrable. Recall that $\varphi^{T}$ is non-negative and that $\varphi^{T}(t) = 0$ for $t < 0$. Then, by \eqref{eq109} and \eqref{eq123}, we know that, for $t \in \mathbb{R}$, 
\begin{align}  
m_{t}^{\mathbb{E}[\varphi^{T}]} & = \mathbf{1}_{[0,\infty)}(t) e^{\alpha t} \int_{[0,\infty)} \mathbb{E}[\varphi^{T}](t-x) e^{-\alpha(t-x)} e^{-\alpha x}\nu^{(\infty)}({\rm d} x)\nonumber \\
& = \mathbf{1}_{[0,\infty)}(t) a e^{(a+b)t}  \int_{0}^{t} \mathbb{E}[\varphi^{T}(x)] e^{-(a+b)x} {\rm d} x.
\end{align}
\noindent Note that, for $\theta \in (0, (a+b)/2)$, $\mathbb{E}[\varphi^{T}(t)] \leq e^{\theta x}$, for $t \in [0, \infty)$. Hence,
\begin{align}  
m_{t}^{\mathbb{E}[\varphi^{T}]} & = \mathbf{1}_{[0,\infty)}(t) a e^{(a+b)t}  \int_{0}^{\infty} \mathbb{E}[\varphi^{T}(x)] e^{-(a+b)x} {\rm d} x  + \mathbf{1}_{[0,\infty)}(t) \frac{a e^{\theta t}}{a+b-\theta}.
\end{align}
\noindent Thus, $\varphi^{T}$ satisfies \ref{E1}. 

One could also use Lemma \ref{lemma13} to show that $\varphi^{T}$ satisfies \ref{E1}. To see this, note that the renewal measure $e^{-\alpha x}\nu^{(\infty)}({\rm d} x)$ admits the Stone's decomposition $e^{-\alpha t} \nu^{(\infty)}({\rm d}x) = \nu^{(\infty)}_{\alpha, 1}({\rm d}x) + \nu^{(\infty)}_{\alpha, 2}({\rm d}x)$, where $\nu^{(\infty)}_{\alpha, 1}({\rm d}x) = (a - a e^{-\theta x}){\rm d}x$ and $\nu^{(\infty)}_{\alpha, 2}({\rm d}x) = a e^{-\theta x} {\rm d}x$, for some $\theta \in ((a+b)/2, (a+b))$. This decomposition satisfies condition \ref{S1} for any $\varepsilon \in ((a+b)/2, \theta)$. Furthermore, $\varphi^{T}$clearly satisfies \eqref{eq124} since $\varphi^{T}(t) = 0$, for $t <0$ and $|\varphi^{T}(t)| \leq \mathbf{1}_{[0, \infty)}(t)$, for $t \in \mathbb{R}$. 
\end{example}

The CMJ-branching process in Example \ref{Example1}, particularly its family tree, is connected to well-known models of random trees with linear preferential attachment, as noted in \cite[Example 6.4]{Holmgren2017}. For instance: the case $b=0$ corresponds to the (non-preferential) random recursive tree \cite[Example 6.1]{Holmgren2017}; the case $b = 1$ corresponds to positive linear preferential attachment \cite[Example 6.6]{Holmgren2017}. The case $b=-1$ and $a=m$ (an integer) corresponds to the $m$-ary increasing tree \cite[Example 6.7]{Holmgren2017}, with $b=-1$ and $a=2$ specifically corresponding to the binary search tree \cite[Example 6.2]{Holmgren2017}.

More precisely, random trees with a specific number of vertices, or trees exhibiting other defined properties, can be constructed as family trees of the CMJ-branching processes when stopped at an appropriate time, as we will explain next; see \cite[Section 5]{Holmgren2017}. Let $\psi$ be an individual characteristic, which we call the {\sl weight}, satisfying $\psi(t) =0$ for $t < 0$. Define 
\begin{align} 
\tau^{\psi}(n) \coloneq \inf\{ t \geq 0: \mathcal{Z}_{t}^{\psi} \geq n\}
\end{align}
\noindent as the first time the total weight of the general branching process counted with characteristic $\psi$ is at least $n \in \mathbb{R}_{+}$. (As usual, we define $\inf \emptyset \coloneq \infty$.) We exclude the trivial case when $\psi(t) = 0$ for all $t \in \mathbb{R}$ almost surely (which would give $\tau^{\psi}(n) = \infty$ a.s.). Define 
\begin{align} 
T^{\psi}_{n} \coloneqq \mathcal{T}_{\tau^{\psi}(n)}
\end{align}
\noindent the family tree of the general branching process at the time the total weight reaches $n$ (provided this ever happens). Random trees $T^{\psi}_{n}$ defined in this way, for some CMJ-branching process, were the main object of study in \cite{Holmgren2017}. If the general branching process satisfies conditions \cite[(A1)-(A5)]{Holmgren2017} (particularly, $\mu$ is non-lattice) and $\psi$ satisfies condition \cite[(A6)]{Holmgren2017}, then $T^{\psi}_{n}$ is indeed well-defined; see \cite[Theorem 5.12]{Holmgren2017}. 

Setting $\psi(t) = \mathbf{1}_{[0, \infty)}(t)$, for $t \in \mathbb{R}$, yields $\mathcal{Z}^{\psi}_{t} = \mathcal{Z}_{t}$ (refer to \eqref{eq15}), and $T_{n} \coloneqq T^{\psi}_{n}$ is the family tree of the
branching process stopped when it has $n$ or more vertices. Furthermore, if $\mu$ is absolutely continuous with respect to Lebesgue measure, then almost surely no two individuals (vertices) are born at the same time, implying that $T_{n}$ has exactly $n$ vertices. In particular, this applies to the general branching process in Example \ref{Example1}. Depending on the parameters $a$ and $b$ defined there, $T_{n}$ can be a linear preferential attachment trees, including random recursive trees, $m$-ary increasing trees, or binary search trees, as discussed previously. Other examples of random tree classes constructible in this way include $m$-ary search trees, general preferential attachment trees, and Median-of-$(2\ell + 1)$ binary search trees; for further details, we refer to \cite{Holmgren2017}.

Holmgren and Janson in \cite[Theorem 5.14]{Holmgren2017} established a law of large numbers for $N_{T}(T_{n}^{\psi})$, for $T \in \mathbb{T}$, under mild conditions, as $n \rightarrow  \infty$. In \cite[Problem 14.2]{Holmgren2017}, they posed the open question of proving the asymptotic normality of $N_{T}(T_{n}^{\psi})$ (up to renormalization) as $n \rightarrow  \infty$. While this work does not fully address this problem, our main result, Theorem \ref{Theo2}, represents a first step in that direction. We intend to investigate and resolve \cite[Problem 14.2]{Holmgren2017} in future work.

It is worth highlighting that asymptotic normality for fringe tree counts has been shown for a wide range of random tree models by many researchers. See, for example, \cite{Devroye1991}, \cite{Flajolet1997}, \cite{Devroye2002}, \cite{Chang2010}, \cite{Feng2010}, \cite{Fuchs2012}, \cite{Mohan2014}, \cite{Holmgren2015}, \cite{Janson2016}, \cite{Holmgren2017}, \cite{Svante2022}, \cite{Berzunza2023}. Thus, providing an answer to \cite[Problem 14.2]{Holmgren2017} is certainly of interest

Finally, motivated by \cite[Theorem 5.14]{Holmgren2017}, we note that more generally, one could consider a property $\mathcal{P}$ defined as any subset of $\mathbb{T}$. For example, $\mathcal{P} = \{T^{\prime} \in \mathbb{T}: T^{\prime} \simeq T  \}$. Then, by defining the characteristic $\varphi^{\mathcal{P}}(t) \coloneqq \mathbf{1}_{\{\mathcal{T}_{t} \in \mathcal{P} \}} \mathbf{1}_{[0, \infty)}(t)$, for $t \in \mathbb{R}$, $\mathcal{Z}_{t}^{\varphi^{\mathcal{P}}}$ counts the number of individuals (vertices) in $\mathcal{T}_{t}$ satisfying property $\mathcal{P}$. If $\mathcal{P}$ only considers trees of a given height $h \in \mathbb{N}_{0}$, then $\varphi^{\mathcal{P}}$ is an $h$-dependent descendant characteristic. Thus, our main result, Theorem \ref{Theo2}, might potentially apply, provided our assumptions are satisfied. For example, this general framework could be used to study the number of protected vertices in $\mathcal{T}_{t}$, defined as vertices that are neither a leaf nor the parent of a leaf. We refer to \cite[Section 10]{Holmgren2017} and the references therein for further discussion of this specific example.

%%%%%%%%%%%%%%%%%%%%%%%%%%%%%%%%%%%%%%%%%%%%%%%%%%%%%%%%%%%%%%%%%%%%%%%
\section{Auxiliary results} \label{Preliminaries}
%%%%%%%%%%%%%%%%%%%%%%%%%%%%%%%%%%%%%%%%%%%%%%%%%%%%%%%%%%%%%%%%%%%%%%%

In this section, we establish some essential auxiliary results. In Sections \ref{ExpSection} and \ref{BigMart}, we recall some known facts about the expectation of the CMJ-branching process and the so-called Biggins martingale. In Section \ref{CondExpect}, we gather some technical identities and estimations involving conditional expectations that will be used frequently in the subsequent proofs. Section \ref{ProofProposition} contains the proof of Proposition \ref{Proposition3} and as well as other key asymptotic results for the proof of our main result. Finally, in Section \ref{MoreAsymptotic}, we establish some crucial properties of the characteristic $\chi^{(\varphi, h)}$.

%%%%%%%%%%%%%%%%%%%%%%%%%%%%%%%%%%%%%%%%%%%%%%%%%%%%%%%%%%%%%%%%%%%%%%%
\subsection{The expectation of the general branching process}  \label{ExpSection}
%%%%%%%%%%%%%%%%%%%%%%%%%%%%%%%%%%%%%%%%%%%%%%%%%%%%%%%%%%%%%%%%%%%%%%%

This section gathers and presents facts from the literature on the first moment of general branching processes, counted with a (descendant or individual)  characteristic that satisfies suitable assumptions, and its associated asymptotic behaviour.

Recall that $\mu$ is the intensity measure of the birth point process $\xi$. Suppose that \ref{A1}-\ref{A4} holds and let $\nu^{(\infty)}$ be the renewal measure defined in \eqref{eq62}. Let $\varphi$ be a characteristic such that $|\varphi|$ satisfies \ref{C1}. Note that \cite[Theorem 3.1]{JagersN1984} can be extended to descendant characteristics satisfying the above, as discussed in the proof of \cite[Theorem 6.1]{JagersN1984} (see also \cite[Sections 4.1 and 4.2]{Iksanov2021} for a similar computation). Thus, for $t \in \mathbb{R}$,
\begin{align} \label{eq109}
m_{t}^{\varphi}= \mathbb{E}[\mathcal{Z}^{\varphi}_{t}] = \int_{[0,\infty)} \mathbb{E}[\varphi](t-x) \nu^{(\infty)}({\rm d}x) = \mathbb{E}[\varphi] \ast \nu^{(\infty)}(t),
\end{align}
\noindent  If assumptions \ref{A1}-\ref{A4} are satisfied and $t \in \mathbb{R} \mapsto \mathbb{E}[|\varphi|](t) e^{-\alpha t}$ is directly Riemann integrable, it follows that $m_{t}^{\varphi} < \infty$. 

Let $\alpha>0$ be the Malthusian parameter defined in \ref{A4}. Recall from \eqref{eq123} that $e^{-\alpha}\nu^{(\infty)}({\rm d}x) = \sum_{i=0}^{\infty} \mu_{\alpha}^{\ast i}({\rm d}x)$ is the renewal measure associated with the probability measure $\mu_{\alpha}({\rm d}x) =  e^{-\alpha x} \mu({\rm d}x)$. Assume that $\beta \coloneqq \int_{[0,\infty)} x e^{-\alpha x} \mu({\rm d}x) \in (0,\infty)$ (this hold whenever \ref{A1}-\ref{A7} are satisfied). Then, by the key renewal theorem,
\begin{align} \label{eq66}
\lim_{\substack{t \rightarrow \infty \\ t \in \mathbb{G}}} e^{-\alpha t} m_{t}^{\varphi} = \lim_{\substack{t \rightarrow \infty \\ t \in \mathbb{G}}} \int_{[0,\infty)} \mathbb{E}[\varphi](t-x) e^{-\alpha(t-x)} \sum_{i=0}^{\infty} \mu_{\alpha}^{\ast i}({\rm d}x) = \frac{1}{\beta} \int_{\mathbb{G}} \mathbb{E}[\varphi](x) e^{-\alpha x} \ell({\rm d} x);
\end{align}
\noindent see e.g., \cite[Theorem 4.2]{Athreya1978}, in the non-lattice case and \cite[Theorem 2.5.3]{Alsmeyer1991} in the lattice (and non-lattice) case. 

%%%%%%%%%%%%%%%%%%%%%%%%%%%%%%%%%%%%%%%%%%%%%%%%%%%%%%%%%%%%%%%%%%%%%%%
\subsection{Biggins' martingale}   \label{BigMart}
%%%%%%%%%%%%%%%%%%%%%%%%%%%%%%%%%%%%%%%%%%%%%%%%%%%%%%%%%%%%%%%%%%%%%%%

Suppose that \ref{A1}-\ref{A7} hold. In particular, \ref{A7} (see also \eqref{eq31}) implies that $\hat{\mu}(\theta) < \infty$ for all $\theta \geq \vartheta$, where $\vartheta \in (0, \alpha/2)$ is as in \ref{A7}. Nerman's martingale \eqref{eq16} relates to the so-called {\sl Biggins martingale} $(Z_{n}^{(\theta)})_{n \in \mathbb{N}_{0}}$, where
\begin{align} 
Z_{n}^{(\theta)} = (\hat{\mu}(\theta))^{-n}\sum_{|u|=n} e^{-\theta S(u)}, \quad \text{for} \quad n \in \mathbb{N}_{0}.
\end{align}
\noindent where the sum $\sum_{|u| = n}$ is over all vertices $u \in \mathcal{I}$ with height $|u| = n$. In particular, for $\theta \geq \vartheta$ and $n \in \mathbb{N}_{0}$,
\begin{align} \label{eq156}
\mathbb{E} \left[\sum_{|u|=n} e^{-\theta S(u)} \right] < \infty.
\end{align}
\noindent Since the Biggins martingale martingale is non-negative, it converges a.s.\ to a finite limit $Z^{(\theta)} \geq 0$ with $\mathbb{E}[Z^{(\theta)}] \leq 1$. Furthermore, \cite[Theorem 3.3]{Dimitris2000} establishes that if $\mu$ is concentrated on $(0,\infty)$, then $W = Z^{(\alpha)}$ a.s., where $W$ is the limit of Nerman's martingale. If \ref{A1}-\ref{A7} hold, then $\frac{\hat{\mu}(2\theta)}{(\hat{\mu}(\theta))^{2}}  < 1$, for $\theta \in [\vartheta, \alpha]$, because $\hat{\mu}(\theta) \geq 1$. Then, by \cite[Theorem 1]{Biggins1992} or \cite[Theorem 2.1]{Quansheng2000}, it holds that, for $\theta \in [\vartheta, \alpha]$, $Z_{n}^{(\theta)} \rightarrow Z^{(\theta)}$, as $n \rightarrow \infty$, in $L^{2}(\Omega, \mathcal{F}, \mathbb{P})$. In particular, $\mathbb{E}[Z^{(\theta)}]=1$ and ${\rm Var}(Z^{(\theta)})<\infty$. Moreover, for $\theta \in [\vartheta, \alpha]$ and $n \in \mathbb{N}_{0}$,
\begin{align} \label{eq152}
\mathbb{E} \left[ \left(\sum_{|u|=n} e^{-\theta S(u)}\right)^{2}  \right] < \infty.
\end{align}

%%%%%%%%%%%%%%%%%%%%%%%%%%%%%%%%%%%%%%%%%%%%%%%%%%%%%%%%%%%%%%%%%%%%%%%
\subsection{Conditional expectations}  \label{CondExpect}
%%%%%%%%%%%%%%%%%%%%%%%%%%%%%%%%%%%%%%%%%%%%%%%%%%%%%%%%%%%%%%%%%%%%%%%

In this section, we present a series of preliminary computations involving conditional expectations, which will be referenced and applied in subsequent sections of this work. For $n \in\mathbb{N}_{0}$, define 
\begin{align} \label{eq97}
\mathcal{I}^{(n)} \coloneqq \{ u \in \mathcal{I}: |u| \leq n \},
\end{align}
\noindent that is, $\mathcal{I}^{(n)}$ contains all vertices up to generation $n$. We let $\mathcal{I}^{(-1)}$ be the empty set $\emptyset$. Define also the sub-$\sigma$-algebras $(\mathcal{A}^{(n)})_{n \geq -1}$ of  $\mathcal{F}$ by letting $\mathcal{A}^{(-1)}$ be the trivial $\sigma$-algebra and for $n \in \mathbb{N}_{0}$,
\begin{align} \label{eq135} 
\mathcal{A}^{(n)} \coloneqq \sigma(\pi_{u}: u \in \mathcal{I}^{(n)}),
\end{align}
\noindent where $\pi_{u}: \Omega \rightarrow  \Omega_{u}$ is the projection onto the life space of individual $u \in \mathcal{I}$.  Note that, for $n \in \mathbb{N}_{0}$, $\mathcal{A}^{(n)}$ is exactly the $\sigma$-algebra defined in \eqref{eq83}. 

\begin{lemma} \label{lemma3}
Let $\varphi$ be a characteristic satisfying \ref{C1}-\ref{C2}. Then, for every  $t \in \mathbb{R}$ and $u \in \mathcal{I}$, $\varphi_{u}(t-S(u)) \in L^{2}(\Omega, \mathcal{F}, \mathbb{P})$. 
\end{lemma}

\begin{proof}
Note that $S(u)$ is $\mathcal{A}^{(|u|-1)}$-measurable and $\varphi_{u}$ is independent of $\mathcal{A}^{(|u|-1)}$. Next, recall that the conditional expectation of a non-negative random variable is well-defined, even if the variable itself is not integrable (see e.g., \cite[Definition 1, Section 7, Chapter II]{Shiryaev1996}). Then, 
\begin{align}  \label{eq30}
\mathbb{E}[ (\varphi_{u}(t-S(u)))^{2} \mid \mathcal{A}^{(|u|-1)}] \stackrel{a.s.}{=} \mathbb{E}[\varphi^{2}](t-S(u)).
\end{align}
\noindent  This can be formally proven using a standard approximation argument, which involves proving the identity for non-negative simple measurable functions and then extending it to non-negative measurable functions via the monotone convergence theorem (see for e.g., the argument used in the proof of \cite[(10.17), p.\ 350]{Resnick2014}). By \ref{C2}, $\varphi(t) \in  L^{2}(\Omega, \mathcal{F}, \mathbb{P})$, for $t \in \mathbb{R}$. Then, by \ref{C1}-\ref{C2}, there exists a constant $C>0$ such that 
\begin{align} \label{eq35}
\mathbb{E}[(\varphi(t))^{2}] =  {\rm Var}(\varphi(t)) + (\mathbb{E}[\varphi(t)])^{2} \leq C(e^{\alpha t} + e^{2 \alpha t}).
\end{align}
\noindent  Finally, \eqref{eq30}, \eqref{eq35} and properties of the conditional expectation (see e.g., \cite[4.\ in Section 7, Chapter II]{Shiryaev1996}) imply that there exists a constant $C>0$ such that 
\begin{align} 
\mathbb{E}[ (\varphi_{u}(t-S(u)))^{2}] \leq C\mathbb{E}[e^{ \alpha \max(t-S(u), 2(t-S(u)))} ] < \infty.
\end{align}
\noindent This concludes our proof.
\end{proof}

For each $u \in \mathcal{I}$ and $n \in \mathbb{N}_{0}$, define 
\begin{align} \label{eq134} 
\mathcal{I}^{(n)}_{u} \coloneqq u\mathcal{I}^{(n)},
\end{align}
\noindent that is, $\mathcal{I}^{(n)}_{u}$ contains all descendants (including $u$ itself) at distance $n$ from $u$. (Here, ``distance'' refers to the graph distance within the tree $\mathcal{I}$).  We let $\mathcal{I}_{u}^{(-1)}$ be the empty set $\emptyset$.  Define also the sub-$\sigma$-algebras $(\mathcal{A}^{(n)}_{u})_{n \geq -1}$ of $\mathcal{F}$ by letting $\mathcal{A}^{(-1)}_{u}$ be the trivial $\sigma$-algebra and for $n \in \mathbb{N}_{0}$,
\begin{align}  \label{eq4}
\mathcal{A}^{(n)}_{u} \coloneqq \sigma(\pi_{u}: u \in \mathcal{I}^{(n)}_{u}).
\end{align}

Let $\varphi^{(k)}$, for $k \in \mathbb{N}_{0}$, be the characteristic defined in \eqref{eq27}. Recall also that $\theta_{u}$ is the shift operator, that is, $\theta_{u}\left((\omega_{v})_{v \in \mathcal{I}}\right)=(\omega_{u v})_{v \in \mathcal{I}}$, for $(\omega_{v})_{v \in \mathcal{I}} \in \Omega$. Then, $\varphi^{(k)}_{u}=\varphi^{(k)}\circ \theta_{u}$. 

\begin{lemma} \label{lemma11}
Let $\varphi$ be a characteristic satisfying \ref{C1}. Then, for every  $t \in \mathbb{R}$, $u \in \mathcal{I}$ and $k \in \mathbb{N}_{0}$,
\begin{align}
\varphi_{u}^{(k)}(t) \stackrel{a.s.}{=} \mathbb{E}[\varphi_{u}(t) \mid \mathcal{A}_{u}^{(k-1)}].
\end{align}
\end{lemma}

\begin{proof}
For $u \in \mathcal{I}$ and $n \in \mathbb{N}_{0}$, let $\pi_{u}^{(n)}: \Omega \rightarrow \bigtimes_{v \in \mathcal{I}_{u}^{(n)}} \Omega_{v}$ be the projection map. In particular, $\mathcal{A}_{u}^{(n)} = \sigma(\pi_{u}^{(n)})$. Note also that, by \eqref{eq135}, $\mathcal{A}_{\varnothing}^{(n)} = \mathcal{A}^{(n)}$. Since the population space $(\Omega, \mathcal{F}, \mathbb{P})$ in \eqref{eq100} is defined as the product of identical probability spaces, we have that for $A \in \mathcal{A}_{u}^{(n)}$ there exists $A^{\prime} \in \mathcal{A}^{(n)}$ such that $\theta^{-1}_{u}(A^{\prime}) = A$. 

Assume that $u \neq \varnothing$ and $k > 0$; otherwise, the claim holds trivially. Let  $A \in \mathcal{A}_{u}^{(k-1)}$ and  $A^{\prime} \in \mathcal{A}^{(k-1)}$ such that $\theta^{-1}_{u}(A^{\prime}) = A$. Observe that
\begin{align}
\mathbb{E}[\mathbf{1}_{A} \varphi_{u}^{(k)}(t)] = \int_{\Omega} \mathbf{1}_{A}(\omega) \varphi_{u}^{(k)}(\omega, t) \mathbb{P}({\rm d} \omega) = \int_{\Omega} \mathbf{1}_{\theta_{u}^{-1}(A^{\prime})}(\omega) \varphi^{(k)} \circ \theta_{u}(\omega, t) \mathbb{P}({\rm d} \omega).
\end{align}
\noindent By substitution (see e.g., \cite[Lemma 1.22]{Kallenberg2002}),
\begin{align}
\mathbb{E}[\mathbf{1}_{A} \varphi_{u}^{(k)}(t)]  = \int_{\Omega} \mathbf{1}_{A^{\prime}}(\omega) \varphi^{(k)}(\omega, t) \mathbb{P} \circ \theta_{u}^{-1}({\rm d} \omega). 
\end{align}
\noindent Since the population space $(\Omega, \mathcal{F}, \mathbb{P})$ is defined as the product of identical probability spaces, we have that shifts preserve the probability measure, that is, $\mathbb{P} \circ \theta_{u}^{-1} = \mathbb{P}$. Hence, by \eqref{eq27},
\begin{align}
\mathbb{E}[\mathbf{1}_{A} \varphi_{u}^{(k)}(t)]  = \int_{\Omega} \mathbf{1}_{A^{\prime}}(\omega) \varphi^{(k)}(\omega, t) \mathbb{P}({\rm d} \omega) = \mathbb{E}[\mathbf{1}_{A^{\prime}} \varphi^{(k)}(t) ] = \mathbb{E}[\mathbf{1}_{A^{\prime}} \mathbb{E}[\varphi(t) \mid \mathcal{A}^{(k-1)}] ].
\end{align}
\noindent Since $A^{\prime} \in \mathcal{A}^{(k-1)}$, we have that
\begin{align}
\mathbb{E}[\mathbf{1}_{A} \varphi_{u}^{(k)}(t)]  = \mathbb{E}[\mathbf{1}_{A^{\prime}} \varphi(t) ] = \int_{\Omega} \mathbf{1}_{A^{\prime}}(\omega) \varphi(\omega, t) \mathbb{P} ({\rm d} \omega). 
\end{align}
\noindent Moreover, recall that $\mathbb{P} \circ \theta_{u}^{-1} = \mathbb{P}$. Then,  by substitution (see e.g., \cite[Lemma 1.22]{Kallenberg2002}),
\begin{align}
\mathbb{E}[\mathbf{1}_{A} \varphi_{u}^{(k)}(t)]  = \int_{\Omega} \mathbf{1}_{A} (\omega) \varphi \circ \theta_{u}(\omega, t) \mathbb{P} ({\rm d} \omega) = \mathbb{E}[\mathbf{1}_{A}\mathbb{E}[\varphi_{u}(t) \mid \mathcal{A}_{u}^{(k-1)}]]. 
\end{align}
\noindent This proves our claim.
\end{proof}

Set 
\begin{align}
\mathcal{P}_{u} \coloneqq \{ v\in \mathcal{I}: v \prec u \}.  
\end{align}
\noindent If $u =\varnothing$, we define $\mathcal{P}_{\varnothing}$ as the empty set $\emptyset$. Note that $\mathcal{P}_{u}$ consists of all strict ancestors of $u$. For $n \geq -1$, define the sub-$\sigma$-algebra of $\mathcal{F}$
\begin{align} \label{eq11}
\mathcal{D}_{u}^{(n)} \coloneqq \sigma(\pi_{v}: v\in \mathcal{P}_{u} \cup \mathcal{I}_{u}^{(n)}).
\end{align}
\noindent If $u =\varnothing$ and $n=-1$, we define $\mathcal{D}_{\varnothing}^{(-1)}$ as the trivial $\sigma$-algebra. 
\begin{lemma} \label{lemma9}
Let $\varphi$ be a characteristic satisfying \ref{C1}-\ref{C2}. For $u \in \mathcal{I}$ and $n \geq -1$, let $\mathcal{J}_{u} \subseteq \mathcal{I}$ such that $\mathcal{P}_{u} \subseteq \mathcal{J}_{u}$ and  $\mathcal{J}_{u} \cap u\mathcal{I} = \emptyset$. Define the sub-$\sigma$-algebra of $\mathcal{F}$
\begin{align} \label{eq42}
\mathcal{C}_{u}^{(n)} \coloneqq \sigma(\pi_{v}: v\in \mathcal{J}_{u} \cup \mathcal{I}_{u}^{(n)}).
\end{align}
\noindent If $u =\varnothing$ and $n=-1$, we define $\mathcal{C}_{\varnothing}^{(-1)}$ as the trivial $\sigma$-algebra. Then, for $t \in \mathbb{R}$, $u \in \mathcal{I}$ and $k \in \mathbb{N}_{0}$, 
\begin{align} \label{eq43} 
\varphi_{u}^{(k)}(t-S(u)) \stackrel{a.s.}{=} \mathbb{E}[\varphi_{u}(t-S(u)) \mid \mathcal{C}_{u}^{(k-1)}] \stackrel{a.s.}{=}  \mathbb{E}[\varphi_{u}(t-S(u)) \mid \mathcal{D}_{u}^{(k-1)}]. 
\end{align}
\end{lemma}

\begin{proof}
By Lemma \ref{lemma3}, $\varphi_{u}(t-S(u)) \in L^{2}(\Omega, \mathcal{F}, \mathbb{P})$, for every $u \in \mathcal{I}$. Thus, the conditional expectations in \eqref{eq43} are well-defined.

Since $\mathcal{P}_{u} \subseteq \mathcal{J}_{u}$, we have that $\mathcal{D}_{u}^{(k-1)} \subseteq \mathcal{C}_{u}^{(k-1)}$. Then, the second equality in \eqref{eq43} is a direct consequence of the first equality and properties of conditional expectation. Thus, it remains to prove the  first equality in \eqref{eq43}. 

Assume that $u \neq \varnothing$ and $k >0$; otherwise, the claim holds trivially. Recall that for any sub-$\sigma$-algebras $\mathcal{B}_{1}, \mathcal{B}_{2}, \mathcal{B}_{3}$ of $\mathcal{F}$, we say that $\mathcal{B}_{1}$ and $\mathcal{B}_{2}$ are conditionally independent given $\mathcal{B}_{3}$ (see for e.g., \cite[page 109]{Kallenberg2002}), denoted by $\mathcal{B}_{1} \indep_{\mathcal{B}_{3}} \mathcal{B}_{2}$,  if 
\begin{align}
\mathbb{P}(B_{1} \cap B_{2}  \mid \mathcal{B}_{3}) \stackrel{a.s.}{=}  \mathbb{P}(B_{1} \mid \mathcal{B}_{3}) \mathbb{P}(B_{2} \mid \mathcal{B}_{3}), \quad \text{for all $B_{1} \in \mathcal{B}_{1}$ and $B_{2} \in \mathcal{B}_{2}$}.
\end{align}

Let $\mathcal{A}_{u}^{(\infty)} \coloneqq \sigma(\pi_{v}: v \in u\mathcal{I})$. Observe that $S(u)$ is $\mathcal{C}^{(k-1)}_{u}$-measurable, while $\varphi_{u}$ is $\mathcal{B}(\mathbb{R})\times \mathcal{A}_{u}^{(\infty)}$-measurable. Define
\begin{align} \label{eq137}
g_{\varphi_{u}}(t) \coloneqq \mathbb{E}[\varphi_{u}(t) \mid \mathcal{C}^{(k-1)}_{u}], \quad \text{for} \quad t \in \mathbb{R}. 
\end{align}
\noindent Clearly, $\mathcal{A}^{(\infty)}_{u} \indep_{\mathcal{C}^{( k-1)}_{u}} \mathcal{C}^{(k-1)}_{u}$. Then, by properties of the conditional expectation,
\begin{align} \label{eq173}
\mathbb{E}[\varphi_{u}(t-S(u)) \mid \mathcal{C}^{(k-1)}_{u}]  \stackrel{a.s.}{=}  \mathbb{E}[g_{\varphi_{u}}(t-S(u)) \mid \mathcal{C}^{(k-1)}_{u}] \stackrel{a.s.}{=}  g_{\varphi_{u}}(t-S(u)). 
\end{align}
\noindent To complete the proof of the first equality in \eqref{eq43}, it remains to prove that $g_{\varphi_{u}}(t) \stackrel{a.s.}{=}  \varphi_{u}^{(k)}(t)$. 

Define the $\sigma$-algebra $\mathcal{B}_{u}^{(k-1)} \coloneqq \sigma(\pi_{v}: v \in \mathcal{J}_{u})$. Note that, by \eqref{eq4} and \eqref{eq42},  
$\mathcal{C}^{(k-1)}_{u} =  \mathcal{A}^{(k-1)}_{u} \vee \mathcal{B}_{u}^{(k-1)}$. Here $\mathcal{A}^{(k-1)}_{u} \vee \mathcal{B}_{u}^{(k-1)}$ denotes the smallest $\sigma$-algebra containing $\mathcal{A}^{(k-1)}_{u}$ and $\mathcal{B}_{u}^{(k-1)}$. By construction (see Section \ref{CMJclassic}), we also have that $\mathcal{A}^{(k-1)}_{u}$ and $\mathcal{B}_{u}^{(k-1)}$ are independent. Moreover, $\mathcal{B}_{u}^{(k-1)} \indep_{\mathcal{A}_{u}^{(k-1)}} \mathcal{A}_{u}^{(\infty)}$. Then, by \cite[Proposition 6.6]{Kallenberg2002}, 
\begin{align}  \label{eq138}
\mathbb{E}[\varphi_{u}(t) \mid \mathcal{C}^{(k-1)}_{u}] = \mathbb{E}[\varphi_{u}(t) \mid \mathcal{A}^{(k-1)}_{u} \vee \mathcal{B}_{u}^{(k-1)}] \stackrel{a.s.}{=}  \mathbb{E}[\varphi_{u}(t) \mid\mathcal{A}^{(k-1)}_{u}].
\end{align}
\noindent Therefore, \eqref{eq137}, \eqref{eq138} and Lemma \ref{lemma11} imply $g_{\varphi_{u}}(t) \stackrel{a.s.}{=}  \varphi_{u}^{(k)}(t)$. This finishes the proof of the first equality in \eqref{eq43}. 
\end{proof}

The following corollaries are a consequence of Lemmas \ref{lemma3}, \ref{lemma11} and \ref{lemma9}. 

\begin{corollary} \label{corollary1}
Let $\varphi$ be a characteristic satisfying \ref{C1}-\ref{C2}. Then, for $t \in \mathbb{R}$ and $k\in \mathbb{N}_{0}$, we have the following:
\begin{enumerate}[label=(\textbf{\roman*})]
\item For every $u \in \mathcal{I}$, $\varphi_{u}^{(k)}(t-S(u)) \stackrel{a.s.}{=} \mathbb{E}[\varphi_{u}(t-S(u)) \mid \mathcal{A}^{(|u|+ k-1)}]$. \label{corollary1Pro6}

\item For every $u,v \in \mathcal{I}$ such that $u \neq v$ and $|u| = |v|$,  \label{corollary1Pro6b}
\begin{align}
\varphi_{u}^{(k)}(t-S(u)) \varphi_{v}^{(k)}(t-S(v)) \stackrel{a.s.}{=} \mathbb{E}[\varphi_{u}(t-S(u)) \varphi_{v}(t-S(v)) \mid \mathcal{A}^{(|u|+ k-1)}].
\end{align}

\item For every $u \in \mathcal{I}$,  $\varphi^{(k)}_{u}(t-S(u)) \in L^{2}(\Omega, \mathcal{F}, \mathbb{P})$. \label{corollary1Pro1}

\item For every $u \in \mathcal{I}$, there exists a constant $C>0$ such that, almost surely, \label{corollary1Pro2}
\begin{align}
\mathbb{E}[ (\varphi_{u}^{(k)}(t-S(u)))^{2} \mid \mathcal{A}^{(|u|-1)}] \leq C
e^{\alpha \max(t-S(u), 2(t-S(u)))}. 
\end{align} 

\item For every $u,v \in \mathcal{I}$ such that $u \neq v$ and $|u| = |v|$ there exists a constant $C>0$ such that, almost surely, \label{corollary1Pro3}
\begin{align}
& \mathbb{E}[ |\varphi_{u}^{(k)}(t-S(u))| |\varphi_{v}^{(k)}(t-S(v))| \mid \mathcal{A}^{(|u|-1)}]  \leq C \prod_{w=u,v}e^{ \frac{\alpha}{2} \max(t-S(w), 2(t-S(w)))}.
\end{align} 

\item For every $u \in \mathcal{I}$, almost surely, \label{corollary1Pro4}
\begin{align}
\mathbb{E}[ (\varphi_{u}^{(k)}(t-S(u)) - \varphi_{u}^{(0)}(t-S(u)))^{2} \mid \mathcal{A}^{(|u|-1)} ] \leq 2 {\rm Var}[\varphi](t-S(u)). 
\end{align}

\item For every $u,v \in \mathcal{I}$ such that $u \neq v$ and $|u| = |v|$, \label{corollary1Pro5}
\begin{align}
\mathbb{E}[ (\varphi_{u}^{(k)}(t-S(u)) - \varphi_{u}^{(0)}(t-S(u))) (\varphi_{v}^{(k)}(t-S(v)) - \varphi_{v}^{(0)}(t-S(v))) \mid \mathcal{A}^{(|u|-1)} ] \stackrel{a.s.}{=} 0. 
\end{align}
\end{enumerate}
\end{corollary}

\begin{proof}
First, we prove \ref{corollary1Pro6}. Note that \ref{corollary1Pro6} follows from Lemma \ref{lemma9}. Specifically, by setting $\mathcal{J}_{u} = \mathcal{I}^{(|u|+ k-1)} \setminus \mathcal{I}^{(k-1)}_{u}$ and thus, $\mathcal{C}^{(k-1)}_{u} = \mathcal{A}^{(|u|+ k-1)}$ (observe that $\mathcal{I}^{(k-1)}_{u} \subseteq \mathcal{I}^{(|u|+ k-1)}$ and $\mathcal{J}_{u} \cap u\mathcal{I} = \emptyset$). 

Next, we prove \ref{corollary1Pro6b}. By \ref{corollary1Pro6} and Lemma \ref{lemma9}, we have that
\begin{align}
\varphi_{u}^{(k)}(t-S(u)) \varphi_{v}^{(k)}(t-S(v)) \stackrel{a.s.}{=} \mathbb{E}[\varphi_{u}(t-S(u)) \mid \mathcal{A}^{(|u|+ k-1)}] \mathbb{E}[\varphi_{u}(t-S(u)) \mid \mathcal{D}^{(k-1)}_{v}].
\end{align}
\noindent Note that, by \eqref{eq135} and \eqref{eq11}, $\mathcal{D}^{(k-1)}_{v} \subseteq \mathcal{A}^{(|u|+ k-1)}$ because $|u|=|v|$ and $u \neq v$. Then, by Lemma \ref{lemma9},
\begin{align}
\varphi_{u}^{(k)}(t-S(u)) \varphi_{v}^{(k)}(t-S(v)) & \stackrel{a.s.}{=} \mathbb{E}[\varphi_{u}(t-S(u))  \mathbb{E}[\varphi_{u}(t-S(u)) \mid \mathcal{D}^{(k-1)}_{v}] \mid \mathcal{A}^{(|u|+ k-1)}] \nonumber \\
& \stackrel{a.s.}{=} \mathbb{E}[\varphi_{u}(t-S(u))  \varphi_{v}^{(k)}(t-S(v)) \mid \mathcal{A}^{(|u|+ k-1)}].
\end{align}
\noindent Set $\mathcal{J}_{v} = (\mathcal{I}^{(|u|+ k-1)} \cup u\mathcal{I}) \setminus \mathcal{I}_{v}^{(k-1)}$. Observe that $\mathcal{I}^{(k-1)}_{v} \subseteq \mathcal{I}^{(|u|+ k-1)} \cup u\mathcal{I}$ and $\mathcal{J}_{v} \cap v\mathcal{I} = \emptyset$. Then, Lemma \ref{lemma9} implies that
\begin{align}
\varphi_{u}^{(k)}(t-S(u)) \varphi_{v}^{(k)}(t-S(v)) & \stackrel{a.s.}{=} 
\mathbb{E}[\varphi_{u}(t-S(u))  \mathbb{E}[\varphi_{v}(t-S(v)) \mid \mathcal{C}_{v}^{(k-1)}] \mid \mathcal{A}^{(|u|+ k-1)}].
\end{align}
\noindent Observe that $\varphi_{u}(t-S(u))$ is $\mathcal{C}_{v}^{(k-1)}$-measurable. Hence
\begin{align} \label{eq174}
\varphi_{u}^{(k)}(t-S(u)) \varphi_{v}^{(k)}(t-S(v)) & \stackrel{a.s.}{=} 
\mathbb{E}[ \mathbb{E}[ \varphi_{u}(t-S(u)) \varphi_{v}(t-S(v)) \mid \mathcal{C}_{v}^{(k-1)}] \mid \mathcal{A}^{(|u|+ k-1)}].
\end{align}
\noindent Finally, by \eqref{eq135} and \eqref{eq42}, $\mathcal{A}^{(|u|+ k-1)} \subseteq \mathcal{C}_{v}^{(k-1)}$. Therefore, \ref{corollary1Pro6b} follows from \eqref{eq174}.

We prove \ref{corollary1Pro1}. By \ref{corollary1Pro6} and Jensen’s inequality, we that, almost surely,
\begin{align} \label{eq44}
(\varphi_{u}^{(k)}(t-S(u)))^{2} \leq \mathbb{E}[(\varphi_{u}(t-S(u)))^{2} \mid \mathcal{A}^{(|u|+ k-1)}].
\end{align}
\noindent Thus, \ref{corollary1Pro1} follows from Lemma \ref{lemma3} and the above inequality.

Next, we prove \ref{corollary1Pro2}. Note that $S(u)$ is $\mathcal{A}^{(|u|-1)}$-measurable and  $\varphi_{u}$ is independent of $\mathcal{A}^{(|u|-1)}$. Moreover, $\mathcal{A}^{(|u|-1)}\subseteq\mathcal{A}^{(|u|+k-1)}$. Then, by \eqref{eq44}, we have that, almost surely,
\begin{align} \label{eq37}
\mathbb{E}[ (\varphi_{u}^{(k)}(t-S(u)) )^{2} \mid \mathcal{A}^{(|u|-1)}] \leq \mathbb{E}[(\varphi_{u})^{2}](t-S(u)).
\end{align}
\noindent Thus, \ref{corollary1Pro2} follows from \eqref{eq35} and \eqref{eq37}. 

We prove \ref{corollary1Pro3}. Consider $u,v \in \mathcal{I}$ such that $u \neq v$ and $|u| = |v|$. Then, by \ref{corollary1Pro6}, \ref{corollary1Pro6b} and the triangle inequality, we that, almost surely,
\begin{align} \label{eq47}
|\varphi_{u}^{(k)}(t-S(u))| |\varphi_{v}^{(k)}(t-S(v))| & \leq \mathbb{E}[|\varphi_{u}(t-S(u)) \varphi_{v}(t-S(v))| \mid \mathcal{A}^{(|u|+ k-1)}].
\end{align}
\noindent Again, observe that $S(u)$ and $S(v)$ are $\mathcal{A}^{(|u|-1)}$-measurable, while  $\varphi_{u}$ and $\varphi_{v}$ are independent of $\mathcal{A}^{(|u|-1)}$.  Moreover, $\mathcal{A}^{(|u|-1)}\subseteq\mathcal{A}^{(|u|+k-1)}$. Then, by \eqref{eq47} and Jensen’s inequality, we have that, almost surely,
\begin{align} \label{eq46}
(\mathbb{E}[ |\varphi_{u}^{(k)}(t-S(u))| |\varphi_{v}^{(k)}(t-S(v))| \mid \mathcal{A}^{(|u|-1)}])^{2} & \leq \mathbb{E}[(\varphi_{u}(t-S(u)))^{2} (\varphi_{v}(t-S(v)))^{2} \mid \mathcal{A}^{(|u|-1)}]  \nonumber \\
&  = \mathbb{E}[(\varphi_{u})^{2}](t-S(u)) \mathbb{E}[(\varphi_{v})^{2}](t-S(v)).
\end{align}
\noindent Thus, \ref{corollary1Pro3} follows from \eqref{eq35} and \eqref{eq46}.

We prove \ref{corollary1Pro4}. Note that, by \ref{corollary1Pro6},  $\varphi_{u}^{(0)}(t-S(u))\stackrel{a.s.}{=} \mathbb{E}[\varphi_{u}(t-S(u)) \mid \mathcal{A}^{(|u|-1)}]$. Then, by \ref{corollary1Pro6} and Jensen’s inequality, we that, almost surely,
\begin{align} \label{eq172}
(\varphi_{u}^{(k)}(t-S(u)) - \varphi_{u}^{(0)}(t-S(u)))^{2} \leq \mathbb{E}[(\varphi_{u}(t-S(u)) - \mathbb{E}[\varphi_{u}(t-S(u)) \mid \mathcal{A}^{(|u|-1)}])^{2} \mid \mathcal{A}^{(|u|+ k-1)}].
\end{align}
\noindent Recall that $S(u)$ is $\mathcal{A}^{(|u|-1)}$-measurable, whereas $\varphi_{u}$ is independent of $\mathcal{A}^{(|u|-1)}$.  Moreover, $\mathcal{A}^{(|u|-1)}\subseteq\mathcal{A}^{(|u|+k-1)}$. Then, by \eqref{eq172} and Jensen’s inequality, we have that, almost surely,
\begin{align} \label{eq19}
& \mathbb{E}[ (\varphi_{u}^{(k)}(t-S(u)) - \varphi_{u}^{(0)}(t-S(u)))^{2} \mid \mathcal{A}^{(|u|-1)} ] \nonumber \\
& \quad \quad \quad  \leq  \mathbb{E}[ (\varphi_{u})^{2}](t-S(u)) - 2(\mathbb{E}[\varphi_{u}](t-S(u)))^{2} + \mathbb{E}[ (\varphi_{u})^{2}](t-S(u)) \nonumber \\
& \quad \quad \quad = 2( \mathbb{E}[ \varphi^{2}](t-S(u)) - (\mathbb{E}[\varphi](t-S(u)))^{2}),
\end{align}
\noindent  which implies \ref{corollary1Pro4}. 

Finally, we prove \ref{corollary1Pro5}. As observed in the proof of  \ref{corollary1Pro3}, for $u,v \in \mathcal{I}$ such that $u \neq v$ and $|u| = |v|$, $S(u)$ and $S(v)$ are $\mathcal{A}^{(|u|-1)}$-measurable, while  $\varphi_{u}$ and $\varphi_{v}$ are independent of $\mathcal{A}^{(|u|-1)}$. Moreover, $\varphi_{u}$ and $\varphi_{v}$ are independent. The preceding observations, together with \ref{corollary1Pro6}, \ref{corollary1Pro6b} and  properties of the conditional expectation, proves \ref{corollary1Pro5}. 
\end{proof}

Let $(u_{n})_{n \in \mathbb{N}}$ be any admissible ordering of $\mathcal{I}$. Recall that $u_{1} = \varnothing$. For each $n \in \mathbb{N}$ and $k \in \mathbb{N}_{0}$, we set
\begin{align} \label{eq126} 
\mathcal{I}_{n} \coloneqq \left\{u_{1}, \ldots, u_{n}\right\} \quad \text{and} \quad  \mathcal{I}^{(k)}_{n} \coloneqq \bigcup_{i=1}^{n} \mathcal{I}^{(k)}_{u_{i}},
\end{align}
\noindent For each $k \in  \mathbb{N}_{0}$, define also a collection of sub-$\sigma$-algebras $(\mathcal{G}_{n}^{(k)})_{n \in  \mathbb{N}_{0}}$ of $\mathcal{F}$ by letting $\mathcal{G}_{0}^{(k)}$ be the trivial $\sigma$-algebra and for $n \in \mathbb{N}$,
\begin{align} \label{eq2}
\mathcal{G}_{n}^{(k)} & \coloneqq \sigma(\pi_{u}: u \in \mathcal{I}^{(k)}_{n})
= \sigma( \pi_{v}: \exists \, \, u \in \mathcal{I}_{n} \, \, \text{such that} \, \, v \in u\mathcal{I} \, \, \text{and} \, \, |v| - |u| \leq k ).
\end{align}

\begin{corollary} \label{LemmaCondExpII} 
Let $(u_{n})_{n \in \mathbb{N}}$ be any admissible ordering of $\mathcal{I}$ and let $\varphi$ be an $h$-dependent descendant characteristic, for some $h \in \mathbb{N}_{0}$, satisfying \ref{C1}-\ref{C2}. Then, for $t \in \mathbb{R}$ and $k=0, \dots, h+1$, 
\begin{align} \label{eq39} 
\mathbb{E}[\varphi_{u_{1}}(t-S(u_{1})) \mid \mathcal{G}^{(k)}_{0}] \stackrel{a.s.}{=}  \varphi^{(0)}_{u_{1}}(t-S(u_{1})),
\end{align} 
\noindent while for $i \geq 2$,
\begin{align} \label{eq38}
\mathbb{E}[\varphi_{u_{i}}(t-S(u_{i})) \mid \mathcal{G}^{(k)}_{i-1}] \stackrel{a.s.}{=}  \varphi^{(k)}_{u_{i}}(t-S(u_{i})).
\end{align}  
\end{corollary}

\begin{proof}
Note that \eqref{eq39} directly follows from Lemma \ref{lemma9} because $\mathcal{G}^{(k)}_{0} = \mathcal{C}^{(-1)}_{\varnothing}$. On the other hand, \eqref{eq38} follows also from Lemma \ref{lemma9}. Specifically, by setting $\mathcal{J}_{u_{i}} = \mathcal{I}^{(k)}_{i-1} \setminus \mathcal{I}^{(k-1)}_{u_{i}}$ and thus, $\mathcal{C}^{(k-1)}_{u_{i}} = \mathcal{G}^{(k)}_{i-1}$ (observe that $\mathcal{I}^{(k-1)}_{u_{i}} \subseteq \mathcal{I}^{(k)}_{i-1}$ and $\mathcal{J}_{u_{i}} \cap u_{i}\mathcal{I} = \emptyset$).
\end{proof}

\begin{corollary} \label{corollary2}
Let $(u_{n})_{n \in \mathbb{N}}$ be any admissible ordering of $\mathcal{I}$ and let $\varphi$ be an $h$-dependent descendant characteristic, for some $h \in \mathbb{N}_{0}$, satisfying \ref{C1}-\ref{C2}. Then, for every fixed $t \in \mathbb{R}$, we have that, for $i \geq 2$ and $u \in \mathcal{I}$ such that $u_{i} \preceq u$ and $|u|-|u_{i}|\leq h$,
\begin{align}
\mathbb{E}[\varphi_{u}^{(h+k+|u_{i}| -|u|)}(t-S(u)) \mid \mathcal{G}_{i-1}^{(h)} ] \stackrel{a.s.}{=}  \varphi_{u}^{(h+|u_{i}| -|u|)}(t-S(u)), \quad  \text{for} \quad k=0,1.
\end{align}
\end{corollary}

\begin{proof}
Note that, by \eqref{eq134} and \eqref{eq126}, $\mathcal{I}^{(h+|u_{i}| -|u|-1)}_{u} \subseteq \mathcal{I}^{(h)}_{i-1}$. Set $\mathcal{J}_{u} = \mathcal{I}^{(h)}_{i-1} \setminus \mathcal{I}^{(h+|u_{i}| -|u|-1)}_{u}$ and observe that $\mathcal{J}_{u} \cap u\mathcal{I} = \emptyset$. Then, by Lemma \ref{lemma9},
\begin{align}
\mathbb{E}[\varphi_{u}^{(h+k+|u_{i}| -|u|)}(t-S(u)) \mid \mathcal{G}_{i-1}^{(h)} ] \stackrel{a.s.}{=} \mathbb{E}[ \mathbb{E}[\varphi_{u}(t-S(u)) \mid \mathcal{C}_{u}^{(h+k+|u_{i}| -|u|-1)} ] \mid \mathcal{G}_{i-1}^{(h)} ].
\end{align}
\noindent Since $\mathcal{I}^{(h)}_{i-1} \subseteq \mathcal{J}_{u} \cup \mathcal{I}^{(h+k+|u_{i}| -|u|-1)}_{u}$, it follows from \eqref{eq42} and \eqref{eq2} that $\mathcal{G}_{i-1}^{(h)} \subseteq \mathcal{C}_{u}^{(h+k+|u_{i}| -|u|-1)}$. Thus, 
\begin{align} \label{eq1}
\mathbb{E}[\varphi_{u}^{(h+k+|u_{i}| -|u|)}(t-S(u)) \mid \mathcal{G}_{i-1}^{(h)} ] \stackrel{a.s.}{=} \mathbb{E}[ \varphi_{u}(t-S(u))  \mid \mathcal{G}_{i-1}^{(h)} ].
\end{align}
\noindent Observe that $\mathcal{G}_{i-1}^{(h)} = \mathcal{C}_{u}^{(h+|u_{i}| -|u|-1)}$. Then, our claim follows from \eqref{eq1} and Lemma \ref{lemma9}. 
\end{proof}

We will conclude this section with a useful remark. 

\begin{remark} \label{Remark7}
Let $\varphi$ be a characteristic satisfying \ref{C1}-\ref{C2}. For $t \in \mathbb{R}$, $u \in \mathcal{I}$ and $k \in \mathbb{N}_{0}$, it will be convenient to identify $\varphi_{u}^{(k)}(t-S(u))$ with its $\mathcal{D}_{u}^{(k-1)}$-measurable version, $\mathbb{E}[\varphi_{u}(t-S(u)) \mid \mathcal{D}_{u}^{(k-1)}]$ (recall Lemma \ref{lemma9}). For simplicity, we will denote this  $\mathcal{D}_{u}^{(k-1)}$-measurable version as $\varphi_{u}^{(k)}(t-S(u))$ as well; the context should prevent any ambiguity.

In the context of Corollary \ref{LemmaCondExpII}, $\varphi^{(0)}_{u_{1}}(t-S(u_{1}))$ is $\mathcal{G}_{0}^{(k)}$-measurable because $\mathcal{D}_{u_{1}}^{(-1)}=\mathcal{G}_{0}^{(k)}$. Similarly, for $i \geq 2$, $\varphi^{(k)}_{u_{i}}(t-S(u_{i}))$ is $\mathcal{G}_{i-1}^{(k)}$-measurable because $\mathcal{D}_{u_{i}}^{(k-1)} \subseteq \mathcal{G}_{i-1}^{(k)}$.

Let $(u_{n})_{n \in \mathbb{N}}$ be any admissible ordering of $\mathcal{I}$ and let $\varphi$ be an $h$-dependent descendant characteristic, for some $h \in \mathbb{N}_{0}$, satisfying \ref{C1}-\ref{C2}. For $i \geq 1$, let $u \in \mathcal{I}$ be such that $u_{i} \preceq u$ and $|u|-|u_{i}|\leq h$. Note that, by \eqref{eq134} and \eqref{eq126}, $\mathcal{I}^{(h+k+|u_{i}| -|u|-1)}_{u} \subseteq \mathcal{I}^{(h)}_{i}$, for $k=0,1$. Set $\mathcal{J}_{u} = \mathcal{I}^{(h)}_{i} \setminus \mathcal{I}^{(h+k+|u_{i}| -|u|-1)}_{u}$ and observe that $\mathcal{J}_{u} \cap u\mathcal{I} = \emptyset$. Then, by \eqref{eq11} and \eqref{eq2}, $\mathcal{D}_{u}^{(h+k+|u_{i}| -|u|-1)} \subseteq \mathcal{G}_{i}^{(h)}$. Therefore, $\varphi_{u}^{(h+k+|u_{i}| -|u|)}(t-S(u))$ is $\mathcal{G}_{i}^{(h)}$-measurable.
\end{remark}

%%%%%%%%%%%%%%%%%%%%%%%%%%%%%%%%%%%%%%%%%%%%%%%%%%%%%%%%%%%%%%%%%%%%%%%
 \subsection{Proof of Proposition \ref{Proposition3}} \label{ProofProposition}
%%%%%%%%%%%%%%%%%%%%%%%%%%%%%%%%%%%%%%%%%%%%%%%%%%%%%%%%%%%%%%%%%%%%%%%

In this section, we prove Proposition \ref{Proposition3}. Beforehand, we establish the following result, analogous to \cite[Lemma 4.1]{Iksanov2021}. Note that, for $k \in \mathbb{N}_{0}$, the collection of sub-$\sigma$-algebras $(\mathcal{G}_{n}^{(k)})_{n \in  \mathbb{N}_{0}}$ of $\mathcal{F}$ defined in \eqref{eq2} forms a filtration. 

\begin{lemma} \label{GeneralLemma}
Suppose that assumptions \ref{A1}-\ref{A7} hold and let $\varphi$ be an $h$-dependent descendant characteristic, for some $h \in \mathbb{N}_{0}$, satisfying \ref{C1}-\ref{C3}. Let $(u_{n})_{n \in \mathbb{N}}$ be any admissible ordering of $\mathcal{I}$. Fix $t \in \mathbb{R}$ and for $k=1,\dots h+1$, define $M_{0}^{(k)}(t) \coloneqq 0$, and for $n \in \mathbb{N}$,
\begin{align} \label{eq21}
M_{n}^{(k)}(t) & \coloneqq \sum_{i=1}^{n} (\varphi_{u_{i}}^{(k)}(t-S(u_{i})) - \mathbb{E}[\varphi_{u_{i}}^{(k)}(t-S(u_{i})) \mid \mathcal{G}^{(k-1)}_{i-1}]).
\end{align}
\noindent Then:
\begin{enumerate}[label=(\textbf{\roman*})]
\item $(M_{n}^{(k)}(t))_{n \in \mathbb{N}_{0}}$ is a well-defined centred martingale with respect to $(\mathcal{G}_{n}^{(k-1)})_{n \in \mathbb{N}_{0}}$. Moreover, for $n \in \mathbb{N}$, \label{Pro1}
\begin{align} 
M^{(k)}_{n}(t) &  \stackrel{a.s.}{=}  (\varphi_{u_{1}}^{(k-1)}(t) - \varphi_{u_{1}}^{(0)}(t)) + \sum_{i=1}^{n} (\varphi_{u_{i}}^{(k)}(t-S(u_{i})) - \varphi_{u_{i}}^{(k-1)}(t-S(u_{i}))).  \label{eq26}
\end{align}

\item The increments of $(M_{n}^{(k)}(t))_{n \in \mathbb{N}_{0}}$ are uncorrelated. \label{Pro2}
\item $(M_{n}^{(k)}(t))_{n \in \mathbb{N}_{0}}$ is bounded in $L^{2}(\Omega, \mathcal{F}, \mathbb{P})$, that is, $\sup_{n \in \mathbb{N}_{0}} \mathbb{E}[(M_{n}^{(k)}(t))^{2}] < \infty$. \label{Pro3}
\item The series \label{Pro4}
\begin{align} \label{eq25}
M^{(k)}(t) \coloneqq (\varphi_{\varnothing}^{(k-1)}(t) - \varphi_{\varnothing}^{(0)}(t)) + \sum_{u \in \mathcal{I}} (\varphi_{u}^{(k)}(t-S(u)) - \varphi_{u}^{(k-1)}(t-S(u))).
\end{align}
\noindent converges unconditionally in $L^{2}(\Omega, \mathcal{F}, \mathbb{P})$ and it is also the almost sure limit  of $(M_{n}^{(k)}(t))_{n \in \mathbb{N}_{0}}$, as $n \rightarrow \infty$. Furthermore, for any deterministic sequence $(\mathcal{I}_{n})_{n \in \mathbb{N}}$ such that $\mathcal{I}_{n} \uparrow \mathcal{I}$, as $n \rightarrow \infty$, we have that 
\begin{align}
M^{(k)}_{\mathcal{I}_{n}}(t) \coloneqq (\varphi_{\varnothing}^{(k-1)}(t) - \varphi_{\varnothing}^{(0)}(t)) + \sum_{u \in \mathcal{I}_{n}} (\varphi_{u_{i}}^{(k)}(t-S(u_{i})) - \varphi_{u_{i}}^{(k-1)}(t-S(u_{i})))
\end{align}
\noindent converges in $L^{2}(\Omega, \mathcal{F}, \mathbb{P})$ towards $M^{(k)}(t)$, as $n \rightarrow \infty$.
\end{enumerate}
\end{lemma}

\begin{proof}
We start with the proof of \ref{Pro1}.  Recall that $u_{1} = \varnothing$ and $S(u_{1}) = 0$. By \eqref{eq27} and \eqref{eq2},
\begin{align} \label{eq36}
\mathbb{E}[\varphi_{u_{1}}^{(k)}(t) \mid \mathcal{G}^{(k-1)}_{0}] \stackrel{a.s.}{=} \mathbb{E}[\varphi_{u_{1}}^{(k)}(t) ] = \mathbb{E}[\varphi(t)] \stackrel{a.s.}{=} \varphi_{u_{1}}^{(0)}(t).
\end{align} 
\noindent On the other hand, by \eqref{eq2}, note that, for every $i \geq 2$, $\mathcal{G}^{(k-1)}_{i-1} \subseteq \mathcal{G}^{(k)}_{i-1}$. Then, by \eqref{eq38} in Corollary \ref{LemmaCondExpII},
\begin{align} \label{eq14}
\mathbb{E}[\varphi_{u_{i}}^{(k)}(t-S(u_{i})) \mid \mathcal{G}^{(k-1)}_{i-1}] & \stackrel{a.s.}{=}  \mathbb{E}[\mathbb{E}[\varphi_{u_{i}}(t-S(u_{i})) \mid  \mathcal{G}^{(k)}_{i-1}] \mid  \mathcal{G}^{(k-1)}_{i-1}] \nonumber \\
& \stackrel{a.s.}{=}  \mathbb{E}[\varphi^{(k)}_{u_{i}}(t-S(u_{i})) \mid  \mathcal{G}^{(k-1)}_{i-1}] \nonumber \\
& \stackrel{a.s.}{=} \varphi_{u_{i}}^{(k-1)}(t-S(u_{i})).
\end{align} 
\noindent Then, by \eqref{eq36}, \eqref{eq14} and Corollary \ref{corollary1} \ref{corollary1Pro1}, $M_{n}^{(k)}(t) \in L^{2}(\Omega, \mathcal{F}, \mathbb{P})$. Clearly, $(M_{n}^{(k)}(t))_{n \in \mathbb{N}_{0}}$ is adapted to $(\mathcal{G}_{n}^{(k-1)})_{n \in \mathbb{N}_{0}}$ (recall Remark \ref{Remark7}). The martingale property follows immediately from the definition of $M_{n}^{(k)}(t)$. Clearly,  $(M_{n}^{(k)}(t))_{n \in \mathbb{N}_{0}}$ is centred. Finally, \eqref{eq36} and \eqref{eq14} imply \eqref{eq26}, which concludes with the proof of \ref{Pro1}. 

We prove \ref{Pro2}. By the Cauchy–Schwarz inequality, for $i,j \in \mathbb{N}$,
\begin{align}
(\varphi_{u_{i}}^{(k)}(t-S(u_{i})) - \mathbb{E}[\varphi_{u_{i}}^{(k)}(t-S(u_{i})) \mid \mathcal{G}^{(k-1)}_{i-1}])(\varphi_{u_{j}}^{(k)}(t-S(u_{j})) - \mathbb{E}[\varphi_{u_{j}}^{(k)}(t-S(u_{j})) \mid \mathcal{G}^{(k-1)}_{j-1}])
\end{align}
\noindent is in $L^{1}(\Omega, \mathcal{F}, \mathbb{P})$. For $i,j \in \mathbb{N}$ such that $i < j$ and $t \in \mathbb{R}$, we have that $\varphi_{u_{i}}^{(k)}(t-S(u_{i}))$ and $\mathbb{E}[\varphi_{u_{i}}^{(k)}(t-S(u_{i})) \mid \mathcal{G}^{(k-1)}_{i-1}]$ are $\mathcal{G}_{j-1}^{(k-1)}$-measurable (recall Remark \ref{Remark7}). Then, \ref{Pro2} follows by properties of the conditional expectation.

We continue with the proof of \ref{Pro3}. Note that $M_{0}^{(k)}
(t) = 0$. Moreover, by \ref{Pro2}, 
\begin{align}  \label{eq48}
\mathbb{E}[(M_{n}^{(k)}(t))^{2}]  & = \sum_{i =1}^{n} \mathbb{E}[(\varphi_{u_{i}}^{(k)}(t-S(u_{i})) - \mathbb{E}[\varphi_{u_{i}}^{(k)}(t-S(u_{i})) \mid \mathcal{G}^{(k-1)}_{i-1}])^{2}].
\end{align}
\noindent Note that,  by \eqref{eq27}, \eqref{eq2}, \eqref{eq36} and Jensen’s inequality,
\begin{align} \label{eq49}
\mathbb{E}[(\varphi_{u_{1}}^{(k)}(t) - \mathbb{E}[\varphi_{u_{1}}^{(k)}(t) \mid \mathcal{G}^{(k-1)}_{0}])^{2}] & = \mathbb{E}[(\mathbb{E}[\varphi(t) \mid \mathcal{G}^{(k)}]  - \mathbb{E}[\varphi(t)])^{2}] \nonumber \\
& \leq \mathbb{E}[\mathbb{E}[(\varphi(t) - \mathbb{E}[\varphi(t)])^{2} \mid \mathcal{G}^{(k)}]] \nonumber \\
& = {\rm Var}(\varphi(t)).
\end{align}
\noindent Recall that $(x+y)^{2} \leq 2 (x^{2}+y^{2})$, for $x,y \in \mathbb{R}$. Then, for $i \geq 2$, by \eqref{eq14} and Corollary \ref{corollary1} \ref{corollary1Pro4},
\begin{align} \label{eq51}
& \mathbb{E}[(\varphi_{u_{i}}^{(k)}(t-S(u_{i})) - \mathbb{E}[\varphi_{u_{i}}^{(k)}(t-S(u_{i})) \mid \mathcal{G}^{(k-1)}_{i-1}])^{2}] \nonumber \\
& \quad \quad  = \mathbb{E}[ (\varphi_{u_{i}}^{(k)}(t-S(u_{i})) - \varphi_{u_{i}}^{(k-1)}(t-S(u_{i}))^{2}] \nonumber \\
& \quad \quad \leq 2 \mathbb{E}[ (\varphi_{u_{i}}^{(k)}(t-S(u_{i})) - \varphi_{u_{i}}^{(0)}(t-S(u_{i})) )^{2}] + 2 \mathbb{E}[ ( \varphi_{u_{i}}^{(k-1)}(t-S(u_{i})) -\varphi_{u_{i}}^{(0)}(t-S(u_{i}))  )^{2}]\nonumber \\
& \quad \quad \leq 8\mathbb{E}[ {\rm Var}[\varphi](t-S(u_{i}))].
\end{align}
\noindent Therefore, by combining \eqref{eq48}, \eqref{eq49} and \eqref{eq51}, we conclude that, for $n \in \mathbb{N}$,
\begin{align} \label{eq22}
\mathbb{E}[(M_{n}^{(k)}(t))^{2}]  \leq 8 \mathbb{E}\left[ \sum_{i =1}^{n} {\rm Var}[\varphi](t-S(u_{i})) \right] \leq 8\mathbb{E}[\mathcal{Z}_{t}^{{\rm Var}[\varphi]}]<\infty.
\end{align}
\noindent Recall that by \ref{C3} the variance function ${\rm Var}[\varphi]$ is c\`adl\`ag. Furthermore, \ref{C2} implies that ${\rm Var}[\varphi]$ satisfies \ref{C1}, which in turn, by \eqref{eq109}, ensures $\mathbb{E}[\mathcal{Z}_{t}^{{\rm Var}[\varphi]}]<\infty$, for $t \in \mathbb{R}$. Therefore, \eqref{eq22} implies \ref{Pro3}.  

Finally, we prove \ref{Pro4}. By \ref{Pro3} and well-known results in martingale theory, the martingale $(M_{n}^{(k)}(t))_{n \in \mathbb{N}_{0}}$ converges in $L^{2}(\Omega, \mathcal{F}, \mathbb{P})$ and almost surely. Then, for any subset $\mathcal{J} \subseteq \mathbb{N}$, finite or infinite, since the martingale increments are uncorrelated, we deduce that
\begin{align} \label{eq12}
\mathbb{E} \left[ \left( \sum_{i \in \mathcal{J}} (\varphi_{u_{i}}^{(k)}(t-S(u_{i})) - \mathbb{E}[\varphi_{u_{i}}^{(k)}(t-S(u_{i})) \mid \mathcal{G}^{(k-1)}_{i-1}]) \right)^{2} \right]  \leq 4 \mathbb{E} \left[  \sum_{i \in \mathcal{J}} {\rm Var}[\varphi](t-S(u_{i}))  \right].
\end{align}
\noindent Note also that for every $u \in \mathcal{I}$, there exists $i \in \mathbb{N}$ such that $u_{i}=u$. From this, \eqref{eq12} and the Cauchy criterion, on the one hand, we deduce the unconditional convergence in $L^{2}(\Omega, \mathcal{F}, \mathbb{P})$ of the series in \eqref{eq25}, thereby justifying to write $M^{(k)}(t)$ for the limit of $(M_{n}^{(k)}(t))_{n \in \mathbb{N}_{0}}$. On the other hand, we also conclude the convergence of $M^{(k)}_{\mathcal{I}_{n}}(t)$ to $M^{(k)}(t)$ in $L^{2}(\Omega, \mathcal{F}, \mathbb{P})$. This finishes the proof of \ref{Pro4}.
\end{proof}

We are now ready to prove Proposition \ref{Proposition3}.

\begin{proof}[Proof of Proposition \ref{Proposition3}]
By \ref{C1}, $\mathbb{E}[\varphi](t)$ is finite for every $t \in \mathbb{R}$, then, for every $u \in \mathcal{I}$,
\begin{align}
\varphi_{u}(t-S(u)) = \mathbb{E}[\varphi](t-S(u)) + (\varphi_{u}(t-S(u)) -\mathbb{E}[\varphi](t-S(u))). 
\end{align}
\noindent To conclude, it suffices to prove that both series
\begin{align} \label{eq20}
\sum_{u \in \mathcal{I}} \mathbb{E}[\varphi](t-S(u)) \quad \text{and} \quad \sum_{u \in \mathcal{I}} (\varphi_{u}(t-S(u)) -\mathbb{E}[\varphi](t-S(u)))
\end{align}
\noindent converge unconditionally in $L^{1}(\Omega, \mathcal{F}, \mathbb{P})$ and almost surely over every admissible ordering of $\mathcal{I}$. For the first series, note that by \ref{C1}, the function $t \in \mathbb{R} \rightarrow |\mathbb{E}[\varphi](t)| e^{- \alpha t}$ is directly Riemann integrable. Then, by \eqref{eq109} (recall also Section \ref{ExpSection}), the first series converges  unconditionally in $L^{1}(\Omega, \mathcal{F}, \mathbb{P})$ and absolutely almost surely over every admissible ordering of $\mathcal{I}$. 

Next, observe that, for $u \in \mathcal{I}$,
\begin{align}
\varphi_{u}(t-S(u)) -\mathbb{E}[\varphi](t-S(u)) = \sum_{k=1}^{h+1} (\varphi_{u}^{(k)}(t-S(u)) -\varphi_{u}^{(k-1)}(t-S(u))).
\end{align}
\noindent Then, by \eqref{eq25},
\begin{align} \label{eq23} 
\sum_{u \in \mathcal{I}} (\varphi_{u}(t-S(u)) -\mathbb{E}[\varphi](t-S(u))) = \sum_{k=1}^{h+1} M^{(k)}(t) - \sum_{k=1}^{h+1} (\varphi_{\varnothing}^{(k-1)}(t) - \varphi_{\varnothing}^{(0)}(t)).
\end{align}
\noindent Therefore, Corollary \ref{corollary1} \ref{corollary1Pro1} and Lemma \ref{GeneralLemma} \ref{Pro4} imply that the second series in \eqref{eq20} converges unconditionally in $L^{2}(\Omega, \mathcal{F}, \mathbb{P})$ and almost surely over every admissible ordering of $\mathcal{I}$.
\end{proof}

%%%%%%%%%%%%%%%%%%%%%%%%%%%%%%%%%%%%%%%%%%%%%%%%%%%%%%%%%%%%%%%%%%%%%%%
\subsection{Properties of the characteristic $\chi^{(\varphi, h)}$}  \label{MoreAsymptotic}
%%%%%%%%%%%%%%%%%%%%%%%%%%%%%%%%%%%%%%%%%%%%%%%%%%%%%%%%%%%%%%%%%%%%%%%

We dedicate this section to establishing some properties of the random function $\chi^{(\varphi, h)}$ defined in \eqref{eq53}. In particular, we will prove that it is a well-defined characteristic. We begin with an auxiliary result. Recall the definition of $\varphi^{(k)}$, for $k \in \mathbb{N}_{0}$, in \eqref{eq27}. 

\begin{lemma} \label{lemma4}
Suppose that assumptions \ref{A1}-\ref{A7} hold and let $\varphi$ be an $h$-dependent descendant characteristic, for some $h \in \mathbb{N}_{0}$, satisfying \ref{C1}-\ref{C2}. Then, we have the following:
\begin{enumerate}[label=(\textbf{\roman*})]
\item \label{lemma4Pro1} For every $t \in \mathbb{R}$, $n \in \mathbb{N}_{0}$ and $k=0, \dots, h+1$, 
\begin{align} \label{eq99}
\sum_{|u|=n} |\varphi_{u}^{(k)}(t-S(u))| \in L^{2}(\Omega, \mathcal{F}, \mathbb{P}). 
\end{align}

\item \label{lemma4Pro2} Suppose that $\varphi$ also satisfies \ref{C5} and let $\varepsilon >0$ be as defined in \ref{C5}. Then, for any $t \in \mathbb{R}$, $k=0, \dots, h+1$ and $i=0, \dots, h$, 
\begin{align} \label{eq139}
\left( \left(\sum_{|u|=i} |\varphi_{u}^{(k)}(s-S(u))| \right)^{2} \right)_{|s-t| \leq \varepsilon} \quad \text{is uniformly integrable}.
\end{align}

\item \label{lemma4Pro3}  Suppose that $\varphi$ also satisfies \ref{C4}. Then, for $i=0, \dots, h$ and $k=0, \dots, h+1$, 
\begin{align}
t \in \mathbb{R} \mapsto  \sum_{|u|=i} \varphi_{u}^{(k)}(t-S(u))
\end{align}
\noindent almost surely has c\`adl\`ag  paths. 
\end{enumerate}
\end{lemma}

\begin{proof}
First, we prove \ref{lemma4Pro1}. Let $(\mathcal{A}^{(n)})_{n \geq -1}$ be the collection $\sigma$-algebras  defined in \eqref{eq135}. For $n \in\mathbb{N}_{0}$ and $k=0, \dots, h+1$, the triangle inequality, Corollary \ref{corollary1} \ref{corollary1Pro2}-\ref{corollary1Pro3} and \eqref{eq152} imply that there exists a constant $C>0$ such that
\begin{align} \label{eq50} 
\mathbb{E} \left[ \left( \sum_{ |u|=n} |\varphi_{u}^{(k)}(t-S(u))| \right)^{2} \right] & \leq  \mathbb{E} \left[  \sum_{|v|=n} \sum_{|u|=n} |\varphi_{u}^{(k)}(t-S(u))| |\varphi_{v}^{(k)}(t-S(v))|  \right] \nonumber \\ 
& = \mathbb{E} \left[\sum_{ |v|=n} \sum_{ |u|=n} \mathbb{E} \left[ |\varphi_{u}^{(k)}(t-S(u))| |\varphi_{v}^{(k)}(t-S(v))|  \mid \mathcal{A}^{(n-1)} \right] \right] \nonumber \\ 
&  \leq  C \mathbb{E} \left[\sum_{ |v|=n} \sum_{ |u|=n} \prod_{w=u,v} e^{\frac{\alpha}{2} \max(t-S(w), 2(t-S(w)))} \right] \nonumber \\
&  \leq Ce^{\alpha\max(t,t/2)} \mathbb{E} \left[ \left( \sum_{|u|=n} e^{-\frac{\alpha}{2} S(u)} \right)^{2} \right] < \infty. 
% & = C e^{\alpha\max(t,t/2)} \mathbb{E}[Z^{2}_{n}] < \infty. 
\end{align}
\noindent Therefore, \eqref{eq50} implies \ref{lemma4Pro1}. 

Next, we prove \ref{lemma4Pro2}. Consider $i= 0, \dots, h$ and $k=0, \dots, h+1$. By \cite[Theorem 2 (f), Section 7, Chapter II]{Shiryaev1996} and Corollary \ref{corollary1} \ref{corollary1Pro6}, we have that almost surely,
\begin{align} \label{eq140}
\mathbb{E} \left[ \sum_{|u|=i} |\varphi_{u}(t-S(u))| \mid \mathcal{A}^{(i+k-1)} \right]  =  \sum_{|u|=i} \mathbb{E} [ |\varphi_{u}(t-S(u))| \mid \mathcal{A}^{(i+k-1)}]  \geq \sum_{|u|=i}   |\varphi_{u}^{(k)}(t-S(u))|.
\end{align}
\noindent Thus, by \eqref{eq140} and Jensen's inequality, we have that almost surely,
\begin{align}
\left( \sum_{|u|=i}  | \varphi_{u}^{(k)}(t-S(u))| \right)^{2}  \leq  \mathbb{E} \left[ \left(\sum_{|u|=i} |\varphi_{u}(t-S(u))| \right)^{2}  \, \, \Big | \mathcal{A}^{(i+k-1)} \right].
\end{align}
\noindent Therefore, \ref{lemma4Pro2} follows from \ref{C5} and Lemma \ref{lemma14}.

Finally, we prove \ref{lemma4Pro3}. Consider $i=0, \dots, h$ and $k=0, \dots, h+1$. Recall that, by \ref{C3}, $\varphi^{(k)}$ has almost surely c\`adl\`ag paths. Then, 
\begin{align} \label{eq179}
t \in \mathbb{R} \mapsto  \sum_{|u|=i} \varphi_{u}^{(k)}(t-S(u)) \mathbf{1}_{[0, \infty)}(t-S(u))
\end{align}
\noindent has almost surely c\`adl\`ag paths. To see this, recall that assumptions \ref{A1}-\ref{A4} imply the total number of births up to time $t \in \mathbb{R}$ is finite almost surely (see Remark \ref{Remark3}). Thus, since the sum in \eqref{eq179} almost surely has only finitely many non-zero summands, it follows that it is a finite sum of c\`adl\`ag functions.

To prove our claim in \ref{lemma4Pro3}, it remains only to show that
\begin{align} \label{eq180}
t \in \mathbb{R} \mapsto  \sum_{|u|=i} \varphi_{u}^{(k)}(t-S(u)) \mathbf{1}_{ (-\infty,0)}(t-S(u))
\end{align}
\noindent has almost surely c\`adl\`ag paths. It follows from \ref{C4}, Corollary \ref{corollary1} \ref{corollary1Pro6} and properties of the conditional expectation that for any $t \in \mathbb{R}$ there exists $\varepsilon >0$ such that  
\begin{align} 
\mathbb{E} \left[  \sum_{|u|=i} \sup_{|s-t|\leq \varepsilon}|\varphi_{u}^{(k)}(s-S(u))\mathbf{1}_{(-\infty,0)}(s-S(u))| \right] < \infty. 
\end{align}
\noindent In particular, for any $t \in \mathbb{R}$ there exists $\varepsilon >0$ such that  
\begin{align} 
 \sum_{|u|=i} \sup_{|s-t|\leq \varepsilon}|\varphi_{u}^{(k)}(s-S(u))\mathbf{1}_{(-\infty,0)}(s-S(u))|  < \infty, \quad \text{almost surely}. 
\end{align}
\noindent Thus, by the Weierstrass M-test (see for e.g., \cite[Theorem 9.6]{Apostol1974}), for every $t \in \mathbb{R}$, the series 
\begin{align}
\sum_{|u|=i} \varphi_{u}^{(k)}(t-S(u)) \mathbf{1}_{ (-\infty,0)}(t-S(u)) 
\end{align}
\noindent converges uniformly on $[t-\varepsilon, t+\varepsilon]$ almost surely. Thus, it is a straightforward exercise to show that this implies the function in \eqref{eq180} is almost surely c\`adl\`ag.

This concludes our proof of \ref{lemma4Pro3}. 
\end{proof}

\begin{remark} \label{Remark5} 
Suppose that assumptions \ref{A1}-\ref{A7} hold and let $\phi$ be an $h$-dependent descendant characteristic, for some $h \in \mathbb{N}_{0}$, satisfying \ref{C1}-\ref{C2}. By Lemma \ref{lemma4} \ref{lemma4Pro1}, the series 
\begin{align} 
\sum_{|u|=n} \varphi_{u}^{(k)}(t-S(u)) 
\end{align}
\noindent converges absolutely almost surely, for every $n \in \mathbb{N}_{0}$ and $k=0, \dots, h+1$. Then, by \eqref{eq53}, we have that
\begin{align} \label{eq177}
\chi^{(\varphi, h)}(t) =  \sum_{i =0}^{h}\sum_{|u| = i} (\varphi_{u}^{(h+1-i)}(t-S(u)) - \varphi_{u}^{(h-i)}(t-S(u))), \quad \text{for} \quad  t \in \mathbb{R}.
\end{align}
\noindent In particular, $|\chi^{(\varphi, h)}(t)|<\infty$ almost surely. 
\end{remark}

\begin{lemma} \label{lemma16}
Suppose that assumptions \ref{A1}-\ref{A4} hold and let $\varphi$ be a characteristic satisfying \ref{C2}-\ref{C3}. Then, for every $h \in \mathbb{N}_{0}$, the function 
\begin{align} \label{eq103}
t\in \mathbb{R} \mapsto e^{-\alpha t} \int_{[0, \infty)} {\rm Var}[\varphi](t-x) \nu^{(h)}({\rm d} x)
\end{align}
\noindent is directly Riemman integrable, where $\nu^{(h)}$ is defined in \eqref{eq62}.
\end{lemma}

The proof of Lemma \ref{lemma16} is postponed to Appendix \ref{Append1}. We now present some key properties of $\chi^{(\varphi, h)}$.

\begin{lemma} \label{lemma6} 
Suppose that assumptions \ref{A1}-\ref{A7} hold and  let $\varphi$ be an $h$-dependent descendant characteristic, for some $h \in \mathbb{N}_{0}$, satisfying \ref{C1}-\ref{C2}. Then, we have the following:
\begin{enumerate}[label=(\textbf{\roman*})]
\item \label{lemma6Pro1} For every $t \in \mathbb{R}$, $\chi^{(\varphi, h)}(t) \in L^{2}(\Omega, \mathcal{F}, \mathbb{P})$. 
\setcounter{Cond3}{\value{enumi}}
\end{enumerate}

\noindent Assume further that \ref{C4}-\ref{C5} holds. Then, we also have the following:
\begin{enumerate}[label=(\textbf{\roman*})]
\setcounter{enumi}{\value{Cond3}}
\item $\chi^{(\varphi,h)}$ almost surely has c\`adl\`ag paths. \label{lemma6Pro1b}

\item $\chi^{(\varphi,h)}$  satisfies \ref{C1}-\ref{C3}.  \label{lemma6Pro4}

\item \label{lemma6Pro4b} Let $(u_{n})_{n \in \mathbb{N}}$ be any admissible ordering of $\mathcal{I}$ and recall that $\mathcal{I}_{n} = \{u_{1}, \dots, u_{n}\}$, for $n \in \mathbb{N}$. For $n \in \mathbb{N}$ and $t \in \mathbb{R}$, define
\begin{align} \label{eq32}
Y_{n}^{\varphi}(t) \coloneqq \sum_{u \in \mathcal{I}_{n}} \chi^{(\varphi,h)}_{u}(t-S(u)) = \sum_{i=1}^{n} \chi^{(\varphi,h)}_{u_{i}}(t-S(u_{i})). 
\end{align}
\noindent Then, for every $t \in \mathbb{R}$, as $n \rightarrow \infty$, $Y_{n}^{\varphi}(t) \rightarrow  \mathcal{Z}_{t}^{\chi^{(\varphi,h)}}$, in $L^{1}(\Omega, \mathcal{F}, \mathbb{P})$.

\item For every $t \in \mathbb{R}$,  $\mathbb{E}[\mathcal{Z}_{t}^{(\chi^{(\varphi,h)})^{2}}] < \infty$.
\label{lemma6Pro2}

\item For all $t \in \mathbb{R}$ and any admissible order $(u_{n})_{n \in \mathbb{N}}$ of $\mathcal{I}$, as $n \rightarrow \infty$, \label{lemma6Pro3}
\begin{align} 
\sum_{i=n+1}^{\infty} (\chi^{(\varphi,h)}_{u_{i}}(t-S(u_{i})))^{2} \rightarrow 0, \quad \text{in} \quad L^{1}(\Omega, \mathcal{F}, \mathbb{P}).
\end{align}

\item For $t \in \mathbb{R}$, \label{lemma6Pro5}
\begin{align} \label{eq82}
\mathcal{Z}_{t}^{\varphi - \mathbb{E}[\varphi]} \stackrel{a.s.}{=} \mathcal{Z}_{t}^{\chi^{(\varphi,h)}} +  \sum_{0 \leq |u| \leq h-1}  ( \varphi_{u}^{(h-|u|)}(t-S(u)) - \varphi_{u}^{(0)}(t-S(u))).
\end{align}
\noindent (If $h=0$, the sum  $\sum_{0 \leq |u| \leq h-1}$ on the right-hand side of \eqref{eq82} is defined to be $0$.) 
\end{enumerate}
\end{lemma}

\begin{proof}
We start with the proof of \ref{lemma6Pro1}. By  Lemma \ref{lemma4} \ref{lemma4Pro1} and \eqref{eq177}, it follows that $\mathbb{E}[(\chi^{(\varphi, h)}(t))^{2}] < \infty$. This implies \ref{lemma6Pro1}.  Note that \ref{lemma6Pro1b} follows from Lemma \ref{lemma4} \ref{lemma4Pro3} and \eqref{eq177}.

Next, we prove \ref{lemma6Pro4}. By  Lemma \ref{lemma4} \ref{lemma4Pro1}, \eqref{eq177} and Fubini's theorem, we have that $\mathbb{E}[\chi^{(\varphi, h)}(t)] = 0$, for all $t \in \mathbb{R}$. Then, by \ref{lemma6Pro1}, $\chi^{(\varphi, h)}$ satisfies \ref{C1}. It remains to prove that $\chi^{(\varphi, h)}$ satisfies \ref{C2}-\ref{C3}. By Lemma \ref{lemma4} \ref{lemma4Pro1}, Fubini's theorem and Corollary \ref{corollary1} \ref{corollary1Pro4}-\ref{corollary1Pro5}, we have that, for $i \in \mathbb{N}_{0}$,
\begin{align} \label{eq24}
& \mathbb{E} \left[ \left( \sum_{|u|=i} (\varphi_{u}^{(k)}(t-S(u)) -  \varphi_{u}^{(0)}(t-S(u))) \right)^{2} \right] \nonumber \\
& \quad \quad = \sum_{|u|=i} \sum_{|v|=i}
\mathbb{E} \left[ \mathbb{E}[ (\varphi_{u}^{(k)}(t-S(u)) -  \varphi_{u}^{(0)}(t-S(u)))(\varphi_{v}^{(k)}(t-S(v)) -  \varphi_{v}^{(0)}(t-S(v))) \mid \mathcal{A}^{(i-1)}] \right] \nonumber  \\
& \quad \quad = \mathbb{E} \left[ \sum_{|u|=i}  \mathbb{E} [(\varphi_{u}^{(k)}(t-S(u)) -  \varphi_{u}^{(0)}(t-S(u)) )^{2}\mid \mathcal{A}^{(i-1)}] \right]  \nonumber  \\
& \quad \quad \leq 2\mathbb{E} \left[ \sum_{|u|=i}  {\rm Var}[\varphi](t-S(u)) \right]. 
\end{align}
\noindent Let $\xi^{(i)}({\rm d} x) \coloneqq  \sum_{|u|=i} \delta_{S(u)}({\rm d} x)$ be the point process giving unit mass to each $S(u)$, for $u \in \mathcal{I}$ such that $|u| = i$. Note that the $i$-fold convolution $\mu^{\ast i}$ of $\mu$ is the intensity measure of $\xi^{(i)}$ ($\mu^{\ast 0}$ is the Dirac measure at $0$). Then, for $k =0, \dots, h+1$, by \eqref{eq24},
\begin{align} \label{eq29}
\mathbb{E} \left[ \left( \sum_{|u|=i} (\varphi_{u}^{(k)}(t-S(u)) -  \varphi_{u}^{(0)}(t-S(u))) \right)^{2} \right] &  \leq 2 \mathbb{E} \left[  \int_{[0, \infty)} {\rm Var}[\varphi](t-x) \xi^{(i)}({\rm d} x)  \right] \nonumber \\
& = 2\int_{[0, \infty)} {\rm Var}[\varphi](t-x) \mu^{\ast i}({\rm d} x). 
\end{align}
\noindent On the other hand, by the Minkowski inequality and \eqref{eq29}, 
\begin{align} \label{eq34}
\mathbb{E} \left[ \left( \sum_{i=0}^{h}\sum_{|u|=i} (\varphi_{u}^{(h+1-i)}(t-S(u)) -  \varphi_{u}^{(0)}(t-S(u))) \right)^{2} \right] & \leq 2(h+1)^{2} \sum_{i=0}^{h} \int_{[0, \infty)} {\rm Var}[\varphi](t-x) \mu^{\ast i}({\rm d} x) \nonumber \\
& = 2(h+1)^{2} \int_{[0, \infty)} {\rm Var}[\varphi](t-x) \nu^{(h)}({\rm d} x),
\end{align}
\noindent where $\nu^{(h)}$ is the measure defined in \eqref{eq62}. Similarly, 
\begin{align} \label{eq52}
\mathbb{E} \left[ \left( \sum_{i=0}^{h}\sum_{|u|=i} (\varphi_{u}^{(h-i)}(t-S(u)) -  \varphi_{u}^{(0)}(t-S(u))) \right)^{2} \right] \leq 2h^{2}\int_{[0, \infty)} {\rm Var}[\varphi](t-x) \nu^{(h)}({\rm d} x).
\end{align}
\noindent By \eqref{eq177}, \eqref{eq34}, \eqref{eq52} and the Minkowski inequality,
\begin{align}  \label{eq65}
\mathbb{E}[ (\chi^{(\varphi, h)}(t))^{2}] \leq 8(h+1)^{2} \int_{[0, \infty)} {\rm Var}[\varphi](t-x) \nu^{(h)}({\rm d} x).
\end{align}

By Lemma \ref{lemma4} \ref{lemma4Pro2} and \eqref{eq177}, for any $t \in \mathbb{R}$ there exists $\varepsilon >0$ such that $((\chi^{(\varphi, h)}(s))^{2})_{|s-t|\leq \varepsilon}$ is uniformly integrable. That is, $\chi^{(\varphi, h)}$ satisfies \ref{C3}. Then, by \cite[Proposition 4.12]{Kallenberg2002}, \ref{lemma6Pro1} and \ref{lemma6Pro1b}, the function $t\in \mathbb{R} \mapsto \mathbb{E}[(\chi^{(\varphi, h)}(t))^{2}]$ has c\`adl\`ag paths. Note that $\text{Var}[\chi^{(\varphi, h)}](t) = \mathbb{E}[\chi^{(\varphi, h)}(t)]$. Therefore, by \ref{lemma6Pro1}, \eqref{eq65}, Lemma \ref{lemma16} and (a slightly extended version of) \cite[Remark 3.10.5, p.\ 237]{Sidney1992}, $\chi^{(\varphi, h)}$ satisfies \ref{C2}. This concludes the proof of \ref{lemma6Pro4}.

Note that \ref{lemma6Pro4b} follows by \ref{lemma6Pro4} and Proposition \ref{Proposition3} (applied to the characteristic $\chi^{(\varphi, h)}$). Observe that $(\chi^{(\varphi, h)})^{2}$ satisfies \ref{C1} by \ref{lemma6Pro4}. Then, \ref{lemma6Pro2} follows by \eqref{eq109} in Section \ref{ExpSection}. In particular, \ref{lemma6Pro3} is a direct consequence of \ref{lemma6Pro2}. 

Finally, we prove \ref{lemma6Pro5}. By \ref{lemma6Pro4} and Proposition \ref{Proposition3}, the series $\mathcal{Z}_{t}^{\chi^{(\varphi, h)}}$ converges unconditionally in $L^{1}(\Omega, \mathcal{F}, \mathbb{P})$ for all $t \in \mathbb{R}$. In particular, for any $n \in \mathbb{N}_{0}$, the infinite series
\begin{align}
\sum_{0 \leq |u| \leq n} \chi^{(\varphi, h)}_{u}(t-S(u))
\end{align}
\noindent also converges unconditionally in $L^{1}(\Omega, \mathcal{F}, \mathbb{P})$ and so is well-defined and converges to $\mathcal{Z}_{t}^{\chi^{(\varphi, h)}}$, as $n \rightarrow \infty$.  (Indeed, the well-definedness can also be seen as a consequence of Lemma \ref{lemma4} \ref{lemma4Pro1}.) Consider $u \in \mathcal{I}$ with $0 \leq |u| \leq n$ and let $v \in \mathcal{I}$ be an ancestor of $u$ (i.e., $v \preceq u$) such that $(|u| - h) \vee 0 \leq |v| \leq |u|$. Then, from \eqref{eq53} (or \eqref{eq177}), $\chi^{(\varphi, h)}_{v}(t-S(v))$ contains only one the term involving the vertex $u$: $\varphi_{u}^{(h+1-(|u|-|v|))}(t-S(u)) - \varphi_{u}^{(h-(|u|-|v|))}(t-S(u))$. However, if $v$ is ancestor of $u$ such that $|u| - |v| > h$, then $\chi^{(\varphi, h)}_{v}(t-S(v))$ contains no term involving $u$. Thus,
\begin{align} \label{eq167}
\sum_{0 \leq |u| \leq n} \chi^{(\varphi, h)}_{u}(t-S(u)) & = \sum_{0 \leq |u| \leq n} \sum_{i=0}^{|u| \wedge h} (\varphi_{u}^{(h+1-i)}(t-S(u)) - \varphi_{u}^{(h-i)}(t-S(u))) \nonumber \\
& = \sum_{0 \leq |u| \leq n}  (\varphi_{u}^{(h+1)}(t-S(u)) - \varphi_{u}^{( 0 \vee (h - |u|))}(t-S(u))) \nonumber \\
& = \sum_{0 \leq |u| \leq n}  (\varphi_{u}^{(h+1)}(t-S(u)) - \varphi_{u}^{(0)}(t-S(u))) \nonumber \\ 
& \quad \quad \quad + \sum_{0 \leq |u| \leq h-1}  (\varphi_{u}^{(0)}(t-S(u)) - \varphi_{u}^{(h-|u|)}(t-S(u))).
\end{align}
\noindent Note that Lemma \ref{lemma4} \ref{lemma4Pro1} justifies the manipulations in the preceding chain of equalities. Recall that, for $t \in \mathbb{R}$, $\varphi^{(h+1)}(t)  \stackrel{a.s.}{=} \varphi(t)$ and $\varphi^{(0)}(t)  \stackrel{a.s.}{=} \mathbb{E}[\varphi](t)$. Recall also from our assumptions and Proposition \ref{Proposition3} (see also Remark \ref{Remark4}) that the series  $ \mathcal{Z}_{t}^{\varphi}$ and $\mathcal{Z}_{t}^{\mathbb{E}[\varphi]}$ converge unconditionally in $L^{1}(\Omega, \mathcal{F}, \mathbb{P})$. Then,
\begin{align}
\sum_{0 \leq |u| \leq n} (\varphi_{u}^{(h+1)}(t-S(u)) - \varphi_{u}^{(0)}(t-S(u)))
\end{align}
\noindent converges in $L^{1}(\Omega, \mathcal{F}, \mathbb{P})$ to $\mathcal{Z}_{t}^{\varphi - \mathbb{E}[\varphi]} = \mathcal{Z}_{t}^{\varphi} - \mathcal{Z}_{t}^{\mathbb{E}[\varphi]}$, as $n \rightarrow \infty$. Therefore, \ref{lemma6Pro5} follows from \eqref{eq167} by letting $n \rightarrow \infty$. 
\end{proof}

The section ends with an estimate that will be utilized in our forthcoming section. Suppose that assumptions \ref{A1}-\ref{A7} hold and let $\varphi$ be an $h$-dependent descendant characteristic, for some $h \in \mathbb{N}_{0}$, satisfying \ref{C1}-\ref{C2}. For fixed $h \in \mathbb{N}_{0}$, Define the (random) function,
\begin{align} \label{eq67}
\varpi^{(\varphi, h)}(t) \coloneqq \sum_{i=0}^{h-1} \sum_{|u|=i} (\varphi_{u}^{(h-i)}(t-S(u)) - \varphi_{u}^{(0)}(t-S(u))), \quad \text{for} \quad  t \in \mathbb{R}.
\end{align}
\noindent  (If $h=0$, we define $\varpi^{(\varphi, 0)}(t) \coloneqq 0$.) By Lemma \ref{lemma4} \ref{lemma4Pro1}, $|\varpi^{(\varphi, h)}(t) |<\infty$ almost surely (recall Remark \ref{Remark5}). Moreover,
\begin{align} \label{eq178}
\varpi^{(\varphi, h)}(t) = \sum_{0 \leq |u| \leq h-1} (\varphi_{u}^{(h+1-|u|)}(t-S(u)) - \varphi_{u}^{(0)}(t-S(u))), \quad \text{for} \quad  t \in \mathbb{R}.
\end{align}

\begin{lemma} \label{lemma5}
Suppose that assumptions \ref{A1}-\ref{A7} hold and  let $\varphi$ be an $h$-dependent descendant characteristic, for some $h \in \mathbb{N}_{0}$, satisfying \ref{C1}-\ref{C3}. Then, as $n \rightarrow \infty$, 
\begin{align}
e^{-\frac{\alpha}{2} t}\varpi^{(\varphi, h)}(t)  \rightarrow 0, \quad \text{in} 
\quad L^{2}(\Omega, \mathcal{F}, \mathbb{P}).
\end{align}
\end{lemma}

\begin{proof}
By \eqref{eq52} and \eqref{eq67},
\begin{align}  \label{eq85}
e^{-\alpha t}\mathbb{E}[ (\chi^{(\varpi, h)}(t))^{2}] \leq 2h^{2} e^{-\alpha t} \int_{[0, \infty)} {\rm Var}[\varphi](t-x) \nu^{(h)}({\rm d} x).
\end{align}
\noindent Lemma \ref{lemma16} (see also \cite[Section 4, Chapter V]{Soren2003} or the proof of Lemma \ref{lemma16}) implies that
\begin{align}
\sum_{n \in \mathbb{N}_{0}} \sup_{t \in [n, n+1)}  e^{-\alpha t}  \int_{[0, \infty)} {\rm Var}[\varphi](t-x) \nu^{(h)}({\rm d} x) < \infty.
\end{align}
\noindent In particular, we can conclude that
\begin{align} \label{eq89} 
\lim_{n \rightarrow \infty} \sup_{t \in [n, n+1)}  e^{-\alpha t}  \int_{[0, \infty)} {\rm Var}[\varphi](t-x) \nu^{(h)}({\rm d} x) = 0.
\end{align}
\noindent Therefore, our claim follows from \eqref{eq85} and \eqref{eq89}. 
\end{proof}

Lemma \ref{lemma6} \ref{lemma6Pro1}-\ref{lemma6Pro1b} implies that $\chi^{(\varphi,h)}$ is a well-defined real-valued (descendant) characteristic. In particular, $(\chi^{(\varphi,h)})^{2}$ is a non-negative descendant characteristics.

\begin{remark} \label{Remark2}
Suppose that assumptions \ref{A1}-\ref{A7} hold and let $\phi$ be an $h$-dependent descendant characteristic, for some $h \in \mathbb{N}_{0}$, satisfying \ref{C1}-\ref{C2}. Let $f: \mathbb{R} \rightarrow \mathbb{R}$ be a deterministic c\`adl\`ag function, that is, $f$ be a deterministic individual characteristic. Set $\varphi \coloneqq \phi + f$. Then, by \eqref{eq27}, we have that, for every $t \in \mathbb{R}$, $\varphi^{(k)}(t) \stackrel{a.s.}{=}  \phi^{(k)}(t) + f(t)$, for $k \in \mathbb{N}_{0}$, where $\phi^{(k)}$ is defined analogously to $\varphi^{(k)}$ but with $\phi$ replacing $\varphi$. Moreover, $\chi^{(\varphi,h)}(t) \stackrel{a.s.}{=}  \chi^{(\phi,h)}(t)$. In particular, for all $h \in \mathbb{N}_{0}$ and $t \in \mathbb{R}$,  $\chi^{(f,h)}(t) \stackrel{a.s.}{=} 0$. 
\end{remark}

\begin{remark} \label{Remark8}
Suppose that assumptions \ref{A1}-\ref{A7} hold and let $\phi$ be a $0$-dependent descendant characteristic (i.e., an individual characteristic) satisfying \ref{C1}-\ref{C3}. By \eqref{eq53} (or \eqref{eq177}), for $t \in \mathbb{R}$, $\chi^{(\varphi,0)}(t) \stackrel{a.s.}{=} \varphi(t) -\mathbb{E}[\varphi](t)$ and thus, $\mathbb{E}[(\chi^{(\varphi,0)}(t))^{2}] = {\rm Var}[\varphi](t)$. Then, by \ref{C3}, $\chi^{(\varphi,0)}$ has almost surely has c\`adl\`ag paths and moreover, $t \in \mathbb{R} \mapsto \mathbb{E}[(\chi^{(\varphi,0)}(t))^{2}]$ has also c\`adl\`ag paths (recall the discussion following the introduction of assumptions \ref{C1}-\ref{C3}). Therefore, in this setting, it should be plain that the claims in Lemma \ref{lemma6} \ref{lemma6Pro1b}-\ref{lemma6Pro5} hold even without \ref{C4}-\ref{C5} being required; a simple modification of our arguments demonstrates this.
\end{remark}

%%%%%%%%%%%%%%%%%%%%%%%%%%%%%%%%%%%%%%%%%%%%%%%%%%%%%%%%%%%%%%%%%%%%%%%
\section{A central limit theorem for centred characteristics} \label{SecCLTcentred}
%%%%%%%%%%%%%%%%%%%%%%%%%%%%%%%%%%%%%%%%%%%%%%%%%%%%%%%%%%%%%%%%%%%%%%%

Throughout this section, suppose that assumptions \ref{A1}-\ref{A7} hold and let $\varphi$ be an $h$-dependent descendant characteristic, for some $h \in \mathbb{N}_{0}$, satisfying \ref{C1}-\ref{C5}. This section proves a limit theorem for the general branching process counted with centred characteristic $\varphi-\mathbb{E}[\varphi]$. 

\begin{theorem} \label{Theo1}
Suppose that assumptions \ref{A1}-\ref{A7} hold and  let $\varphi$ be an $h$-dependent descendant characteristic, for some $h \in \mathbb{N}_{0}$, satisfying \ref{C1}-\ref{C5}. Then, as $t \rightarrow \infty$, $t \in \mathbb{G}$, 
\begin{align}
 e^{-\frac{\alpha}{2}t} \mathcal{Z}_{t}^{\varphi-\mathbb{E}[\varphi]} \xrightarrow[]{\rm st} \left( \frac{W}{\beta} \int_{\mathbb{G}} \mathbb{E}[(\chi^{(\varphi,h)}(s))^{2}] e^{-\alpha s} \ell( {\rm d} s) \right)^{1/2} \mathcal{N},
\end{align}
\noindent where $\beta \in (0, \infty)$ is defined in \eqref{eq3}, $W$ is the limit of Nerman's martingale and $\mathcal{N}$ is a standard normal random variable independent of $\mathcal{F}$ (and thus of $W$). 
\end{theorem}

Our approach involves approximating the process $e^{-\frac{\alpha}{2}t} \mathcal{Z}_{t}^{\varphi-\mathbb{E}[\varphi]}$ by a martingale. Then, we apply a well-known martingale central limit theorem as stated in \cite[Corollary 3.1 on p.\ 58]{Hall1980}. Recall the definition of $Y^{\varphi}_{n}(t)$ in \eqref{eq32}. %Recall also that the real number $\alpha$ denotes that Malthusian parameter defined in \ref{A4}. 

\begin{lemma} \label{lemma2}
Suppose that assumptions \ref{A1}-\ref{A7} hold and  let $\varphi$ be an $h$-dependent descendant characteristic, for some $h \in \mathbb{N}_{0}$, satisfying \ref{C1}-\ref{C5}. Let $(u_{n})_{n \in \mathbb{N}}$ be any admissible ordering of $\mathcal{I}$ and let $(t_{n})_{n \in \mathbb{N}}$ be an arbitrary increasing sequence in $\mathbb{G}$ such that $\lim_{n \rightarrow \infty} t_{n} = \infty$. Then, there exists an increasing sequence of non-negative integers $(m_{n})_{n \in \mathbb{N}}$ such that, as $n \rightarrow \infty$,
\begin{enumerate}[label=(\textbf{\alph*})]
\item $\displaystyle e^{-\frac{\alpha}{2}t_{n}}(\mathcal{Z}_{t_{n}}^{\chi^{(\varphi,h)}}-Y^{\varphi}_{m_{n}}(t_{n})) \xrightarrow[]{\rm p} 0$, \label{Conv1}

\item $\displaystyle e^{- \alpha t_{n}}\sum_{i=m_{n}+1}^{\infty} (\chi^{(\varphi,h)}_{u_{i}}(t_{n}-S(u_{i})))^{2} \xrightarrow[]{\rm p}0$. \label{Conv3}
\end{enumerate}
\end{lemma}

\begin{proof}
By Lemma \ref{lemma6} \ref{lemma6Pro4b} and \ref{lemma6Pro3}, it is a simple exercise to construct an increasing sequence of non-negative integers $(m_{n})_{n \in \mathbb{N}}$ such that $\mathbb{E}[ |e^{-\frac{\alpha}{2}t_{n}}\mathcal{Z}_{t_{n}}^{\chi^{(\varphi,h)}} - e^{-\frac{\alpha}{2}t_{n}}Y^{\varphi}_{m_{n}}(t_{n})| ] \leq 2^{-n}$ and
\begin{align}
\mathbb{E} \left[ \left| e^{- \alpha t_{n}}\sum_{i=m_{n}+1}^{\infty} (\chi^{(\varphi,h)}_{u_{i}}(t_{n}-S(u_{i})))^{2} \right| \right] \leq 2^{-n}. 
\end{align}
\noindent Therefore, our claim follows by Markov's inequality. 
\end{proof}

Let $(u_{n})_{n \in \mathbb{N}}$ be any admissible ordering of $\mathcal{I}$. Let $(t_{n})_{n \in \mathbb{N}}$ be an arbitrary increasing sequence in $\mathbb{G}$ such that $\lim_{n \rightarrow \infty} t_{n} = \infty$ and let $(m_{n})_{n \in \mathbb{N}}$ be an increasing sequence of non-negative integers such that \ref{Conv1}-\ref{Conv3} in Lemma \ref{lemma2} hold. For each $n \in \mathbb{N}$, we proceed to define a martingale $(M_{n,j}^{\varphi})_{j\in \mathbb{N}_{0}}$. Clearly, the collection of sub-$\sigma$-algebras $(\mathcal{G}_{j}^{(h)})_{j \in  \mathbb{N}_{0}}$ of $\mathcal{F}$ defined in \eqref{eq2} forms a filtration. Define $M_{n,0}^{\varphi} \coloneqq 0$ and for $j \in \mathbb{N}$, define 
\begin{align}
M_{n,j}^{\varphi} & \coloneqq \sum_{i=1}^{j} e^{-\frac{\alpha}{2}t_{n}} \chi^{(\varphi, h)}_{u_{i}}(t_{n}-S(u_{i})).  \label{eq17} 
\end{align}

\begin{lemma} \label{lemma10}
Suppose that assumptions \ref{A1}-\ref{A7} hold and  let $\varphi$ be an $h$-dependent descendant characteristic, for some $h \in \mathbb{N}_{0}$, satisfying \ref{C1}-\ref{C5}. Then, for every $n \in \mathbb{N}$, $(M_{n,j}^{\varphi})_{j\in \mathbb{N}_{0}}$ is a zero-mean square-integrable martingale with respect the filtration $(\mathcal{G}_{j}^{(h)})_{j\in \mathbb{N}_{0}}$.
\end{lemma}

\begin{proof}
By \eqref{eq177} (or \eqref{eq53}), we observe that, for $i \in \mathbb{N}$,
\begin{align} \label{eq54}
\chi^{(\varphi, h)}_{u_{i}}(t_{n}-S(u_{i})) =  \sum_{l =0}^{h}\sideset{}{^{(l)}}\sum_{u_{i} \preceq u} (\varphi_{u}^{(h+1-l)}(t_{n}-S(u)) - \varphi_{u}^{(h-l)}(t_{n}-S(u))),
\end{align}
\noindent where the sum $\sum_{u_{i} \preceq u}^{(l)}$ is over all the vertices $u \in \mathcal{I}$ such that $u_{i} \preceq u$ and $|u|  -|u_{i}| = l$. 

Then, by Lemma \ref{lemma4} \ref{lemma4Pro1}, \eqref{eq54} and Fubini's theorem, we have that, for every $i \in \mathbb{N}$, $\mathbb{E}[\chi^{(\varphi, h)}_{u_{i}}(t_{n}-S(u_{i}))] = 0$. Thus,  $(M_{n,j}^{\varphi})_{j\in \mathbb{N}_{0}}$ is centred. Moreover, by Lemma \ref{lemma4} \ref{lemma4Pro1} and \eqref{eq54}, $(M_{n,j}^{\varphi})_{j\in \mathbb{N}_{0}}$ is square-integrable. Clearly, for every $j \in \mathbb{N}_{0}$, $M_{n,j}^{\varphi}$ is $\mathcal{G}_{j}^{(h)}$-measurable (recall Remark \ref{Remark7}). All that remains is to prove the martingale property, and for that, it is enough to show that, for $i \in \mathbb{N}_{0}$,
\begin{align} \label{eq72}
\mathbb{E}[\chi^{(\varphi, h)}_{u_{i}}(t_{n}-S(u_{i})) \mid \mathcal{G}_{i}^{(h)}] \stackrel{a.s.}{=} 0.
\end{align}
\noindent Since $\mathcal{G}_{0}^{(h)}$ is the trivial $\sigma$-algebra, \eqref{eq72} is clearly satisfied for $i =0$. Next, we consider $i \in \mathbb{N}$. Suppose that we have already proven that, for $k =0,1$, 
\begin{align} \label{eq73}
\mathbb{E} \left [ \sum_{l =0}^{h} \sideset{}{^{(l)}}\sum_{u_{i} \preceq u} \varphi_{u}^{(h+k-l)}(t_{n}-S(u)) \, \, \Big | \mathcal{G}_{i}^{(h)} \right ] \stackrel{a.s.}{=} \sum_{l =0}^{h} \sideset{}{^{(l)}}\sum_{u_{i} \preceq u} \mathbb{E} [\varphi_{u}^{(h+k-l)}(t_{n}-S(u)) \mid \mathcal{G}_{i}^{(h)}].
\end{align}
\noindent Then, Corollary \ref{corollary2} and \eqref{eq54} imply that 
\begin{align}
\mathbb{E}[\chi^{(\varphi, h)}_{u_{i}}(t_{n}-S(u_{i})) \mid \mathcal{G}_{i}^{(h)}] & \stackrel{a.s.}{=} \sum_{l =0}^{h}\sideset{}{^{(l)}}\sum_{u_{i} \preceq u} \mathbb{E} [\varphi_{u}^{(h+1-l)}(t_{n}-S(u)) - \varphi_{u}^{(h-l)}(t_{n}-S(u)) \mid \mathcal{G}_{i}^{(h)}] \nonumber \\
& \stackrel{a.s.}{=} \sum_{l =0}^{h}\sideset{}{^{(l)}}\sum_{u_{i} \preceq u} (\varphi_{u}^{(h-l)}(t_{n}-S(u)) - \varphi_{u}^{(h-l)}(t_{n}-S(u))) = 0,
\end{align}
\noindent which proves \eqref{eq72} for $i \in \mathbb{N}$. 

We will now prove \eqref{eq73} to conclude. By the triangle inequality,  we that almost surely
\begin{align}
\sum_{l =0}^{h} \sideset{}{^{(l)}}\sum_{u_{i} \preceq u} | \mathbb{E}[\varphi_{u}^{(h+k-l)}(t_{n}-S(u)) \mid \mathcal{G}_{i}^{(h)}] | \leq \sum_{l =0}^{h} \sideset{}{^{(l)}}\sum_{u_{i} \preceq u} \mathbb{E}[|\varphi_{u}^{(h+k-l)}(t_{n}-S(u))| \mid \mathcal{G}_{i}^{(h)} ].
\end{align}
\noindent Then, Lemma \ref{lemma4} \ref{lemma4Pro1} implies that
\begin{align} \label{eq81}
\sum_{l =0}^{h} \sideset{}{^{(l)}}\sum_{u_{i} \preceq u} | \mathbb{E}[\varphi_{u}^{(h+k-l)}(t_{n}-S(u)) \mid \mathcal{G}_{i}^{(h)} ] | \in L^{1}(\Omega, \mathcal{F}, \mathbb{P}). 
\end{align}

For $A \in \mathcal{G}_{i}^{(h)}$, the definition of conditional expectation implies that
\begin{align}
\mathbb{E} \left[ \mathbf{1}_{A} \mathbb{E} \left [ \sum_{l =0}^{h} \sideset{}{^{(l)}}\sum_{u_{i} \preceq u} \varphi_{u}^{(h+k-l)}(t_{n}-S(u)) \, \, \Big | \mathcal{G}_{i}^{(h)} \right ] \right] = \mathbb{E} \left[  \mathbf{1}_{A} \sum_{l =0}^{h} \sideset{}{^{(l)}}\sum_{u_{i} \preceq u} \varphi_{u}^{(h+k-l)}(t_{n}-S(u))  \right]. 
\end{align}
\noindent By Lemma \ref{lemma4} \ref{lemma4Pro1}, \eqref{eq81} and Fubini's theorem,
\begin{align}
\mathbb{E} \left[ \mathbf{1}_{A} \mathbb{E} \left [ \sum_{l =0}^{h} \sideset{}{^{(l)}}\sum_{u_{i} \preceq u} \varphi_{u}^{(h+k-l)}(t_{n}-S(u)) \, \, \Big | \mathcal{G}_{i}^{(h)} \right ] \right]  & = \sum_{l =0}^{h} \sideset{}{^{(l)}}\sum_{u_{i} \preceq u} \mathbb{E} [  \mathbf{1}_{A} \varphi_{u}^{(h+k-l)}(t_{n}-S(u))  ] \nonumber \\
& =  \sum_{l =0}^{h} \sideset{}{^{(l)}}\sum_{u_{i} \preceq u}  \mathbb{E} [  \mathbf{1}_{A}  \mathbb{E}[\varphi_{u}^{(h+k-l)}(t_{n}-S(u)) \mid \mathcal{G}_{i}^{(h)} ]  ] \nonumber \\
& =  \mathbb{E} \left[  \mathbf{1}_{A} \sum_{l =0}^{h} \sideset{}{^{(l)}}\sum_{u_{i} \preceq u}  \mathbb{E}[\varphi_{u}^{(h+k-l)}(t_{n}-S(u)) \mid \mathcal{G}_{i}^{(h)} ]  \right].
\end{align}
\noindent This implies \eqref{eq73}. 
\end{proof}

Before proving Theorem \ref{Theo1}, we establish a series of results specifically verifying that the martingale $(M_{n,j}^{\varphi})_{j\in \mathbb{N}_{0}}$ satisfies the conditions of \cite[Corollary 3.1 on p.\ 58]{Hall1980}. 

\begin{proposition} \label{Proposition1} 
Suppose that assumptions \ref{A1}-\ref{A7} hold and  let $\varphi$ be an $h$-dependent descendant characteristic, for some $h \in \mathbb{N}_{0}$, satisfying \ref{C1}-\ref{C5}. Then, for any $\varepsilon>0$, as $n \rightarrow \infty$,
\begin{align} \label{eq71}
\sum_{i=1}^{m_{n}} \mathbb{E}[(e^{-\frac{\alpha}{2}t_{n}} \chi^{(\varphi, h)}_{u_{i}}(t_{n}-S(u_{i})))^{2} \mathbf{1}_{\{ |e^{-\frac{\alpha}{2}t_{n}} \chi^{(\varphi, h)}_{u_{i}}(t_{n}-S(u_{i}))|> \varepsilon \}} \mid \mathcal{G}_{i-1}^{(h)} ]  \xrightarrow[]{\rm p} 0. 
\end{align}
\end{proposition}

\begin{proof}
For $s >0$, define the non-negative characteristic $g^{(s)}(t) \coloneqq (\chi^{(\varphi,h)}(t))^{2} \mathbf{1}_{\{ |\chi^{(\varphi,h)}(t)| > s \}}$, for $t \in \mathbb{R}$. Since $g^{(s)}(t) \leq (\chi^{(\varphi,h)}(t))^{2}$, Lemma \ref{lemma6} \ref{lemma6Pro2} implies that $\mathbb{E}[\mathcal{Z}^{g^{(s)}}_{t}] < \infty$, for all $t \in \mathbb{R}$. By  adapting the proof of Lemma \ref{lemma6} \ref{lemma6Pro4} (using \eqref{eq65}, \ref{C4}-\ref{C5}, Lemma \ref{lemma4} and Lemma \ref{lemma16}), we deduce that the characteristic $g^{(s)}$ satisfies \ref{C1}. Then, under our assumptions \ref{A1}-\ref{A7}, \eqref{eq66} implies that, for any $\varepsilon>0$,
\begin{align} \label{eq69}
&\limsup_{n \rightarrow \infty} \sum_{i=1}^{m_{n}} \mathbb{E}[(e^{-\frac{\alpha}{2}t_{n}} \chi^{(\varphi, h)}_{u_{i}}(t_{n}-S(u_{i})))^{2} \mathbf{1}_{\{ |e^{-\frac{\alpha}{2}t_{n}} \chi^{(\varphi, h)}_{u_{i}}(t_{n}-S(u_{i}))|> \varepsilon \}} \mid \mathcal{G}_{i-1}^{(h)} ] \nonumber \\
& \quad \quad \quad \quad = \limsup_{n \rightarrow \infty} e^{-\alpha t_{n}} \sum_{i=1}^{m_{n}} \mathbb{E}[ g^{(e^{\frac{\alpha}{2} t_{n}})}_{u_{i}}(t-S(u_{i})) ]   \nonumber \\
& \quad \quad \quad \quad  \leq  \liminf_{s \rightarrow \infty} \limsup_{n \rightarrow \infty} e^{-\alpha t_{n}}\mathbb{E}[\mathcal{Z}^{g^{(s)}}_{t_{n}}] \nonumber \\
& \quad \quad \quad \quad  = \liminf_{s \rightarrow \infty} \frac{1}{\beta} \int_{\mathbb{G}} \mathbb{E}[g^{(s)}(x)] e^{-\alpha x} \ell({\rm d} x) = 0,
\end{align}
\noindent where the last equality follows by the dominated convergence theorem and Lemma \ref{lemma6} \ref{lemma6Pro1}.
\end{proof}

\begin{proposition} \label{Proposition2} 
Suppose that assumptions \ref{A1}-\ref{A7} hold and  let $\varphi$ be an $h$-dependent descendant characteristic, for some $h \in \mathbb{N}_{0}$, satisfying \ref{C1}-\ref{C5}. Then, as $n \rightarrow \infty$, 
\begin{align}
\sum_{i=1}^{m_{n}} \mathbb{E}[(e^{-\frac{\alpha}{2}t_{n}} \chi^{(\varphi, h)}_{u_{i}}(t_{n}-S(u_{i})))^{2} \mid \mathcal{G}_{i-1}^{(h)} ]  \xrightarrow[]{\rm p} \frac{W}{\beta} \int_{\mathbb{G}} \mathbb{E}[(\chi^{(\varphi, h)}(s))^{2}]  e^{-\alpha s} \ell({\rm d} s).
\end{align}
\end{proposition}

\begin{proof}
First, we show that, as $n \rightarrow \infty$, 
\begin{align} \label{eq58} 
\sum_{i=1}^{m_{n}} (e^{-\frac{\alpha}{2}t_{n}} \chi^{(\varphi, h)}_{u_{i}}(t_{n}-S(u_{i})))^{2} - \sum_{i=1}^{m_{n}} \mathbb{E}[(e^{-\frac{\alpha}{2}t_{n}} \chi^{(\varphi, h)}_{u_{i}}(t_{n}-S(u_{i})))^{2} \mid \mathcal{G}_{i-1}^{(h)} ]  \xrightarrow[]{\rm p} 0. 
\end{align}
\noindent We observe from Lemma \ref{lemma6} \ref{lemma6Pro2} that
\begin{align} \label{eq56}
\mathbb{E} \left[ \sum_{i=1}^{m_{n}} \mathbb{E}[(e^{-\frac{\alpha}{2}t_{n}} \chi^{(\varphi, h)}_{u_{i}}(t_{n}-S(u_{i})))^{2} \mid \mathcal{G}_{i-1}^{(h)} ] \right] & = \sum_{i=1}^{m_{n}} \mathbb{E}[(e^{-\frac{\alpha}{2}t_{n}} \chi^{(\varphi, h)}_{u_{i}}(t_{n}-S(u_{i})))^{2}] \nonumber \\
& \leq e^{-\alpha t_{n}}\mathbb{E}[\mathcal{Z}_{t_{n}}^{(\chi^{(\varphi, h)})^{2}}]  < \infty.
\end{align}
\noindent On the other hand, by Lemma \ref{lemma6} \ref{lemma6Pro4}, $ (\chi^{(\varphi, h)})^{2}$ satisfies \ref{C1}. Therefore, under assumptions \ref{A1}-\ref{A7}, \eqref{eq66} implies that
\begin{align} \label{eq92} 
\limsup_{n \rightarrow \infty} e^{-\alpha t_{n}}\mathbb{E}[\mathcal{Z}_{t_{n}}^{(\chi^{(\varphi, h)})^{2}}] =  \frac{1}{\beta} \int_{\mathbb{G}} \mathbb{E}[(\chi^{(\varphi, h)}(s))^{2}] e^{-\alpha s}  \ell({\rm d} s) < \infty.  
\end{align}
\noindent Then, \eqref{eq56} and \eqref{eq92} imply that
\begin{align} \label{eq57}
\limsup_{n \rightarrow \infty}\mathbb{E} \left[ \sum_{i=1}^{m_{n}} \mathbb{E}[(e^{-\frac{\alpha}{2}t_{n}} \chi^{(\varphi, h)}_{u_{i}}(t_{n}-S(u_{i})))^{2} \mid \mathcal{G}_{i-1}^{(h)} ] \right]  < \infty. 
\end{align}
\noindent In particular, \eqref{eq57} and the Markov inequality imply that, as $\lambda \rightarrow \infty$,
\begin{align} \label{eq59}
\sup_{n \in \mathbb{N}} \mathbb{P} \left( \sum_{i=1}^{m_{n}} \mathbb{E}[(e^{-\frac{\alpha}{2}t_{n}} \chi^{(\varphi, h)}_{u_{i}}(t_{n}-S(u_{i})))^{2} \mid \mathcal{G}_{i-1}^{(h)} ]  > \lambda  \right) \rightarrow 0.
\end{align}
\noindent Then, \eqref{eq58} follows from Proposition \ref{Proposition1}, \eqref{eq59} and \cite[Theorem 2.23, Chapter 2]{Hall1980}. 

Next, it follows from our choice of the sequence $(m_{n})_{n \in \mathbb{N}}$ and Lemma \ref{lemma2} \ref{Conv3} that, as $n \rightarrow \infty$,
\begin{align} \label{eq60}
\sum_{i=1}^{\infty} (e^{-\frac{\alpha}{2}t_{n}} \chi^{(\varphi, h)}_{u_{i}}(t_{n}-S(u_{i})))^{2} - \sum_{i=1}^{m_{n}} (e^{-\frac{\alpha}{2}t_{n}} \chi^{(\varphi, h)}_{u_{i}}(t_{n}-S(u_{i})))^{2} \xrightarrow[]{\rm p} 0. 
\end{align}
\noindent Recall that $(\chi^{(\varphi, h)})^{2}$ satisfies \ref{C1}. Therefore, by \cite[Theorem 6.1]{JagersN1984}, as $n \rightarrow \infty$,
\begin{align} \label{eq61}
\sum_{i=1}^{\infty} (e^{-\frac{\alpha}{2}t_{n}} \chi^{(\varphi, h)}_{u_{i}}(t_{n}-S(u_{i})))^{2}= e^{-\alpha t_{n}} \mathcal{Z}_{t_{n}}^{(\chi^{(\varphi, h)})^{2}}  \xrightarrow[]{\rm p} \frac{W}{\beta} \int_{\mathbb{G}}  \mathbb{E}[(\chi^{(\varphi, h)}(s))^{2}] e^{-\alpha s} \ell({\rm d} s).
\end{align}
\noindent Finally, our claim follows from Slutsky’s theorem \cite[Theorem 3.1 in Chapter I]{Billingsley1999} by combining \eqref{eq58}, \eqref{eq60} and \eqref{eq61}. To obtain \eqref{eq61} in the lattice-case, one can verify that \cite[Theorem 6.1]{JagersN1984} remains valid with the necessary modifications {\sl mutatis mutandis} (see also \cite[Theorem 5.1]{Dimitris2000}).
\end{proof}

We have now all the ingredients to prove Theorem \ref{Theo1}. 

\begin{proof}[Proof of Theorem \ref{Theo1}]
Let $(t_{n})_{n \in \mathbb{N}}$ be an arbitrary increasing sequence in $\mathbb{G}$ such that $\lim_{n \rightarrow \infty} t_{n} = \infty$. By Lemma \ref{lemma2}, we know that there exists an increasing sequence of non-negative integers $(m_{n})_{n \in \mathbb{N}}$ such that \ref{Conv1} and \ref{Conv3} hold. By Propositions \ref{Proposition1} and \ref{Proposition2}, the zero-mean, square-integrable martingale array $(M_{n,j}^{\varphi}, 0 \leq j \leq m_{n}, n \in \mathbb{N})$ satisfies the conditions of the martingale central limit theorem \cite[Corollary 3.1 on p.\ 58]{Hall1980} and thus,
\begin{align} \label{eq18} 
M_{n,m_{n}}^{\varphi} \xrightarrow[]{\rm st} \left( \frac{W}{\beta} \int_{\mathbb{G}}  \mathbb{E}[(\chi^{(\varphi, h)}(s))^{2}] e^{-\alpha s} \ell({\rm d} s) \right)^{1/2} \mathcal{N}, \quad \text{as} \quad n \rightarrow \infty,
\end{align}
\noindent where $\mathcal{N}$ is a standard normal random variable.

Observe that, by  Lemma \ref{lemma6} \ref{lemma6Pro5} and \eqref{eq178}, 
\begin{align}
|e^{-\frac{\alpha}{2}t_{n}} \mathcal{Z}_{t_{n}}^{\varphi-\mathbb{E}[\varphi]} - M_{n,m_{n}}^{\varphi}| \leq  e^{-\frac{\alpha}{2}t_{n}}| \mathcal{Z}_{t_{n}}^{\chi^{(\varphi,h)}}-Y^{\varphi}_{m_{n}}(t_{n})| + e^{-\frac{\alpha}{2}t_{n}} |\varpi^{(\varphi,h)}(t_{n})|
\end{align}
\noindent Then, it follows from Lemma \ref{lemma2} \ref{Conv1} and Lemma \ref{lemma5} that
\begin{align} \label{eq63} 
e^{-\frac{\alpha}{2}t_{n}} \mathcal{Z}_{t_{n}}^{\varphi-\mathbb{E}[\varphi]} - M_{n,m_{n}}^{\varphi} \xrightarrow[]{\rm p} 0, \quad \text{as} \quad n \rightarrow \infty.
\end{align}
\noindent Therefore, Slutsky’s theorem \cite[Theorem I.3.1]{Billingsley1999}, \eqref{eq18} and \eqref{eq63} imply that, as $n \rightarrow \infty$,
\begin{align} 
e^{-\frac{\alpha}{2}t_{n}} \mathcal{Z}_{t_{n}}^{\varphi-\mathbb{E}[\varphi]}  \xrightarrow[]{\rm st} \left( \frac{W}{\beta} \int_{\mathbb{G}}  \mathbb{E}[(\chi^{(\varphi, h)}(s))^{2}] e^{-\alpha s} \ell({\rm d} s) \right)^{1/2} \mathcal{N}.
\end{align}

Finally, we justify that indeed the convergence \eqref{eq18} is stable and that limiting random variable $\mathcal{N}$ is independent of $\mathcal{F}$. It follows from the proof of the preceding \cite[Lemma 3.1, cf. Eq.\ (3.15)]{Hall1980}, that is, for any $B \in \mathcal{F}$, we have $\mathbb{E}[ e^{{\rm i} z M_{n,m_{n}}^{\varphi}} \mathbf{1}_{B}] = \mathbb{E}[e^{-W \frac{z^{2}}{2}} \mathbf{1}_{B}]$, as $n \rightarrow \infty$, for every $z \in \mathbb{R}$. The latter is equivalent, by a standard approximation argument, to say that for any $\mathcal{F}$-measurable random variable $\mathcal{Y}$ and $z,z^{\prime} \in \mathbb{R}$,
\begin{align}
\mathbb{E}\left[ e^{{\rm i} z M_{n,m_{n}}^{\varphi}} e^{{\rm i} z^{\prime} \mathcal{Y}} \right] \rightarrow \mathbb{E}\left[e^{-W \frac{z^{2}}{2}} e^{{\rm i} z^{\prime} \mathcal{Y}} \right] = \mathbb{E}\left[e^{{\rm i} z \sqrt{W} \mathcal{N}} e^{{\rm i} z^{\prime} \mathcal{Y}} \right], \quad \text{as} \quad n \rightarrow \infty,
\end{align}
\noindent for a standard normal variable $\mathcal{N}$ independent of $(W, \mathcal{Y})$. This implies the stable convergence by \cite[Proposition 1]{Aldous1978}.
\end{proof}

%%%%%%%%%%%%%%%%%%%%%%%%%%%%%%%%%%%%%%%%%%%%%%%%%%%%%%%%%%%%%%%%%%%%%%%
\section{Proof of our main result} \label{ProofMainTheandCoro}
%%%%%%%%%%%%%%%%%%%%%%%%%%%%%%%%%%%%%%%%%%%%%%%%%%%%%%%%%%%%%%%%%%%%%%%

This section is dedicated to the proof of our main result, Theorem \ref{Theo2}. We first introduce essential characteristics in Sections \ref{SecMatrixCharac} and \ref{DeterCharac}, and show that they fulfil conditions \ref{C1}-\ref{C5}. Following this, Section \ref{ProofofMainTheoI} provides the proof of Theorem \ref{Theo2}. Lastly, Section \ref{ProofofCorollaryI} contains the proof of Corollary \ref{corollary5}.

Throughout this section, we always assume that \ref{A1}-\ref{A8} are satisfied.

%%%%%%%%%%%%%%%%%%%%%%%%%%%%%%%%%%%%%%%%%%%%%%%%%%%%%%%%%%%%%%%%%%%%%%%
\subsection{Some special individual matrix-valued characteristics} \label{SecMatrixCharac}
%%%%%%%%%%%%%%%%%%%%%%%%%%%%%%%%%%%%%%%%%%%%%%%%%%%%%%%%%%%%%%%%%%%%%%%
For any $s \in \mathbb{R}$ and $\gamma \in \mathbb{C}$ we define the following lower triangular $k \times k$ matrix:

\begin{align}
\exp (\gamma, s, k) \coloneqq e^{\gamma s} \times \left(\begin{array}{ccccc}
1 & 0 & 0 & \cdots & 0  \\
s & 1 & 0 & \cdots & 0 \\
s^{2} & 2 s & 1 & \cdots & 0 \\
\vdots & \vdots & \vdots & \ddots & \vdots \\
s^{k-1} &  \binom{k-1}{1} s^{k-2} & \binom{k-1}{2} s^{k-3} & \cdots & 1
\end{array}\right).
\end{align}

\noindent The $(i, j)$-th entry of the matrix is $e^{\gamma s}\binom{i-1}{j-1} s^{i-j}$, for $i, j=1, \dots, k$, where $\binom{i-1}{j-1}=0$ for $j>i$. Denote by $\| \cdot \|$ the operator norm and by $\| \cdot \|_{\mathrm{HS}}$ the Hilbert-Schmidt norm. Then, for every $\delta>0$, we have the following bound (see \cite[(4.17)]{Iksanov2021}), 
\begin{align} \label{eq75}
\| \exp (\gamma, s, k) \| \leq \|  \exp (\gamma, s, k)\|_{\mathrm{HS}} \leq C^{\prime}(1+|s|)^{k-1} e^{\mathrm{Re}(\gamma) s} \leq C e^{\mathrm{Re}(\gamma) s + \delta |s|},
\end{align}
\noindent for some constant $C^{\prime} >0$ depending on $k$ only and another constant $C>0$ depending on $k$ and $\delta$. Note that, for any $s, t \in \mathbb{R}$ and $\gamma \in \mathbb{C}$,
\begin{align} \label{eq112}
\exp (\gamma, s, k) \cdot \exp (\gamma, t, k)=\exp (\gamma, s+t, k).
\end{align}
For a vector $a$, we write $a^{\top}$ for its transpose. Further, we write $\mathrm{e}_{1}, \mathrm{e}_{2}, \dots$ for the canonical base vectors in Euclidean space. For notational simplicity, we omit explicit dimension specifications, relying on context for clarity. (Formally, all Euclidean spaces can be embedded in a suitable infinite-dimensional space, such as $\ell^{2}$). Then, for instance,
\begin{align} \label{eq113}
\exp (\gamma, s, k) \cdot \mathrm{e}_{1}=e^{\gamma s}\left(\begin{array}{c}
1 \\
s \\
s^{2} \\
\vdots \\
s^{k-1}
\end{array}\right).
\end{align}

Recall from Section \ref{AssumpCMJ} that $\Lambda$ denotes the set of $\lambda \in \mathbb{C}$ satisfying $\hat{\mu}(\lambda)=1$ such that ${\rm Re}(\lambda) > \alpha /2$.  Furthermore, $\partial \Lambda$ denotes the subset of $\lambda \in \mathbb{C}$ satisfying $\hat{\mu}(\lambda)-1=0$ and ${\rm Re}(\lambda) = \alpha /2$. Recall that in the lattice case, $\Lambda$ denotes the set of roots such that ${\rm Re}(\lambda)>\frac{\alpha}{2}$ and $\mathrm{Im}(\lambda) \in (-\pi, \pi]$, while $\partial \Lambda$ denotes the set of roots on ${\rm Re}(\lambda)=\frac{\alpha}{2}$ such that $\mathrm{Im}(\lambda) \in (-\pi, \pi]$. Here, $\alpha >0$ is the Malthusian parameter defined in \ref{A4}. Recall also that we have assumed the set $\Lambda_{\geq} = \Lambda \cup \partial\Lambda$ of roots to be finite, that is, \ref{A8}. For a root $\lambda \in \Lambda_{\geq}$ of multiplicity $k \coloneqq k(\lambda) \in \mathbb{N}$ and $n \in \mathbb{N}_{0}$, we denote by $Z_{n}(\lambda, k)$ the following random matrix:
\begin{align} \label{eq149}
Z_{n}(\lambda, k) \coloneqq \sum_{|u|=n} \exp (\lambda,-S(u), k).
\end{align}
\noindent We also set $Z_{n}(\lambda) \coloneqq Z_{n}(\lambda,1)$. %In particular, by \eqref{eq74}, \ref{A7} and \eqref{eq31}, $\hat{\mu}(\vartheta) = \mathbb{E}[Z_{1}(\vartheta)]<\infty$, where $\vartheta \in (0, \alpha/2)$ is as in \ref{A7}. Indeed, 
Since ${\rm Re}(\lambda) \geq \frac{\alpha}{2}$, by \eqref{eq75} and \ref{A7}, $\mathbb{E}[\|Z_{1}(\lambda,k) \|^{2}] < \infty$. By \cite[(5.4)]{Iksanov2021}, we also have that
\begin{align} \label{eq76}
\mathbb{E}[Z_{1}(\lambda, k)] = \mathbb{E} \left[ \sum_{i=1}^{N} \exp (\lambda,-X_{i}, k) \right] = I_{k},
\end{align}
\noindent where $I_{k}$ is the $k \times k$ identity matrix. Define the random matrix on the probability space $(\Omega_{\varnothing}, \mathcal{F}_{\varnothing}, P_{\varnothing})$, 
\begin{align} \label{eq105}
Y \coloneqq Z_{1}(\lambda, k) -I_{k} = \int_{[0, \infty)} \exp(\lambda,-x, k) \xi({\rm d} x) - I_{k}.
\end{align}
\noindent By \eqref{eq76}, $\mathbb{E}[Y] = \mathbb{E}[Y_{\varnothing}]$, where $Y_{\varnothing} = Y \circ \pi_{\varnothing}$, is the $k \times k$ zero matrix. By \eqref{eq75} and \ref{A7}, $\mathbb{E}[\|Y \|^{2}] < \infty$. Recall that $\pi_{u}: \Omega \rightarrow  \Omega_{u}$ is the projection onto the life space of individual $u \in \mathcal{I}$. Then, $Y_{u} = Y \circ \pi_{u}$. 

We define two (matrix-valued individual) characteristics $\phi_{\lambda}$ and $\chi_{\lambda}$. For $\lambda \in \Lambda_{\geq}$ of multiplicity $k \in \mathbb{N}$, and $t \in \mathbb{R}$, we set
\begin{align} \label{eq77}
\phi_{\lambda}(t) \coloneqq \sum_{i=1}^{N} \mathbf{1}_{[0, X_{i})}(t) \exp (\lambda,t-X_{i}, k).
\end{align}
\noindent For $\lambda \in \Lambda$ of multiplicity $k \in \mathbb{N}$, and $t \in \mathbb{R}$, we set
\begin{align} \label{eq78}
\chi_{\lambda}(t) \coloneqq \mathbf{1}_{(-\infty, 0)}(t) \exp (\lambda,t, k) Y.
\end{align}
\noindent For $u \in \mathcal{I}$, $\phi_{\lambda,u} \coloneqq \phi_{\lambda} \circ \pi_{u}$ and $\chi_{\lambda,u} \coloneqq \chi_{\lambda} \circ \pi_{u}$.

In Section \ref{CMJclassic}, we defined characteristics as (random) real-valued c\`adl\`ag functions. Here, we employ a straightforward extension to complex-valued characteristic functions by considering their real and imaginary components separately. 

Recall the definition of the measure $\nu^{(h)}$, for $h \in \mathbb{N}_{0}$, given in \eqref{eq62}. 

\begin{lemma} \label{lemma7}
Suppose that assumptions \ref{A1}-\ref{A8} hold. Then,
\begin{enumerate}[label=(\textbf{\roman*})]
\item For $\lambda \in \Lambda_{\geq}$ of multiplicity $k \in \mathbb{N}$ and fix vectors $a,b \in \mathbb{R}^{k}$, the characteristic $a^{\top} \phi_{\lambda} b$ has c\`adl\`ag paths and satisfies \ref{C1}-\ref{C3}. Moreover, for all $h \in \mathbb{N}_{0}$ and $\varepsilon >0$, \label{lemma7Pro1}
\begin{align} \label{eq101}
\mathbb{E} \left[ \sum_{i=0}^{h} \sum_{|u|=i} \sup_{|s-t|\leq \varepsilon}|a^{\top}\phi_{\lambda, u}(s-S(u))b| \mathbf{1}_{(-\infty,0)}(s-S(u)) \right] < \infty
\end{align}
\noindent and 
\begin{align} \label{eq142}
\left(\sum_{i=0}^{h} \left( \sum_{|u|=i}|a^{\top}\phi_{\lambda, u}(s-S(u))b|  \right)^{2}\right)_{|s-t|\leq \varepsilon} \quad \text{is uniformly integrable}.
\end{align}

\item For $\lambda \in \Lambda$ of multiplicity $k \in \mathbb{N}$ and fix vectors $a,b \in \mathbb{R}^{k}$, the characteristic $a^{\top} \chi_{\lambda} b$ has c\`adl\`ag paths and satisfies \ref{C1}-\ref{C3}. Moreover, for all $h \in \mathbb{N}_{0}$ and $\varepsilon >0$,  \label{lemma7Pro2}
\begin{align} \label{eq102}
\mathbb{E} \left[ \sum_{i=0}^{h} \sum_{|u|=i} \sup_{|s-t|\leq \varepsilon}|a^{\top}\chi_{\lambda, u}(s-S(u))b| \mathbf{1}_{(-\infty,0)}(s-S(u)) \right] < \infty
\end{align}
\noindent and 
\begin{align} \label{eq143}
\left(\sum_{i=0}^{h} \left( \sum_{|u|=i}|a^{\top}\chi_{\lambda, u}(s-S(u))b|  \right)^{2}\right)_{|s-t|\leq \varepsilon} \quad \text{is uniformly integrable}.
\end{align}
\end{enumerate}
\end{lemma}

\begin{proof}
As shown in \cite[Lemma 5.1]{Iksanov2021}, the characteristics $a^{\top} \phi_{\lambda} b$ and $a^{\top} \chi_{\lambda} b$ have c\`adl\`ag paths and satisfy \ref{C1}-\ref{C3}. We observe that \cite[assumption (A3)]{Iksanov2021} is implied by our assumption \ref{A7} (see \cite[remark 2.1]{Iksanov2021}). Thus, we are left to prove that they also satisfy \eqref{eq101}-\eqref{eq142} and \eqref{eq102}-\eqref{eq143}, respectively. Without of loss of generality, we may assume that $|a|, |b| \leq 1$. 

First, we prove \eqref{eq101}-\eqref{eq142}. Let $\vartheta \in (0, \alpha/2)$ be as defined in \ref{A7}. Recall that ${\rm Re}(\lambda) \geq \alpha/2$ because $\lambda \in \Lambda_{\geq}$. By \eqref{eq75} (with $\delta = {\rm Re}(\lambda) - \vartheta$), there exists a constant $C_{1}>0$ such that, for $t \in \mathbb{R}$,
\begin{align} \label{eq144}
|a^{\top}\phi_{\lambda}(t)b| \leq \sum_{i=1}^{N} \mathbf{1}_{[0, X_{i})}(t) |a^{\top}\exp(\lambda, t-X_{i}, k)b| \leq C_{1} \mathbf{1}_{[0, \infty)}(t) \sum_{i=1}^{N} e^{\vartheta (t-X_{i}) }.
\end{align}
\noindent Recall that \ref{A7} implies that $\mathbb{E} [ (\sum_{i=1}^{N} e^{-\vartheta X_{i}} )^{2} ] < \infty$. Note that, under our assumptions \ref{A1}-\ref{A7}, \eqref{eq144} implies that the characteristic $|x^{\top}\phi_{\lambda}y|$ satisfies the conditions of Lemma \ref{lemma12}  \ref{lemma12Pro2} (with $\delta_{2}= \vartheta$) and \ref{lemma12Pro1} (with $\delta_{1} = \delta_{2} = \vartheta$). This implies \eqref{eq101}-\eqref{eq142}.

Next, we prove \eqref{eq102}-\eqref{eq143}. Recall that ${\rm Re}(\lambda) > \alpha/2$ because $\lambda \in \Lambda$. In particular, there exists $\gamma \in (\alpha/2, {\rm Re}(\lambda))$. Moreover, by \eqref{eq75} (with $\delta = {\rm Re}(\lambda) - \gamma$), there exists a constant $C_{2}>0$ such that, for $t \in \mathbb{R}$,
\begin{align} \label{eq148}
|a^{\top}\chi_{\lambda}(t)b| \leq \mathbf{1}_{(-\infty, 0)}(t) |a^{\top}\exp(\lambda, t, k)	Yb| \leq C_{2} \mathbf{1}_{(-\infty, 0)}(t) \|Y \| e^{\gamma t},
\end{align}
\noindent where $Y$ is defined in \eqref{eq105}. Recall that, by \eqref{eq75} and \ref{A7}, $\mathbb{E}[\|Y \|^{2}] < \infty$. Note that, under our assumptions \ref{A1}-\ref{A7}, \eqref{eq148} implies that the characteristic $|x^{\top}\phi_{\lambda}y|$ satisfies the conditions of Lemma \ref{lemma12}  \ref{lemma12Pro2} (with $\delta_{2}= \gamma$) and \ref{lemma12Pro1} (with $\delta_{1} = \delta_{2} = \gamma$), noting that $\mathrm{Re}(\lambda) \leq \alpha$. This implies \eqref{eq102}-\eqref{eq143}.  
\end{proof}

%%%%%%%%%%%%%%%%%%%%%%%%%%%%%%%%%%%%%%%%%%%%%%%%%%%%%%%%%%%%%%%%%%%%%%%
\subsection{Deterministic individual characteristics} \label{DeterCharac}
%%%%%%%%%%%%%%%%%%%%%%%%%%%%%%%%%%%%%%%%%%%%%%%%%%%%%%%%%%%%%%%%%%%%%%%

Let $f: \mathbb{R} \rightarrow \mathbb{R}$ be a deterministic c\`adl\`ag function (that is, $f$ is a deterministic individual characteristic). Suppose that $t \in \mathbb{R} \mapsto f(t) e^{-\alpha t}$ is directly Riemann integrable and define the individual characteristic,
\begin{align} \label{eq88}
\chi_{f}(t) \coloneqq f \ast \xi \ast \nu^{(\infty)}(t)- f \ast \mu \ast \nu^{(\infty)}(t)=m^{f} \ast \xi(t)-m^{f} \ast \mu(t), \quad \text{for} \quad t \in \mathbb{R},
\end{align}
\noindent where $\nu^{(\infty)}$ is defined in \eqref{eq62} and $\ast$ denotes Lebesgue-Stieltjes convolution. Assumptions \ref{A1}-\ref{A7} imply, by \cite[Lemma 6.4 (a)]{Iksanov2021}, that $\chi_{f}$ is well-defined and has almost surely c\`adl\`ag paths. For $u \in \mathcal{I}$, $\chi_{f,u} \coloneqq \chi_{f} \circ \pi_{u}$.

\begin{lemma} \label{lemma8}
Suppose that assumptions \ref{A1}-\ref{A7} hold. Let $f: \mathbb{R} \rightarrow \mathbb{R}$ be a deterministic c\`adl\`ag function such that $t \in \mathbb{R} \mapsto f(t) e^{-\alpha t}$ is directly Riemann integrable. Suppose further that $t \in \mathbb{R} \mapsto m_{t}^{f} e^{-\frac{\alpha}{2}t}(1+t^{2})$ is bounded. Then,  
\begin{enumerate}[label=(\textbf{\roman*})]
\item The characteristic  $\chi_{f}$ satisfies \ref{C1}-\ref{C3}. Moreover, for all $h \in \mathbb{N}_{0}$ and $\varepsilon>0$, \label{lemma8Pro1}
\begin{align} \label{eq104}
\mathbb{E} \left[ \sum_{i=0}^{h} \sum_{|u|=i} \sup_{|s-t|\leq \varepsilon} |\chi_{f,u}(s-S(u))| \mathbf{1}_{(-\infty,0)}(s-S(u)) \right] < \infty
\end{align}
\noindent and 
\begin{align}  \label{eq154}
\left(\sum_{i=0}^{h} \left( \sum_{|u|=i} |\chi_{f,u}(s-S(u))| \right)^{2}\right)_{|s-t|\leq \varepsilon} \quad \text{is uniformly integrable}.
\end{align}

\item For every $t \in \mathbb{R}$, the series $\mathcal{Z}_{t}^{\chi_{f}}$ converges unconditionally in $L^{1}(\Omega, \mathcal{F}, \mathbb{P})$ and almost surely over every admissible ordering of $\mathcal{I}$. Moreover, for all $t \in \mathbb{R}$, $\mathbb{E}[\chi_{f}(t)] = 0$ and  \label{lemma8Pro2}
\begin{align}
\mathcal{Z}_{t}^{f} - m^{f}_{t} \stackrel{a.s.}{=}  \mathcal{Z}_{t}^{\chi_{f}}. 
\end{align}

\item For $n \in \mathbb{N}_{0}$ and $t \in \mathbb{R}$, \label{lemma8Pro3}
\begin{align}
\sum_{0\leq |u| \leq n} \chi_{f, u}(t - S(u)) \stackrel{a.s.}{=}  \sum_{0\leq |u| \leq n} f(t - S(u)) - f \ast \nu^{(\infty)}(t) + \sum_{|u| = n+1} f \ast \nu^{(\infty)}(t - S(u)). 
\end{align}
\end{enumerate}
\end{lemma} 

\begin{proof}
First, we prove \ref{lemma8Pro1}. By \cite[Lemma 6.4]{Iksanov2021}, the characteristic $\chi_{f}$ satisfies \ref{C1}-\ref{C3}. Recall that \cite[assumption (A3)]{Iksanov2021} is implied by our assumption \ref{A7} (see \cite[remark 2.1]{Iksanov2021}). Thus, we are left to prove that it also satisfies \eqref{eq104} and \eqref{eq154}. 

By \ref{A7}, \eqref{eq31}, \eqref{eq88} and our assumptions on $m^{f}$, note that, for $s \in \mathbb{R}$, $|\chi_{f}(s)| \leq C  e^{\frac{\alpha}{2}s} (\hat{\xi}(\alpha/2) + \hat{\mu}(\alpha/2) )$, where $C = \sup_{s \in \mathbb{R}} |m_{s}^{f}| e^{-\frac{\alpha}{2}s}(1+s^{2})$. Then, \eqref{eq104} and \eqref{eq154} follows directly from Lemma \ref{lemma12} \ref{lemma12Pro2}-\ref{lemma12Pro1}. Specifically, under assumptions \ref{A1}-\ref{A7} and the preceding remarks, we observe that the characteristic $\chi_{f}$ satisfies the conditions of Lemma \ref{lemma12} \ref{lemma12Pro2} (with $\delta_{2} = \alpha/2$) and \ref{lemma12Pro1} (with $\delta_{1} = \delta_{2} = \alpha/2$). This concludes with the proof of \ref{lemma8Pro1}. 

The first claim in \ref{lemma8Pro2} follows from  \ref{lemma8Pro1} and Proposition \ref{Proposition3}. The second part of \ref{lemma8Pro2} was proved in \cite[Lemma 6.4 (c)]{Iksanov2021}. Finally, \ref{lemma8Pro3} was established within the proof of \cite[Lemma 6.4 (c)]{Iksanov2021}. 
\end{proof}

%%%%%%%%%%%%%%%%%%%%%%%%%%%%%%%%%%%%%%%%%%%%%%%%%%%%%%%%%%%%%%%%%%%%%%%
\subsection{Proof of Theorem \ref{Theo2}} \label{ProofofMainTheoI}
%%%%%%%%%%%%%%%%%%%%%%%%%%%%%%%%%%%%%%%%%%%%%%%%%%%%%%%%%%%%%%%%%%%%%%%

This section proves our main Theorem \ref{Theo2}. We use Theorem \ref{Theo1} to prove a limit theorem for not centred $h$-dependent descendant characteristic $\varphi$, when $m_{t}^{\mathbb{E}[\varphi]} = \mathbb{E}[\mathcal{Z}_{t}^{\mathbb{E}[\varphi]}]$ grows slowly as $|t| \rightarrow \infty$, and then reduce the general case to this. Recall the definition of the characteristic $\chi^{(\varphi,h)}$ given in \eqref{eq53}. 
	
\begin{theorem} \label{Theo3}
Suppose that assumptions \ref{A1}-\ref{A7} hold and  let $\varphi$ be an $h$-dependent descendant characteristic, for some $h \in \mathbb{N}_{0}$, satisfying \ref{C1}-\ref{C5}. Suppose further that $t \in \mathbb{R} \mapsto m_{t}^{\mathbb{E}[\varphi]} e^{-\frac{\alpha}{2}t}(1+t^{2})$ is bounded. Then, as $t \rightarrow \infty$, $t \in \mathbb{G}$,
\begin{align} \label{eq110}
 e^{-\frac{\alpha}{2}t} \mathcal{Z}_{t}^{\varphi}  \xrightarrow[]{\rm st} \sigma_{\varphi} \left( \frac{W}{\beta} \right)^{1/2} \mathcal{N},
\end{align}
\noindent where $\beta \in (0, \infty)$ is defined in \eqref{eq3}, $W$ is the limit of Nerman's martingale, $\mathcal{N}$ is a standard normal random variable independent of $\mathcal{F}$ (and thus of $W$) and
\begin{align}
\sigma_{\varphi}^{2} \coloneqq  \int_{\mathbb{G}}  \mathbb{E}[(\chi^{(\varphi+m^{\mathbb{E}[\varphi]}\ast \xi,h)}(s))^{2}] e^{-\alpha s} \ell( {\rm d} s).
\end{align}
\noindent Moreover, if $\sigma_{\varphi}^{2}=0$ then
\begin{align} \label{eq171} 
t \in \mathbb{R} \mapsto \sum_{0\leq |u| \leq h-1}\varphi_{u}^{(h-|u|)}(t-S(u))  + \sum_{|u| = h} m^{\mathbb{E}[\varphi]}_{t - S(u)}
\end{align}
\noindent is a version of the process $\mathcal{Z}^{\varphi}$. 
(If $h=0$, the sum $\sum_{0\leq |u| \leq h-1}$ in \eqref{eq171} is defined to be $0$.)
\end{theorem}

\begin{proof}
By Lemma \ref{lemma8} (with $f = \mathbb{E}[\varphi]$), we can write,  for $t \in \mathbb{R}$,
\begin{align}
\mathcal{Z}_{t}^{\varphi}  =\mathcal{Z}_{t}^{\varphi-\mathbb{E}[\varphi]}+\mathcal{Z}_{t}^{\mathbb{E}[\varphi]} = \mathcal{Z}_{t}^{\varphi-\mathbb{E}[\varphi]}+\mathcal{Z}_{t}^{\chi \mathbb{E}[\varphi]}+m_{t}^{\mathbb{E}[\varphi]}= \mathcal{Z}_{t}^{\varphi-\mathbb{E}[\varphi]+\chi_{\mathbb{E}[\varphi]}}+m_{t}^{\mathbb{E}[\varphi]}, \quad \text{almost surely}.
\end{align}
\noindent By our assumptions, $e^{-\alpha t / 2} m_{t}^{\mathbb{E}[\varphi]} \rightarrow 0$, as $t \rightarrow \infty$, $t \in \mathbb{G}$. Therefore, it remains to prove that $e^{-\frac{\alpha}{2} t} \mathcal{Z}_{t}^{\varphi-\mathbb{E}[\varphi]+\chi_{\mathbb{E}[\varphi]}}$ converges to the stated limit. To this end, we apply Theorem \ref{Theo1} to the branching process counted with characteristic $\varphi-\mathbb{E}[\varphi]+\chi_{\mathbb{E}[\varphi]}$. Recall from \eqref{eq88} (with $f = \mathbb{E}[\varphi]$), that $\chi_{\mathbb{E}[\varphi]} = m^{\mathbb{E}[\varphi]}\ast \xi-\mathbb{E}[m^{\mathbb{E}[\varphi]}\ast \xi]$. By \ref{C1} and Lemma \ref{lemma8}, $\mathbb{E}[\varphi-\mathbb{E}[\varphi]+\chi_{\mathbb{E}[\varphi]}] =0$. Observe that $\chi_{\mathbb{E}[\varphi]}$ is an individual characteristic, and in particular, an $h$-dependent descendant characteristic. Similarly, $\mathbb{E}[\varphi]$ is an $h$-dependent descendant characteristic. Consequently, $\varphi-\mathbb{E}[\varphi]+\chi_{\mathbb{E}[\varphi]}$ is an $h$-dependent descendant characteristic. According to Lemma \ref{lemma8} \ref{lemma8Pro1}, $\chi_{\mathbb{E}[\varphi]}$ satisfies \ref{C1}-\ref{C5}.  The same holds for $\mathbb{E}[\varphi]$ and $\varphi$. Then, Remarks \ref{Remark1} and \ref{Remark6} imply that $\varphi-\mathbb{E}[\varphi]+\chi_{\mathbb{E}[\varphi]}$ is an $h$-dependent descendant characteristic satisfying \ref{C1}-\ref{C5}. Since $\varphi-\mathbb{E}[\varphi]+\chi_{\mathbb{E}[\varphi]} = \varphi+m^{\mathbb{E}[\varphi]}\ast \xi-\mathbb{E}[\varphi+m^{\mathbb{E}[\varphi]}\ast \xi]$, Remark \ref{Remark2} implies that $\chi^{(\varphi-\mathbb{E}[\varphi]+\chi_{\mathbb{E}[\varphi]},h)} = \chi^{(\varphi+m^{\mathbb{E}[\varphi]}\ast \xi,h)}$. Therefore, Theorem \ref{Theo1} implies \eqref{eq110}.  

If $\sigma_{\varphi}^{2}=0$, then $\chi^{(\varphi+m^{\mathbb{E}[\varphi]}\ast \xi,h)}(s) \stackrel{a.s.}{=} 0$ for $\ell$-almost every $s \in \mathbb{G}$. On the other hand, Lemma \ref{lemma6} \ref{lemma6Pro1b} implies that that $\chi^{(\varphi+m^{\mathbb{E}[\varphi]}\ast \xi,h)}$ has almost surely c\`adl\`ag paths, which in turn implies that, except on a $\mathbb{P}$-null set, $\chi^{(\varphi+m^{\mathbb{E}[\varphi]}\ast \xi,h)}(s) = 0$ for every $s \in \mathbb{G}$. Therefore, by Lemma \ref{lemma6} \ref{lemma6Pro5} and Lemma \ref{lemma8} \ref{lemma8Pro2} (with $f = \mathbb{E}[\varphi]$), for every fixed $t \in \mathbb{G}$,
\begin{align} \label{eq168}
\mathcal{Z}_{t}^{\varphi} & \stackrel{a.s.}{=} \mathcal{Z}_{t}^{\mathbb{E}[\varphi]} - \mathcal{Z}_{t}^{\chi_{\mathbb{E}[\varphi]}} + \mathcal{Z}_{t}^{\chi^{(\varphi+m^{\mathbb{E}[\varphi]}\ast \xi,h)}} \nonumber \\
& \quad \quad \quad  +  \sum_{0 \leq |u| \leq h-1}  ( \varphi_{u}^{(h-|u|)}(t-S(u)) - \varphi_{u}^{(0)}(t-S(u))) \nonumber \\
& \quad \quad \quad  + \sum_{0 \leq |u| \leq h-1}  ( \chi_{\mathbb{E}[\varphi], u}^{(h-|u|)}(t-S(u)) - \chi_{\mathbb{E}[\varphi], u}^{(0)}(t-S(u)))
\nonumber \\ 
& \stackrel{a.s.}{=} m_{t}^{\mathbb{E}[\varphi]} +  \sum_{0 \leq |u| \leq h-1}  ( \varphi_{u}^{(h-|u|)}(t-S(u)) - \varphi_{u}^{(0)}(t-S(u))) \nonumber \\
& \quad \quad \quad  + \sum_{0 \leq |u| \leq h-1}  ( \chi_{\mathbb{E}[\varphi], u}^{(h-|u|)}(t-S(u)) - \chi_{\mathbb{E}[\varphi], u}^{(0)}(t-S(u))).
\end{align}
\noindent Lemma \ref{lemma4} \ref{lemma4Pro1} justifies the manipulations in the preceding chain of equalities. From \eqref{eq27} and \eqref{eq88} (where $\varphi$ is replaced by $\chi_{\mathbb{E}[\varphi]}$), we observe that $\chi_{\mathbb{E}[\varphi]}^{(0)}(t) \stackrel{a.s.}{=} \mathbb{E}[\chi_{\mathbb{E}[\varphi]}(t)] =  0$, while $\chi_{\mathbb{E}[\varphi]}^{(h-i)}(t) \stackrel{a.s.}{=} \chi_{\mathbb{E}[\varphi]}(t)$, for $i=1, \dots, h$ (if $h \geq 1$). Then, 
\begin{align}  \label{eq169}
\sum_{0 \leq |u| \leq h-1}  ( \chi_{\mathbb{E}[\varphi], u}^{(h-|u|)}(t-S(u)) - \chi_{\mathbb{E}[\varphi], u}^{(0)}(t-S(u))) \stackrel{a.s.}{=} \sum_{0 \leq |u| \leq h-1}  \chi_{\mathbb{E}[\varphi], u}(t-S(u)). 
\end{align}
\noindent Recall from \eqref{eq27} that $\varphi^{(0)}(t) \stackrel{a.s.}{=} \mathbb{E}[\varphi](t)$, while from \eqref{eq109}, $m_{t}^{\mathbb{E}[\varphi]} = \mathbb{E}[\varphi] \ast \nu^{(\infty)}(t)$. Consequently, Lemma \ref{lemma8} \ref{lemma8Pro3}, \eqref{eq168} and \eqref{eq169} imply that
\begin{align}
\mathcal{Z}_{t}^{\varphi} & \stackrel{a.s.}{=} m_{t}^{\mathbb{E}[\varphi]}  +  \sum_{0\leq |u| \leq h-1} \varphi^{(0)}_{u}(t - S(u)) - \mathbb{E}[\varphi] \ast \nu^{(\infty)}(t) + \sum_{|u| = h} \mathbb{E}[\varphi] \ast \nu^{(\infty)}(t - S(u)) \nonumber \\
& \quad \quad \quad  + \sum_{0 \leq |u| \leq h-1}  ( \varphi_{u}^{(h-|u|)}(t-S(u)) - \varphi_{u}^{(0)}(t-S(u))) \nonumber \\
 & \stackrel{a.s.}{=} \sum_{0\leq |u| \leq h-1}\varphi_{u}^{(h-|u|)}(t-S(u))  + \sum_{|u| = h} m^{\mathbb{E}[\varphi]}_{t - S(u)}.
\end{align}
\noindent Again, Lemma \ref{lemma4} \ref{lemma4Pro1} justifies the manipulations in the preceding chain of equalities. This concludes our proof.
\end{proof}

\begin{remark}
If $\varphi$ is a $0$-dependent descendant characteristic (i.e., an individual characteristic), then within the framework of Theorem \ref{Theo3},
\begin{align}
\sigma^{2}_{\varphi} = \int_{\mathbb{G}} \mathrm{Var}(\varphi(s) + m^{\mathbb{E}[\varphi]}\ast \xi(s)) e^{-\alpha s}  \ell({\rm d} s).
\end{align}
\noindent (Recall also Remark \ref{Remark8}). Moreover, if $\sigma^{2}_{\varphi} =0$, then from \eqref{eq171} it follows that $t \in \mathbb{R} \mapsto m_{t}^{\mathbb{E}[\varphi]}$ is a version of $\mathcal{Z}^{\varphi}$. This can be compared with \cite[Theorem 6.6]{Iksanov2021}.  
\end{remark}

We now prove Theorem \ref{Theo2}.

\begin{proof}[Proof of Theorem \ref{Theo2}]
Let $\varphi$ be a random characteristic satisfying \ref{C1}-\ref{C5} and \ref{E1}, for some constant $a_{\lambda, i} \in \mathbb{C}$ and a function $r: \mathbb{G} \rightarrow \mathbb{C}$ satisfying $|r(t)| \leq C e^{\alpha t / 2} /(1+t^{2})$, for all $t \in \mathbb{G}$, and some finite constant $C \geq 0$. For any $\lambda \in \Lambda_{\geq}$ of multiplicity $k(\lambda) \in \mathbb{N}$, we put $\vec{a}_{\lambda} \coloneqq \sum_{i=1}^{k(\lambda)} a_{\lambda, i-1} \mathrm{e}_{i}$. Recall the definition of the functions $\phi_{\lambda}$ and $\chi_{\lambda}$ in \eqref{eq77} and \eqref{eq78}, respectively. Consider the following (individual) characteristics. For $t \in \mathbb{R}$,
\begin{align}
\psi_{\Lambda}(t) = \sum_{\lambda \in \Lambda} \vec{a}_{\lambda}^{\top}(\phi_{\lambda}(t) + \chi_{\lambda}(t)) \mathrm{e}_{1},
\end{align} 
\noindent  
\begin{align}
\psi_{\partial \Lambda}(t) = \sum_{\lambda \in \partial \Lambda} \vec{a}_{\lambda}^{\top} \mathbb{E}[\phi_{\lambda}(t)] \mathrm{e}_{1},
\end{align} 
\noindent and 
\begin{align}
\phi_{\partial \Lambda}(t) =  \sum_{\lambda \in \partial \Lambda} \vec{a}_{\lambda}^{\top}\phi_{\lambda}(t) \mathrm{e}_{1} - \sum_{ \lambda \in \partial  \Lambda} \mathbb{E}[\vec{a}_{\lambda}^{\top}\phi_{\lambda}(t) \mathrm{e}_{1}]. 
\end{align}
\noindent These characteristics $\psi_{\Lambda}$, $\psi_{\partial \Lambda}$ and  $\phi_{\partial \Lambda}$ are in fact real-valued. This is noted in the proof of \cite[Theorem 2.15, p.\ 1588]{Iksanov2021}. Indeed, this follows from \cite[Remark 2.16]{Iksanov2021}, which implies that  $\overline{\vec{a}_{\lambda}} = \vec{a}_{\overline{\lambda}}$, $\overline{\phi_{\lambda}} = \phi_{\overline{\lambda}}$ and $\overline{\chi_{\lambda}} = \chi_{\overline{\lambda}}$.

Recall the definition of $H_{\Lambda}$ and $H_{\partial \Lambda}$ in \eqref{eq84} and \eqref{eq86}, respectively. Recall also that $H(t) \coloneqq H_{\Lambda}(t) + H_{\partial \Lambda}(t)$, for $t \in \mathbb{R}$. By \eqref{eq78}, $\mathbb{E}[\chi_{\lambda}(t)] = 0$, for all $\lambda \in \Lambda$ and $t \in \mathbb{R}$. Moreover, by \cite[(6.23)]{Iksanov2021},
\begin{align} \label{eq114}
m_{t}^{\mathbb{E}[\psi_{\Lambda}]} = \mathbf{1}_{[0,\infty)}(t)\sum_{\lambda \in \Lambda} \vec{a}_{\lambda}^{\top} \exp(\lambda, t, k(\lambda)) \mathrm{e}_{1} = \mathbf{1}_{[0,\infty)}(t) \mathbb{E}[H_{\Lambda}(t)],  \quad \text{for} \quad t \in \mathbb{R}. 
\end{align}
\noindent On the other hand, by \cite[(6.24)]{Iksanov2021}
\begin{align} \label{eq115}
m_{t}^{\mathbb{E}[\psi_{\partial \Lambda}]} = \mathbf{1}_{[0,\infty)}(t)\sum_{\lambda \in \partial  \Lambda} \vec{a}_{\lambda}^{\top} \exp(\lambda, t, k(\lambda)) \mathrm{e}_{1} = \mathbf{1}_{[0,\infty)}(t) H_{\partial \Lambda}(t),  \quad \text{for} \quad t \in \mathbb{R}. 
\end{align}

Next, by \eqref{eq87}, \eqref{eq112}, \eqref{eq113} \eqref{eq77}, \eqref{eq78} and \eqref{eq114}, we observe that, for $t \in \mathbb{R}$,
\begin{align} \label{eq116}
\psi_{\Lambda}(t) & =  \sum_{i=1}^{N} \sum_{\lambda \in \Lambda} \vec{a}_{\lambda}^{\top} \mathbf{1}_{(-\infty, X_{i})}(t) \exp(\lambda, t-X_{i}, k(\lambda)) \mathrm{e}_{1} - \mathbf{1}_{(-\infty, 0)}(t)  \sum_{\lambda \in \Lambda} \vec{a}_{\lambda}^{\top}\exp (\lambda,t, k(\lambda)) \mathrm{e}_{1} \nonumber \\
& = \sum_{i=1}^{N} \left( \sum_{\lambda \in \Lambda} \vec{a}_{\lambda}^{\top} \exp(\lambda, t-X_{i}, k(\lambda)) \mathrm{e}_{1} -m_{t-X_{i}}^{\mathbb{E}[\psi_{\Lambda}]} \right) - \mathbf{1}_{(-\infty, 0)}(t) \sum_{\lambda \in \Lambda} \vec{a}_{\lambda}^{\top}\exp (\lambda,t, k(\lambda)) \mathrm{e}_{1} 
\nonumber \\
& = \sum_{i=1}^{N}  \sum_{\lambda \in \Lambda} e^{\lambda (t-X_{i})} \sum_{i=0}^{k(\lambda)-1} a_{\lambda, i} (t-X_{i})^{i} - m^{\mathbb{E}[\psi_{\Lambda}]} \ast \xi(t) - \mathbf{1}_{(-\infty, 0)}(t) \sum_{\lambda \in \Lambda}  e^{\lambda t} \sum_{i=0}^{k(\lambda)-1} a_{\lambda, i} t^{i} \nonumber \\
% & = ( \mathbb{E}[H_{\Lambda}]- m^{\mathbb{E}[\psi_{\Lambda}]}) \ast \xi(t) - \mathbf{1}_{(-\infty, 0)}(t) \mathbb{E}[H_{\Lambda}(t)] \nonumber \\
& =  (\mathbf{1}_{(-\infty, 0)} \mathbb{E}[H_{\Lambda}])\ast \xi(t) - \mathbf{1}_{(-\infty, 0)}(t) \mathbb{E}[H_{\Lambda}(t)].
\end{align}
\noindent Recall that $\mathbb{E}[H_{\Lambda}]$ represents the function $t \in \mathbb{R} \mapsto \mathbb{E}[H_{\Lambda}(t)]$. Similarly, by \eqref{eq77} and \eqref{eq115}, we have that, for $t \in \mathbb{R}$,
\begin{align} \label{eq117}
\psi_{\partial \Lambda}(t) + \phi_{\partial \Lambda}(t)& =  \sum_{i=1}^{N} \sum_{\lambda \in \partial \Lambda} \vec{a}_{\lambda}^{\top} \mathbf{1}_{[0, X_{i})}(t) \exp(\lambda, t-X_{i}, k(\lambda)) \mathrm{e}_{1}\nonumber \\
& = \sum_{i=1}^{N} \mathbf{1}_{[0, \infty)}(t)  \sum_{\lambda \in \partial \Lambda} \vec{a}_{\lambda}^{\top}  \exp(\lambda, t-X_{i}, k(\lambda)) \mathrm{e}_{1} - m^{\mathbb{E}[\psi_{\partial \Lambda}]} \ast \xi(t) \nonumber \\
% & = \mathbf{1}_{[0, \infty)}(t) H_{\partial \Lambda}\ast \xi(t)  - m^{\mathbb{E}[\psi_{\partial \Lambda}]} \ast \xi(t) \nonumber \\
& = \mathbf{1}_{[0, \infty)}(t) H_{\partial \Lambda}\ast \xi(t)  - (\mathbf{1}_{[0,\infty)} H_{\partial \Lambda}) \ast \xi(t).
\end{align}

Let $\varrho(t) \coloneqq \varphi(t) - \psi_{\Lambda}(t) - \psi_{\partial \Lambda}(t)$, for $t \in \mathbb{R}$. By \cite[(6.28)]{Iksanov2021}, the following decomposition, which remains valid for descendant characteristics, is established:
\begin{align} \label{eq170}
\mathcal{Z}^{\varphi}_{t} = H_{\Lambda}(t) + H_{\partial \Lambda}(t) + \mathcal{Z}^{\varrho - \phi_{\partial \Lambda}}_{t} + \mathcal{Z}^{\chi}_{t}, \quad \text{for} \quad t \in \mathbb{R}_{+},
\end{align}
\noindent where $\chi$ is an individual characteristic defined in \cite[(6.26)]{Iksanov2021}. Therefore, it is sufficient to prove the limit theorem for $\mathcal{Z}^{\varrho - \phi_{\partial \Lambda}}_{t}$ and $\mathcal{Z}^{\chi}_{t}$, in order to establish Theorem \ref{Theo2}.

%We start with $\mathcal{Z}^{\varrho - \phi_{\partial \Lambda}}_{t}$. 
By Lemma \ref{lemma7}, Remark \ref{Remark1} and Remark \ref{Remark6}, the characteristic $\varrho - \phi_{\partial \Lambda}$ has c\`adl\`ag paths and satisfies \ref{C1}-\ref{C5}. Note that $\mathbb{E}[\phi_{\partial \Lambda}(t)] = 0$, for $t \in \mathbb{R}$ and that, by \ref{E1}, \eqref{eq114} and \eqref{eq115},
\begin{align} \label{eq118}
m_{t}^{\mathbb{E}[\varrho - \phi_{\partial \Lambda}]} & =  m_{t}^{\mathbb{E}[\varrho]}   =  m_{t}^{\mathbb{E}[\varphi]} - m_{t}^{\mathbb{E}[\psi_{\Lambda}]} - m_{t}^{\mathbb{E}[\psi_{\partial \Lambda}]} \nonumber \\
& =  m_{t}^{\mathbb{E}[\varphi]} - \mathbf{1}_{[0, \infty)}(t)(\mathbb{E}[H_{\Lambda}(t)] + H_{\partial \Lambda}(t))  = r(t), \quad \text{for} \quad t \in \mathbb{G} 
\end{align}
\noindent Then, under our assumptions, we can apply Theorem \ref{Theo3} to conclude that, as $t \rightarrow \infty$, $t \in \mathbb{G}$,
\begin{align} \label{eq111}
e^{-\frac{\alpha}{2}t} \mathcal{Z}^{\varrho - \phi_{\partial \Lambda}}_{t} \xrightarrow[]{\rm st} \sigma \left( \frac{W}{\beta} \right)^{1/2} \mathcal{N},
\end{align}
\noindent where $\mathcal{N}$ is a standard normal random variable independent of $\mathcal{F}$ and
\begin{align}
\sigma^{2} \coloneqq  \int_{\mathbb{G}}  \mathbb{E}[(\chi^{(\varrho - \phi_{\partial \Lambda}+r\ast \xi,h)}(s))^{2}]  e^{-\alpha s} \ell({\rm d} s).
\end{align}

The function $\chi^{(\varrho - \phi_{\partial \Lambda}+r\ast \xi,h)}$ can be simplified. By \eqref{eq116}, \eqref{eq117} and \eqref{eq118}, 
\begin{align}
\varrho - \phi_{\partial \Lambda}+r\ast \xi(t) & = \varphi(t) - (\mathbb{E}[H_{\Lambda}] -m^{\mathbb{E}[\varphi]})\ast \xi(t) - \mathbf{1}_{[0, \infty)}(t) H_{\partial \Lambda} \ast \xi(t) + \mathbf{1}_{(-\infty, 0)}(t) \mathbb{E}[H_{\Lambda}(t)],
\end{align}
\noindent for $t \in \mathbb{G}$. Observe that $\mathbf{1}_{(-\infty, 0)}(t) \mathbb{E}[H_{\Lambda}(t)]$ is deterministic. Thus, Remark \ref{Remark2} implies that $\chi^{(\varrho - \phi_{\partial \Lambda}+r\ast \xi,h)} = \chi^{(\varphi + (m^{\mathbb{E}[\varphi]} - \mathbb{E}[H_{\Lambda}])\ast \xi - \mathbf{1}_{[0, \infty)} \cdot H_{\partial \Lambda}\ast \xi,h)}$. 

This proves Theorem \ref{Theo2} under the assumption that $\rho_{i} =0$ for all $i \in \mathbb{N}_{0}$ (that is, $\max\{i \in \mathbb{N}_{0}: \rho_{i} >0\} =0$). Recall that $\rho_{i}$ is defined in \eqref{eq120}. Indeed, as explained in the proof of \cite[Theorem 2.15, p.\ 1587]{Iksanov2021}, in this case, we have that $\sum_{i=1}^{N}\vec{a}_{\lambda}^{\top}  \exp(\lambda, t-X_{i}, k(\lambda)) \mathrm{e}_{1}$ is almost surely constant for any $\lambda \in \partial \Lambda$ and $t \in \mathbb{G}$. This follows from \cite[(6.27)]{Iksanov2021} and the definition of $(R_{\lambda,i})_{i=0}^{k(\lambda)-1}$ in \eqref{eq119} and $\rho_{i}$ for $i \in \mathbb{N}_{0}$. Then, by the definition of $\chi$ in \cite[(6.26)]{Iksanov2021}, $\mathcal{Z}^{\chi}_{t} = 0$ almost surely for $t \in \mathbb{R}$. In particular, we also have that $H_{\partial \Lambda} \ast \xi(t)$ is almost surely constant for any $t \in \mathbb{G}$ and Remark \ref{Remark2} implies that $\chi^{(\varrho - \phi_{\partial \Lambda}+r\ast \xi,h)} = \chi^{(\varphi + (m^{\mathbb{E}[\varphi]} - \mathbb{E}[H_{\Lambda}])\ast \xi,h)}$, which establishes \eqref{eq121}.

Indeed, if $\sigma^{2}>0$, then \eqref{eq170} and \eqref{eq111} imply \eqref{eq165}. On the other hand, recall that $\varrho \coloneqq \varphi - \psi_{\Lambda} - \psi_{\partial \Lambda}$. Since $\psi_{\Lambda}$, $\psi_{\partial \Lambda}$ and $\phi_{\partial \Lambda}$ are $0$-dependent descendant characteristics (i.e., an individual characteristics), it follows that $\psi_{\Lambda}^{(i)}(t) \stackrel{a.s.}{=} \psi_{\Lambda}(t)$, $\psi_{\partial \Lambda}^{(i)}(t) \stackrel{a.s.}{=} \psi_{\partial \Lambda}(t)$ and $\phi_{\partial \Lambda}^{(i)}(t) \stackrel{a.s.}{=} \phi_{\partial \Lambda}(t)$, for $i = 1,\dots, h$ (if $h \geq 1$); see \eqref{eq27}. Then, if $\sigma^{2}=0$, Theorem \ref{Theo3} and \eqref{eq118} imply that, for all $t \in \mathbb{G}$,
\begin{align} \label{eq131}
\mathcal{Z}^{\varrho - \phi_{\partial \Lambda}}_{t} \stackrel{a.s.}{=} \sum_{0\leq |u| \leq h-1} (\varphi_{u}^{(h-|u|)}(t-S(u)) - \phi_{u}(t-S(u)))  + \sum_{|u| = h} r(t - S(u)),
\end{align}
\noindent where $\phi \coloneqq \psi_{\Lambda} + \psi_{\partial \Lambda} + \phi_{\partial \Lambda}$.
\noindent This and \eqref{eq170} establish Theorem \ref{Theo2} \ref{Theo2Claim1}.

We now prove Theorem \ref{Theo2} assuming  $\rho_{i} > 0$ for some $i \in \mathbb{N}_{0}$. From \eqref{eq111}, we have that $t^{-\frac{1}{2}}e^{-\frac{\alpha}{2}t} \mathcal{Z}^{\varrho - \phi_{\partial \Lambda}}_{t} \xrightarrow[]{\rm p} 0$, as $t \rightarrow \infty$, $t \in \mathbb{G}$. Let $n \coloneqq \max\{i \in \mathbb{N}_{0}: \rho_{i} >0\}$. Clearly $n \in \mathbb{N}$. Since $\chi$ is the characteristic defined in \cite[(6.26)]{Iksanov2021}, the argument in the proof of \cite[Theorem 2.15, p.\ 1588]{Iksanov2021} directly applies and thus, 
\begin{align}
\left(\frac{\rho_{n}^{2}t^{2n+1}}{2n+1}e^{\alpha t} \right)^{-\frac{1}{2}} \mathcal{Z}^{\chi}_{t} \xrightarrow[]{\rm st} \left( \frac{W}{\beta} \right)^{1/2} \mathcal{N}, \quad \text{as} \quad t \rightarrow \infty, \, \, t \in \mathbb{G},
\end{align}
\noindent where $\mathcal{N}$ is a standard normal random variable independent of $\mathcal{F}$.

This finishes the proof of Theorem \ref{Theo2}.
\end{proof}

%%%%%%%%%%%%%%%%%%%%%%%%%%%%%%%%%%%%%%%%%%%%%%%%%%%%%%%%%%%%%%%%%%%%%%%
\subsection{Proof of Corollary \ref{corollary5}} \label{ProofofCorollaryI}
%%%%%%%%%%%%%%%%%%%%%%%%%%%%%%%%%%%%%%%%%%%%%%%%%%%%%%%%%%%%%%%%%%%%%%%

This section presents the proof of Corollary \ref{corollary5}, which is a consequence of Theorem \ref{Theo2}. We will specifically derive sufficient conditions under which \ref{E1} is satisfied. Recall that assumption \ref{E1} concerns the asymptotic expansion of the mean $m_{t}^{\mathbb{E}[\varphi]}$ of a supercritical general branching process.
Recall the probability measure $\mu_{\alpha}$ on $\mathbb{R}$ defined in \eqref{eq153} and the definition of a spread-out probability measure given before Corollary \ref{corollary5}.

\begin{lemma} \label{lemma13}
Suppose that assumptions \ref{A1}-\ref{A4} hold and that $\mu_{\alpha}$ is spread-out such that there exists a Stone's decomposition of its associated renewal measure $e^{-\alpha x} \nu^{(\infty)}({\rm d}x)$ that satisfies \ref{S1}. Furthermore, suppose that
\begin{align} \label{eq160}
\beta \coloneqq \int_{[0,\infty)} x e^{-\alpha x} \mu({\rm d}x) \in (0,\infty).
\end{align}
\noindent Let $\varphi$ be a non-negative characteristic that vanishes on the negative half-line (i.e., $\varphi(t)= 0$, for $t<0$) such that $t \in \mathbb{R} \mapsto \mathbb{E}[\varphi](t)$ is c\`adl\`ag, $\varphi(t) \in L^{1}(\Omega, \mathcal{F}, \mathbb{P})$, and 
\begin{align} \label{eq124}
\mathbb{E}[\varphi(t)] \leq C e^{(\alpha-\delta) t} \mathbf{1}_{[0,\infty)}(t), \quad \text{for every} \, \, t \in \mathbb{R},
\end{align}
\noindent for some $\delta \in (\varepsilon, \alpha)$ and finite constant $C > 0$, where $\varepsilon \in (\alpha/2, \alpha)$ is as in \ref{S1}. Then, 
\begin{align} \label{eq161}
m_{t}^{\mathbb{E}[\varphi]} = \mathbf{1}_{[0,\infty)}(t) e^{\alpha t} \beta^{-1} \int_{0}^{\infty} \mathbb{E}[\varphi(x)] e^{-\alpha x} {\rm d} x + r(t), 
\end{align}
\noindent where $r: \mathbb{R} \rightarrow \mathbb{R}$ is a function  satisfying $|r(t)| \leq C^{\prime} e^{(\alpha-\varepsilon)t} \mathbf{1}_{[0,\infty)}(t)$  for all $t \in \mathbb{R}$ and some finite constant $C^{\prime} > 0$.
\end{lemma}

\begin{proof}
By \eqref{eq124}, one can readily deduce from \cite[Remark 3.10.4 on page 236]{Resnick2014} (or \cite[Proposition 2.6 (a)]{Iksanov2021}) that $t \in \mathbb{R} \mapsto \mathbb{E}[\varphi](t)e^{-\alpha t}$ is directly Riemann integrable. Recall that $\varphi$ is non-negative and that $\varphi(t) = 0$ for $t < 0$. Then, by \eqref{eq123} and \eqref{eq109}, we know that, for $t \in \mathbb{R}$, 
\begin{align} \label{eq159} 
m_{t}^{\mathbb{E}[\varphi]} & = \mathbf{1}_{[0,\infty)}(t) e^{\alpha t} \int_{[0,\infty)} \mathbb{E}[\varphi](t-x) e^{-\alpha(t-x)} \sum_{i=0}^{\infty} \mu_{\alpha}^{\ast i}({\rm d}x) \nonumber \\
& = \mathbf{1}_{[0,\infty)}(t) e^{\alpha t} \int_{[0,t]} \mathbb{E}[\varphi](t-x) e^{-\alpha(t-x)} e^{-\alpha x} \nu^{(\infty)}({\rm d}x). 
\end{align}
\noindent Since $\mu_{\alpha}$ is spread-out with associated renewal measure $e^{-\alpha x} \nu^{(\infty)}({\rm d}x)$ having a Stone's decomposition $e^{-\alpha t} \nu^{(\infty)}({\rm d}x) = \nu^{(\infty)}_{\alpha, 1}({\rm d}x) + \nu^{(\infty)}_{\alpha, 2}({\rm d}x)$ satisfying  \ref{S1}, it follows from \eqref{eq159} that, for $t \in \mathbb{R}$, 
\begin{align} \label{eq127} 
m_{t}^{\mathbb{E}[\varphi]} & = \mathbf{1}_{[0,\infty)}(t) e^{\alpha t} \left( \int_{[0,t]} \mathbb{E}[\varphi](t-x) e^{-\alpha(t-x)} \nu^{(\infty)}_{\alpha, 1}({\rm d}x) +  \int_{[0,t]} \mathbb{E}[\varphi](t-x) e^{-\alpha(t-x)} \nu^{(\infty)}_{\alpha, 2}({\rm d}x) \right) \nonumber \\
& = \mathbf{1}_{[0,\infty)}(t) e^{\alpha t} \left( \int_{0}^{t} \mathbb{E}[\varphi](t-x) e^{-\alpha(t-x)} f^{(\infty)}_{\alpha, 1}(x) {\rm d}x +  \int_{[0,t]} \mathbb{E}[\varphi](t-x) e^{-\alpha(t-x)} \nu^{(\infty)}_{\alpha, 2}({\rm d}x) \right)\nonumber \\
& = \mathbf{1}_{[0,\infty)}(t) e^{\alpha t} \beta^{-1} \int_{0}^{\infty} \mathbb{E}[\varphi(x)] e^{-\alpha x} {\rm d}x +  r(t),
\end{align}
\noindent where 
\begin{align}
r(t) & \coloneqq \mathbf{1}_{[0,\infty)}(t) e^{\alpha t} \Big( \int_{[0,t]} \mathbb{E}[\varphi](t-x) e^{-\alpha(t-x)} \nu^{(\infty)}_{\alpha, 2}({\rm d}x)  \nonumber \\
& \quad \quad \quad \quad + \int_{0}^{t}  \mathbb{E}[\varphi](x) e^{-\alpha x} g(t-x) {\rm d}x - \beta^{-1} \int_{t}^{\infty} \mathbb{E}[\varphi](x) e^{-\alpha x} {\rm d}x  \Big).
\end{align}
\noindent On the other hand, by \ref{S1} and \eqref{eq124}, there exists a finite constant $C_{1}>0$ such that, for $t \in \mathbb{R}$,
\begin{align}
|r(t)| \leq  \mathbf{1}_{[0,\infty)}(t) C_{1}  e^{\alpha t} \left( \int_{[0,t]} e^{-\delta (t-x)}\nu^{(\infty)}_{\alpha, 2}({\rm d}x) + \int_{0}^{t} e^{-\delta x}  e^{-\varepsilon (t-x)}  {\rm d}x + \int_{t}^{\infty} e^{-\delta x}  {\rm d}x  \right).
\end{align}
\noindent Recall that $\delta \in (\varepsilon, \alpha)$. Then, 
\begin{align} \label{eq128} 
|r(t)| & \leq  \mathbf{1}_{[0,\infty)}(t) C_{1} e^{\alpha t}  \left( \int_{[0,t]} e^{-\varepsilon (t-x)}\nu^{(\infty)}_{\alpha, 2}({\rm d}x) + e^{-\varepsilon t} \int_{0}^{t} e^{-(\delta-\varepsilon) x}  {\rm d}x + \delta^{-1} e^{-\delta x}   \right)
\nonumber \\
& \leq  \mathbf{1}_{[0,\infty)}(t) C_{1} e^{(\alpha -\varepsilon) t}  \left( \int_{[0,\infty]} e^{\varepsilon x}\nu^{(\infty)}_{\alpha, 2}({\rm d}x) +  (\delta-\varepsilon)^{-1} + \delta^{-1}\right).
\end{align}
\noindent Therefore, \eqref{eq161} follows from \eqref{eq127}, \eqref{eq128} and \ref{S1}. This concludes our proof. 
\end{proof}

\begin{proof}[Proof of Corollary \ref{corollary5}]
To apply Theorem \ref{Theo2}, we only need to verify that $\varphi$ satisfies \ref{C1}, \ref{C4} and \ref{E1}. Recall our assumption that $\varphi$ is non-negative and vanishes on the negative half-line. Then, \ref{C4} is clearly satisfied.

We continue by showing that $\varphi$ satisfies \ref{C1}. By \ref{C3}, $\varphi(t) \in L^{1}(\Omega, \mathcal{F}, \mathbb{P})$ for every $t \in \mathbb{R}$ and $t \in \mathbb{R} \mapsto \mathbb{E}[\varphi](t)$ is c\`adl\`ag; see the discussion after \ref{C3}. Then, by using \eqref{eq129}, one can readily deduce from \cite[Remark 3.10.4 on page 236]{Resnick2014} (or \cite[Proposition 2.6 (a)]{Iksanov2021}) that $t \in \mathbb{R} \mapsto \mathbb{E}[\varphi](t)e^{-\alpha t}$ is directly Riemann integrable. Therefore, $\varphi$ satisfies \ref{C1}. 

Recall that \ref{A1}-\ref{A7} (see also \eqref{eq3}) implies \eqref{eq160}. Then, \ref{S1} and Lemma \ref{lemma13} implies that $\varphi$ also satisfies \ref{E1}. Therefore, our claim follows from Theorem \ref{Theo2}, with $a_{\alpha} \coloneqq \beta^{-1} \int_{0}^{\infty} \mathbb{E}[\varphi(x)] e^{-\alpha x} {\rm d} x$ and $H(t) = H_{\Lambda}(t) = a_{\alpha} e^{\alpha t} W$, for $t \in \mathbb{R}$ (note that $n =-1$ and thus, $\rho_{-1}=0$). 
\end{proof}

\begin{appendices}
%%%%%%%%%%%%%%%%%%%%%%%%%%%%%%%%%%%%%%%%%%%%%%%%%%%%%%%%%%%%%%%%%%%%%%%
\section{Proof of Lemmas \ref{lemma12} and \ref{lemma16}} \label{Append1}
%%%%%%%%%%%%%%%%%%%%%%%%%%%%%%%%%%%%%%%%%%%%%%%%%%%%%%%%%%%%%%%%%%%%%%%

\begin{proof}[Proof of Lemma \ref{lemma12}]
Observe that, for $\delta_{1}, \delta_{2} \in \mathbb{R}_{+}$ such that $\delta_{1} \leq \delta_{2}$, $u \in \mathcal{I}$ and $t \in \mathbb{R}$, 
\begin{align} 
|\varphi_{u}(t-S(u))| & \leq C_{u}^{\varphi} \left( e^{\delta_{1} (t-S(u))}\mathbf{1}_{[0, \infty)}(t-S(u)) + e^{\delta_{2}(t-S(u))}  \mathbf{1}_{(-\infty, 0)}(t-S(u)) \right) \label{eq145}  \\
& \leq  C_{u}^{\varphi} e^{\delta_{2}t} e^{-\delta_{1} S(u)},  \label{eq132}
\end{align}
\noindent where $C_{u}^{\varphi} \coloneqq \sup_{t\in \mathbb{R}} e^{-(\delta_{1} t \wedge \delta_{2} t)} |\varphi_{u}(t)|$.

First, we prove \ref{lemma12Pro3}. Consider, $\delta_{1} \in [0, \alpha)$ and $\delta_{2} \in  (\alpha, \infty)$ such that $\sup_{t\in \mathbb{R}} e^{-(\delta_{1} t \wedge \delta_{2} t)} |\varphi(t)| \in L^{1}(\Omega, \mathcal{F}, \mathbb{P})$. In particular, $\varphi(t) \in L^{1}(\Omega, \mathcal{F}, \mathbb{P})$, for every $t \in \mathbb{R}$. For $t \in \mathbb{R}$, define $\varphi^{\ast}(t) \coloneqq \sup_{|s-t|\leq 1} |\varphi(s)|$.  By \eqref{eq145} (with $u = \varnothing$), we have that for $t \in \mathbb{R}$,
\begin{align} 
\mathbb{E}[\varphi^{\ast}(t)] \leq  \mathbb{E}[C^{\varphi}_{\varnothing}]  \left( e^{\delta_{1} (t+1)}\mathbf{1}_{[-1, \infty)}(t) +  e^{\delta_{2}(t+1)}  \mathbf{1}_{(-\infty, 1)}(t) \right).
\end{align}
\noindent Then, 
\begin{align}
\int_{-\infty}^{\infty} \mathbb{E}[\varphi^{\ast}(t)] e^{-\alpha t} {\rm d}t & \leq  \mathbb{E}[C_{\varnothing}^{\varphi}]  \left( e^{\delta_{1}} \int_{-1}^{\infty} e^{-(\alpha - \delta_{1}) t} {\rm d}t  + e^{\delta_{2}}  \int_{-\infty}^{1} e^{-(\alpha - \delta_{2}) t} {\rm d}t  \right)  \nonumber \\
& = \mathbb{E}[C_{\varnothing}^{\varphi}]  \left( \frac{e^{\alpha}}{\alpha -\delta_{1}} -\frac{e^{-\alpha +2\delta_{2}}}{\alpha -\delta_{2}} \right)  <\infty.
\end{align}
\noindent Then, \ref{lemma12Pro3} follows from \cite[Proposition 2.6 (b)]{Iksanov2021} (a careful inspection reveals that the proof of \cite[Proposition 2.6 (b)]{Iksanov2021} extends to descendant characteristic).

We now proceed to prove \ref{lemma12Pro4}. Consider, $\delta_{1} \in [0, \alpha/2)$ and $\delta_{2} \in  (\alpha/2, \infty)$ such that $\sup_{t\in \mathbb{R}} e^{-(\delta_{1} t \wedge \delta_{2} t)} |\varphi(t)| \in L^{2}(\Omega, \mathcal{F}, \mathbb{P})$. In particular, $\varphi(t) \in L^{2}(\Omega, \mathcal{F}, \mathbb{P})$, for every $t \in \mathbb{R}$. Recall that $(x+y)^{2} \leq 2 (x^{2}+y^{2})$, for $x,y \in \mathbb{R}$. By \eqref{eq145} (with $u = \varnothing$), we have that for $t \in \mathbb{R}$,
\begin{align} 
\mathbb{E}[(\varphi^{\ast}(t))^{2}] \leq 2 \mathbb{E}[(C^{\varphi}_{\varnothing})^{2}]  \left(   e^{2\delta_{1} (t+1)}\mathbf{1}_{[-1, \infty)}(t) + e^{2\delta_{2}(t+1)}  \mathbf{1}_{(-\infty, 1)}(t) \right).
\end{align}
\noindent Then, 
\begin{align}
\int_{-\infty}^{\infty} \mathbb{E}[(\varphi^{\ast}(t))^{2}] e^{-\alpha t} {\rm d}t & \leq 2 \mathbb{E}[(C_{\varnothing}^{\varphi})^{2}]   \left( e^{2\delta_{1}} \int_{-1}^{\infty} e^{-(\alpha - 2\delta_{1}) t} {\rm d}t  + e^{2\delta_{2}}  \int_{-\infty}^{1} e^{-(\alpha - 2\delta_{2}) t} {\rm d}t  \right)  \nonumber \\
& = 2 \mathbb{E}[(C_{\varnothing}^{\varphi})^{2}] \left( \frac{e^{\alpha}}{\alpha -2\delta_{1}} -\frac{e^{-\alpha +4\delta_{2}}}{\alpha -2\delta_{2}} \right)  <\infty.
\end{align}
\noindent Then, \ref{lemma12Pro4} follows from \cite[Proposition 2.6 (c)]{Iksanov2021}.

Next, we prove \ref{lemma12Pro2}. Consider $\delta_{2} \in  [\vartheta, \infty)$ such that $\sup_{t\in (-\infty, 0)} e^{-  \delta_{2} t} |\varphi(t)| \in L^{1}(\Omega, \mathcal{F}, \mathbb{P})$. Observe that, for $u \in \mathcal{I}$ and $t \in \mathbb{R}$, 
\begin{align} 
|\varphi_{u}(t-S(u))\mathbf{1}_{(-\infty, 0)}(t-S(u))|  \leq  \tilde{C}_{u}^{\varphi} e^{\delta_{2}t} e^{-\delta_{2} S(u)},  
\end{align}
\noindent where $\tilde{C}_{u}^{\varphi} \coloneqq \sup_{t\in (-\infty, 0)} e^{-\delta_{2} t} |\varphi_{u}(t)|$. Fix $\varepsilon >0$. Then, for $t \in \mathbb{R}$, we have that
\begin{align} \label{eq157}
\sum_{i=0}^{h}  \sum_{|u| =i} \sup_{|s-t|\leq \varepsilon}|\varphi_{u}(s-S(u))\mathbf{1}_{(-\infty, 0)}(t-S(u))|  \leq  e^{\delta_{2} (t+\varepsilon)} \sum_{i=0}^{h}\sum_{|u| =i} \tilde{C}_{u}^{\varphi} e^{-\delta_{2} S(u)}.
\end{align}
\noindent Let $(\mathcal{A}^{(n)})_{n \geq -1}$ be the collection $\sigma$-algebras defined in \eqref{eq135}. Observe that, $S(u)$ is $\mathcal{A}^{(|u|-1)}$-measurable, while $\tilde{C}_{u}^{\varphi}$ is independent of $\mathcal{A}^{(|u|-1)}$. Then, by \cite[Theorem 2 (f), Section 7, Chapter II]{Shiryaev1996} and \eqref{eq156},
\begin{align}  
\mathbb{E}\left[ \sum_{|u| =i} \tilde{C}_{u}^{\varphi} e^{-\delta_{2} S(u)}   \right] & = \mathbb{E}\left[ \sum_{|u| =i} \mathbb{E}\left[  \tilde{C}_{u}^{\varphi} e^{-\delta_{2} S(u)}    \mid \mathcal{A}^{(i-1)} \right] \right]  = \mathbb{E}[\tilde{C}_{u}^{\varphi}] \mathbb{E}\left[ \sum_{|u| =i} e^{-\delta_{2} S(u)}  \right] < \infty,
\end{align}
\noindent for $i=0,  \dots, h$. Therefore, the random variable on the right-hand side of \eqref{eq157} is in $L^{1}(\Omega, \mathcal{F}, \mathbb{P})$, which implies \eqref{eq158}. This concludes the proof of \ref{lemma12Pro2}.

Finally, we prove \ref{lemma12Pro1}. Fix $\varepsilon >0$. Consider $\delta_{1} \in [\vartheta, \alpha]$ and $\delta_{2} \in  [\delta_{1}, \infty)$ such that $\sup_{t\in \mathbb{R}} e^{-(\delta_{1} t \wedge \delta_{2} t)} |\varphi(t)| \in L^{2}(\Omega, \mathcal{F}, \mathbb{P})$. Then, by \eqref{eq132}, for $s,t \in \mathbb{R}$ such that $|s-t| \leq \varepsilon$, we have that
\begin{align} \label{eq146}
\sum_{i=0}^{h} \left( \sum_{|u| =i} |\varphi_{u}(s-S(u))| \right)^{2} \leq  e^{2 \delta_{2} (t+\varepsilon)} \sum_{i=0}^{h} \left(\sum_{|u| =i} C_{u}^{\varphi} e^{-\delta_{1} S(u)}  \right)^{2}.
\end{align}
\noindent  Recall that, $S(u)$ is $\mathcal{A}^{(|u|-1)}$-measurable, while $C_{u}^{\varphi}$ is independent of $\mathcal{A}^{(|u|-1)}$. Then, by \cite[Theorem 2 (f), Section 7, Chapter II]{Shiryaev1996},
\begin{align}  \label{eq147}
\mathbb{E}\left[ \sum_{|u| =i} (C_{u}^{\varphi})^{2} e^{-2\delta_{1} S(u)}   \right] & = \mathbb{E}\left[ \sum_{|u| =i} \mathbb{E}\left[  (C_{u}^{\varphi})^{2} e^{-2\delta_{1} S(u)}    \mid \mathcal{A}^{(i-1)} \right] \right]  = \mathbb{E}[(C_{\varnothing}^{\varphi})^{2}] \mathbb{E}\left[ \sum_{|u| =i} e^{-2\delta_{1} S(u)}  \right],
\end{align}
\noindent for $i=0,  \dots, h$. Observe also that $C_{u}^{\varphi}$ and $C_{v}^{\varphi}$ are independent for $u, v \in \mathcal{I}$ such that $u \neq v$ and $|u| = |v|$. Similarly to \eqref{eq147}, we have that
\begin{align}  \label{eq151}
\mathbb{E}\left[ \sum_{|u| =i}  \sum_{|v| =i} \mathbf{1}_{\{u \neq v \}}  C_{u}^{\varphi} C_{v}^{\varphi} e^{-\delta_{1} (S(u)+ S(v))}  \right]  = (\mathbb{E}[C_{\varnothing}^{\varphi}])^{2} \mathbb{E}\left[ \sum_{|u| =i}  \sum_{|v| =i} \mathbf{1}_{\{u \neq v\}}  e^{-\delta_{1} (S(u)+ S(v))}  \right],
\end{align}
\noindent for $i=0,  \dots, h$. Then, by \eqref{eq152}, \eqref{eq147} and \eqref{eq151}, 
\begin{align} \label{eq150}
\mathbb{E}\left[ \left(\sum_{|u| =i} C_{u}^{\varphi} e^{-\delta_{1} S(u)}  \right)^{2} \right] \leq (\mathbb{E}[(C_{\varnothing}^{\varphi})^{2}]  + (\mathbb{E}[C_{\varnothing}^{\varphi}])^{2}) \mathbb{E}\left[ \left(\sum_{|u| =i}  e^{-\delta_{1} S(u)}  \right)^{2}\right] < \infty,
\end{align}
\noindent for $i=0,  \dots, h$. Consequently, the random variable on the right-hand side of \eqref{eq146} is in $L^{1}(\Omega, \mathcal{F}, \mathbb{P})$, which implies \eqref{eq155}; see for e.g., \cite[(2), p.\ 183]{Resnick2014}. This concludes the proof of \ref{lemma12Pro1}.
\end{proof}

\begin{proof}[Proof of Lemma \ref{lemma16}]
Let $\varphi$ be a characteristic satisfying \ref{C2}-\ref{C3}. Note that
\begin{align} \label{eq106}
\sum_{n \in \mathbb{Z}} \sup_{t \in [n, n+1)}  e^{-\alpha t}  \int_{[0, \infty)} {\rm Var}[\varphi](t-x) \nu^{(h)}({\rm d} x) \leq  \int_{[0, \infty)}  \sum_{n \in \mathbb{Z}} e^{-\alpha n} \sup_{t \in [n, n+1)} {\rm Var}[\varphi](t-x) \nu^{(h)}({\rm d} x).
\end{align}
\noindent On the other hand, for $x \in [0, \infty)$,
\begin{align}
\sum_{n \in \mathbb{Z}} e^{-\alpha n} \sup_{t \in [n, n+1)} {\rm Var}[\varphi](t-x) & = e^{-\alpha x}  \sum_{n \in \mathbb{Z}} e^{-\alpha (n-x)} \sup_{t \in [n-x, n-x+1)} {\rm Var}[\varphi](t) \nonumber \\
& = e^{-\alpha x}  \sum_{n \in \mathbb{Z}} e^{-\alpha (n-x+\lfloor x \rfloor)} \sup_{t \in [n-x+\lfloor x \rfloor, n-x+\lfloor x \rfloor+1)} {\rm Var}[\varphi](t).
\end{align}
\noindent To see the second equality, simply note that the sum includes all integers. Since $\lfloor x \rfloor \leq x \leq \lfloor x \rfloor +1$, it follows that
\begin{align} \label{eq141}
\sum_{n \in \mathbb{Z}} e^{-\alpha n} \sup_{t \in [n, n+1)} {\rm Var}[\varphi](t-x) & \leq  e^{-\alpha (x-1)}  \sum_{n \in \mathbb{Z}} e^{-\alpha n} \sup_{t \in [n-1, n+1)} {\rm Var}[\varphi](t) \nonumber \\
& \leq 2 e^{-\alpha (x-1)}  \sum_{n \in \mathbb{Z}} e^{-\alpha n} \sup_{t \in [n, n+1)} {\rm Var}[\varphi](t).
\end{align}
\noindent Note that, by \ref{A4} and \eqref{eq62}, $\int_{[0, \infty)} e^{-\alpha x}\nu^{(h)}({\rm d} x) = h+1$. Then, by \eqref{eq106}, \eqref{eq141} and \ref{C2},
\begin{align}\label{eq91}
\sum_{n \in \mathbb{Z}} e^{-\alpha n} \sup_{t \in [n, n+1)}  \int_{[0, \infty)} {\rm Var}[\varphi](t-x) \nu^{(h)}({\rm d} x) & \leq 2 e^{\alpha} \int_{[0, \infty)} e^{-\alpha x}\nu^{(h)}({\rm d} x) \cdot \sum_{n \in \mathbb{Z}} e^{-\alpha n} \sup_{t \in [n, n+1)} {\rm Var}[\varphi](t) \nonumber \\
& \leq 2 (h+1) e^{2\alpha}  \sum_{n \in \mathbb{Z}}\sup_{t \in [n, n+1)}  e^{-\alpha t} {\rm Var}[\varphi](t) < \infty. 
\end{align}

Recall that since the characteristic $\varphi$ satisfies \ref{C3}, its variance function ${\rm Var}[\varphi]$ is c\`adl\`ag. Thus, the dominated convergence theorem implies that the function in \eqref{eq103} is c\`adl\`ag. Therefore, our claim follows from \eqref{eq91} and, for example, (a slightly extended version of) \cite[Remark 3.10.4, p.\ 236]{Sidney1992} (or \cite[Proposition 4.1, Chapter V]{Soren2003}). 
\end{proof}

%%%%%%%%%%%%%%%%%%%%%%%%%%%%%%%%%%%%%%%%%%%%%%%%%%%%%%%%%%%%%%%%%%%%%%%
\section{Uniform integrability} \label{Append2}
%%%%%%%%%%%%%%%%%%%%%%%%%%%%%%%%%%%%%%%%%%%%%%%%%%%%%%%%%%%%%%%%%%%%%%%

This section provides a technical result on uniform integrability that is used in the proof of Lemma \ref{lemma4} and, subsequently, our main theorem. Although likely a known result, we could not locate a reference.

\begin{lemma} \label{lemma14}
Let $(\mathcal{X}_{t})_{t \in T}$ be a family of real-valued random variables indexed by $T$, defined on the probability space $(\Omega, \mathcal{B}, \mathbb{P})$. Let $\mathcal{G}$ be a sub-$\sigma$-algebra of $\mathcal{F}$. Suppose that $(\mathcal{X}_{t})_{t \in T}$ is uniformly integrable, then the family $(\mathbb{E}[\mathcal{X}_{t} \mid \mathcal{G}])_{t \in T}$ is uniformly integrable. 
\end{lemma}

\begin{proof}
Since $(\mathcal{X}_{t})_{t \in T}$ is uniformly integrable, we know from \cite[Theorem 6.5.1]{Resnick2014} that there exists a  finite constant $b>0$ such that
$\sup_{t \in T} \mathbb{E}[|\mathcal{X}_{t}|] \leq b$. Moreover, for all $\varepsilon>0$, there exists $\delta = \delta(\varepsilon)>0$ such that for all $A \in \mathcal{B}$, $\sup_{t \in T} \mathbb{E}[|\mathcal{X}_{t}| \mathbf{1}_{A}] < \varepsilon$ if $\mathbb{P}(A) < \delta$. 

Fix $\varepsilon >0$ and consider $\delta >0$ as before. Set $a \coloneq b / \delta$. By the triangle inequality, for all $t \in T$, we have that $|\mathbb{E}[\mathcal{X}_{t} \mid \mathcal{G}]| \leq \mathbb{E}[|\mathcal{X}_{t}| \mid \mathcal{G}]$ almost surely. In particular, $\sup_{t \in T} \mathbb{E}[|\mathbb{E}[\mathcal{X}_{t} \mid \mathcal{G}]|] < \infty$ and by Markov's inequality, for any $t \in T$,
\begin{align} \label{eq163}
\mathbb{P}(|\mathbb{E}[\mathcal{X}_{t} \mid \mathcal{G}]| > a) \leq a^{-1} \mathbb{E}[|\mathcal{X}_{t}|] \leq b/a = \delta. 
\end{align}
\noindent Recall that $\mathbb{E}[\mathcal{X}_{t} \mid \mathcal{G}]$ and $\mathbb{E}[|\mathcal{X}_{t}| \mid \mathcal{G}]$ are $\mathcal{G}$-measurable. In particular, $\{ |\mathbb{E}[\mathcal{X}_{t} \mid \mathcal{G}]| > a\}$ is also $\mathcal{G}$-measurable. Therefore, by the definition of conditional expectation, \eqref{eq163} and the uniformly integrability of $(\mathcal{X}_{t})_{t \in T}$,
\begin{align}
\sup_{t \in T}\mathbb{E} \left[ |\mathbb{E}[\mathcal{X}_{t} \mid \mathcal{G}]| \mathbf{1}_{\{|\mathbb{E}[\mathcal{X}_{t} \mid \mathcal{G}]| > a \}}   \right] \leq \sup_{t \in T}\mathbb{E} \left[ |\mathcal{X}_{t}| \mathbf{1}_{\{|\mathbb{E}[\mathcal{X}_{t} \mid \mathcal{G}]| > a \}}   \right] < \varepsilon,
\end{align}
\noindent which implies our claim.
\end{proof}
\end{appendices}

\paragraph{Acknowledgements.}
We thank Takis Konstantopoulos for reading an earlier version of this manuscript and providing comments that improved its presentation. GB also thanks Svante Janson for discussions held during the early stages of this work.

%\bibliography{Gab}
%\bibliographystyle{Myplain5}

\providecommand{\bysame}{\leavevmode\hbox to3em{\hrulefill}\thinspace}
\providecommand{\MR}{\relax\ifhmode\unskip\space\fi MR }
\providecommand{\MRhref}[2]{%
  \href{http://www.ams.org/mathscinet-getitem?mr=#1}{#2}
}
\providecommand{\href}[2]{#2}

\end{document}